\documentclass[12pt]{article}

\usepackage{amsthm,amssymb,amsfonts,amsmath,color,soul,mathtools}
\usepackage{fullpage}
\usepackage{graphicx}
\graphicspath{{figures/}}
\usepackage{enumerate}
\usepackage{hyperref}
\usepackage{comment}
\usepackage{stackengine}
\usepackage[dvipsnames]{xcolor}
\stackMath
\newtheorem{problem}{Problem}

\newtheorem{gproblem}{General Problem}

\newtheorem{theorem}{Theorem}[section]

\newtheorem{corollary}[theorem]{Corollary}

\newtheorem{definition}[theorem]{Definition}

\newtheorem{lemma}[theorem]{Lemma}

\newtheorem{proposition}[theorem]{Proposition}
\newtheorem{remark}[theorem]{Remark}

\newcommand{\spn}{{\mathrm span}}
\newcommand{\supp}{\mathrm supp}
\newcommand{\bR}{{\mathbb R}}
\newcommand{\bC}{{\mathbb C}}

\def\k{\dot{k}}

\newcommand*\todo[1]{{\textbf{*** #1 ***}}}

\input{tcilatex}

\providecommand{\keywords}[1]
{
  \small	
  \textbf{\textit{Keywords---}} #1
}

\begin{document}
\title{Optimal linear response for Markov Hilbert-Schmidt integral operators and
stochastic dynamical systems}
\author{Fadi Antown$^{1}$, Gary Froyland$^{1}$, and Stefano  Galatolo$^{2}$ \\ \\
 \small $^{1}$School of Mathematics and Statistics, University of New South Wales, Sydney NSW 2052, Australia \\
        \small $^{2}$Department of Mathematics, University of Pisa,
Via Buonarroti 1, 56127 Pisa, Italy \\}
\date{\today}
\maketitle
\begin{abstract}
We consider optimal control problems for discrete-time random dynamical systems, finding unique perturbations that provoke maximal responses of statistical properties of the system.
We treat systems whose transfer operator has an $L^2$ kernel, and we consider the problems of finding (i) the infinitesimal perturbation maximising the expectation of a given observable and (ii) the infinitesimal perturbation maximising the spectral gap, and hence the exponential mixing rate of the system.
Our perturbations are either (a)~perturbations of the kernel or (b) perturbations of a deterministic map subjected to additive noise.
We develop a general setting in which these optimisation problems have a unique solution and construct explicit formulae for the unique optimal perturbations.
We apply our results to a Pomeau-Manneville map and an interval exchange map, both subjected to additive noise, to explicitly compute the perturbations provoking maximal responses. 
\end{abstract}


[{\bf MSC2020}]{: 37H30, 47A55, 49J50, 49N05}

\section{Introduction}
The statistical properties of the long-term behaviour of deterministic or stochastic dynamical systems are strongly related to the properties of invariant or stationary measures and to the spectral properties of the associated transfer operator. 
When the dynamical system is perturbed it is useful to understand and predict the response of the statistical properties of the system through these objects. 
When such responses are differentiable, we say that the system exhibits a \emph{linear response} to the class of perturbations.
To first order, this response can be described by a suitable derivative expressing the infinitesimal rate of change in e.g.\ the natural invariant measure or in the spectrum.
Understanding the response of
statistical properties 
to perturbation has particular importance in applications, including
to  climate science 
(see e.g.\  \cite{GhLu}, \cite{HM} and the references therein).

In the present paper we go beyond quantifying responses and address 
natural problems concerning the \emph{optimal response}, namely 
which perturbations elicit a \emph{maximal} response. 
For example, given an observation function, which perturbation produces the greatest change in the expectation of this observation, and which perturbation produces the greatest change in the rate of convergence to equilibrium.
Continuing the climate science application, one may wish to know which small climate action (which perturbation) would produce the greatest reduction in the average
temperature (the expected observation value).
We note that by considering trajectories of a perturbed map and using ergodicity, one may view the problem of maximising the response in the expectation of an observation as an infinite-horizon optimal control problem, averaging an observation along trajectories. 


The linear response of dynamical systems is an area of intense research and we 
present a brief overview of the literature 
that is related to the present work.
Early results concerning the response of invariant measures to the perturbation of a deterministic system have been
obtained by Ruelle \cite{R2} in the uniformly hyperbolic case. 
More recently, these results
have been extended to several other situations in which one has some hyperbolicity
and sufficient regularity of the system and its perturbations. 
We
refer the reader to the survey \cite{BB} for an extended discussion of the literature
about linear response (and its failure) for deterministic systems. 

The mathematical literature on linear response of invariant measures of stochastic or random dynamical systems is more recent.
In the framework of continuous-time random processes and stochastic differential equations, linear response results were proved in \cite{HM,KLP19}.
Results related to the linear response of the stationary measure for diffusion in random media appear in
\cite{KOlla,GMP,Gantert,Faggio,MatPi}.
In the discrete-time case, examples of linear response for small random perturbations of uniformly hyperbolic deterministic
systems appeared in \cite{GL}.
In \cite{BRS} linear response results are given for random compositions of expanding or non-uniformly expanding maps.
In the paper \cite{ZH07} the smoothness of the invariant measure response under suitable perturbations is proved for a class of random diffeomorphisms, but no explicit formula is given for the derivatives; an application to the smoothness of the rotation number of Arnold circle maps with additive noise is presented.
Systems generated by the iteration of a deterministic map subjected to i.i.d.\ additive random perturbations are one class of stochastic systems studied in the present paper (see Section \ref{sec:sys-add-noise}).
The linear response of such systems  is considered systematically in \cite{GG} and linear response results are proved for perturbations to the deterministic map or to the additive noise. 
These results are used to by \cite{MSGDG} to extend some  results of \cite{ZH07}  outside the diffeomorphism case and applied to an idealized model of El Ni\~{n}o-Southern Oscillation, given by a noninvertible circle map with additive noise.
Higher derivative results for the response of systems with additive noise are presented in \cite{GS}.
Response results for random systems in the so-called {\em quenched} point of view appeared recently in \cite{SJ}, \cite{SRu} where the random composition of expanding maps is considered using Hilbert cones techniques and in  \cite{DS} where the random composition of hyperbolic maps is considered by a  transfer operator based approach.

We remark that the addition of random perturbations is not necessarily sufficient to guarantee a linear response.
An i.i.d.\ composition of the identity map and a rotation on the circle is considered in \cite{Gpre}, and it is shown that
using observables with square-integrable first derivative, one only has H\"older continuity of the response with respect to $C^0$ perturbations of the circle rotation.


One can similarly consider the linear response of the dominant eigenvalues of the transfer operator under perturbation.
In the literature there are several results describing the way eigenvalues and eigenvectors of suitable classes of operators change when those operators are perturbed in some way, for example classical results  concerning compact operators subjected to analytic perturbations \cite{K}, and    quasi-compact Markov operators subjected to $C^k$ perturbations \cite{HH}.
In specific classes of dynamics, differentiability of isolated spectral data is demonstated in \cite{GL} for transfer operators of Anosov maps where the map is subjected to $C^k$ perturbations and in \cite{KLP19} for transfer operators arising from SDEs subjected to $C^k$  perturbations of the drift. 


Optimal linear response questions have been considered in the dynamical setting of homogeneous (and inhomogeneous) finite-state Markov chains \cite{ADF}, where explicit formulae are provided for the unique maximising perturbations that (i) maximise the norm of the response, (ii) maximise the expectation of a given observable, and (iii) maximise the spectral gap.
The efficient Lagrange multiplier approach  developed in \cite{ADF} for questions (ii) and (iii) will be extended to the infinite-dimensional setting in the present paper.
In continuous time, \cite{FS17} maximised the spectral gap of a numerical discretisation of a periodically forced Fokker-Planck equation (perturbing the velocity field to maximally speed up or slow down the exponential mixing rate).
The same problem is considered by \cite{FKP}, but for general aperiodic forcing over a finite time, using the Lagrange multiplier approach of \cite{ADF}.
A non-spectral approach to increasing mixing rates by optimal kernel perturbations in discrete time is \cite{FGTW16}.

Related optimal control problems have been considered in \cite{GP} 
where the goal was to find a minimal perturbation realising a specific response to the invariant measure of a deterministic system (about the problem of finding an infinitesimal perturbation realising a given response see also \cite{Kl}).
These kinds of questions and other similar ones were also briefly considered in \cite{GG} for random dynamical systems consisting of deterministic maps perturbed by additive noise. Similar problems in the case of   probabilistic cellular automata were considered in \cite{Mac}.

The present work takes the point of view  of \cite{ADF}, extending the theory to the infinite-dimensional setting of stochastic integral operators, proving the existence of unique optimal perturbations, deriving explicit formulae for these optimal perturbations, and illustrating the formulae and their conclusions via two topical examples.
We consider the class of stochastic dynamical systems with transfer operators representable by an $L^2$-compact, integral operator, which includes deterministic systems perturbed by additive noise.
The transfer operator $L$ has the form 
\begin{equation}
    \label{introLeqn}
    Lf(x)=\int k(x,y)f(y)\ dy,
\end{equation} 
where $k$ is a stochastic kernel; in the case of deterministic systems $T$ with additive noise, $k(x,y)=\rho(x-T(y))$, with $\rho$ a probability density representing the distribution of the noise intensity (see Section \ref{sec:sys-add-noise}).
We consider perturbations of two types: firstly, perturbations to the kernel $k$, and secondly, perturbations to the map $T$.
An outline of the paper is as follows.
In Section \ref{SEC2} we consider general compact, integral-preserving operators $L:L^2 \to L^2$ (see \eqref{eq:int-pers-cond}) and state general linear response statements for the normalised fixed points and the leading eigenvalues of these operators (Theorem \ref{th:linearresponse copy(1)} and Proposition \ref{th:lin-resp-2nd-eval}).
In Section \ref{newsec3}, we derive response formulae for the normalised fixed points (Corollary \ref{LR}) and spectral values (Corollary \ref{cor:lin-resp-2nd-eval-HS}) of operators of the form \eqref{introLeqn}, under perturbation of the kernel $k$.
In Section \ref{sec4}
we consider the problem of finding the perturbation that provokes a \emph{maximal} response in 
the average of a given observable (General Problem 1) and the spectral gap (General Problem 2).
We show that if the feasible set of perturbations is convex, an optimal solution exists, and that this optimum is unique if the feasible set is strictly convex.
In Section \ref{sec:explicit-formula}, using Lagrange multipliers we derive 
an explicit formula for the unique optimal kernel perturbation that maximises the expectation of an observable (Theorem \ref{thm:explicit-formula}).
In section \ref{seceig} we prove an explicit formula for the perturbation that maximise the change in spectral gap (and therefore the rate of mixing) of the system (Theorem \ref{thm:explicit-formula-eval}).

In Section \ref{sec:sys-add-noise} we specialise our integral operators to annealed transfer operators corresponding to deterministic maps $T$ with additive noise. 
For these systems the kernel $k$ has the form $k(x,y)=\rho(x-T(y))$ for some nonsingular transformation $T$, and we consider perturbations of the map $T$ directly.
Response formulas for these perturbations are developed in Proposition \ref{prop:lin-resp-formula-add-noise} and Proposition \ref{prop:lin-resp-formula-add-noise-eval} for the invariant measure and the dominating eigenvalues, respectively.
In this framework we again prove existence and uniqueness of the map perturbation maximising the derivative of the expectation of an observation (Proposition \ref{solutionT}) and then derive an explicit formula for the extremiser (Theorem \ref{thm:explicit-formula2}).
Proposition \ref{solutionT2} and Theorem \ref{extremiser2} state results analogous to 
Proposition \ref{solutionT} and Theorem \ref{thm:explicit-formula2} for the optimization of the spectral gap and mixing rate.




In section \ref{sec:numerics} we apply and illustrate the theoretical findings of this work on the Pomeau-Manneville map and a weakly mixing interval exchange, each perturbed by additive noise. 
For each map we numerically estimate (i) the optimal stochastic perturbation (perturbing the kernel $k$) and (ii) the optimal deterministic perturbation (perturbing the map $T$) that maximises the derivatives of the expectation of an observable and the mixing rate.
One of the interesting lessons is that to maximally increase the mixing rate of the noisy Pomeau-Manneville map, one should perturb the kernel (stochastic perturbation) to move mass away from the indifferent fixed point or deform the map to transport mass away from the fixed point (deterministic perturbation);  see Figures \ref{fig:PM-mix} and \ref{fig:PM-map-mix}, respectively.
Further numerical outcomes are discussed and explained in Section \ref{sec:numerics}.

\bigskip

\bigskip

\section{Linear response for compact integral-preserving operators}\label{SEC2}

In this section we introduce general response results for integral-preserving compact operators. We consider both
the response of the invariant measure to the perturbations and the response of the dominant eigenvalues.

\subsection{Existence of linear response for the invariant measure}

\label{sec:gen-result-resp-inv-meas}

In the following, we consider integral-preserving compact operators acting on $L^{2}$, which are not necessarily positive. 
We will give a general linear response statement for their invariant measures.
In Section \ref{newsec3} we show how these results can be applied to Hilbert-Schmidt integral
operators, which will later be transfer operators of suitable random dynamical systems.

Let $L^{2}([0,1])$ be the space of square-integrable functions over the unit interval (considered with the Lebesgue measure $m$); for brevity, we will denote it as simply $L^{2}$. 
We remark that the  analysis in the rest of the paper can be extended to manifolds, but we keep the setting simple so as not to obscure the main ideas.
Let us consider the space of zero-average functions
\begin{equation*}
V:=\bigg\{f\in L^{2}~s.t.~~\int f\,dm=0\bigg\}.
\end{equation*}

\begin{definition}
We say that an operator $L:L^{2}\rightarrow L^{2}$ has  \emph{exponential contraction of the zero average space} $V$ if there are $C\geq 0$ and $%
\lambda <0$ such that $\forall g\in V$
\begin{equation}
\Vert L^{n}g\Vert _{2}\leq Ce^{\lambda n}\Vert g\Vert _{2}  \label{equil}
\end{equation}%
for all $n\geq 0$.
\end{definition}

For $\bar{\delta}>0$ and $\delta\in [0,\bar{\delta})$, we consider a family of integral-preserving, compact operators $L_{\delta }:L^{2}\rightarrow L^{2}$; we think of $L_{\delta }$ as perturbations of $L_{0}$.
We say that $f_\delta\in L^2$ is an \emph{invariant function} of $L_\delta$ if $L_\delta f_\delta=f_\delta$.
We will see that under natural assumptions, the operators $L_\delta$, $\delta\in[0,\bar{\delta})$, have a family of normalized invariant functions $f_{\delta } \in L^{2}$.
Furthermore, for suitable perturbations the invariant functions vary smoothly in $L^{2}$ and we get an explicit formula for the resulting derivative $\frac{df_\delta}{d \delta}$.
We remark that since the operators we consider are not necessarily positive, the invariant 
functions are not necessarily positive.

\begin{theorem}[Linear response for integral-preserving compact operators]
\label{th:linearresponse copy(1)} 
Let us consider a family of compact operators $L_{\delta }:L^{2}\rightarrow L^{2}$, with $\delta \in \left[ 0,\overline{\delta }\right)$, preserving the integral: for each $g\in L^2$
\begin{equation}  \label{eq:int-pers-cond}
\int L_\delta g~dm=\int g~dm.
\end{equation} Then,
\begin{enumerate}
    \item[(I)] 
The operators have invariant functions in $L^2$: for each $\delta$ there is $g_\delta \neq 0$ such that $L_\delta g_\delta=g_\delta $.
\item[(II)] 
Suppose $L_0$ also satisfies the following:

(A1) (mixing of the unperturbed operator) For every $g\in V$,
\begin{equation*}
\lim_{n\rightarrow \infty }\Vert L_{0}^{n}g\Vert _{2}=0.
\end{equation*}
Under this assumption, the unperturbed operator $L_0 $ has a unique {\em normalized} invariant function $f_0$ such that $\int {f}_0\ dm=1$. Furthermore, $L_0$ has exponential contraction of the zero average space $V$.
\item[(III)]  
Suppose the family of operators $L_\delta$ also satisfy the following:

(A2) ($L_\delta$ are small perturbations and existence of derivative operator at $f_0$) Suppose there is a $%
K\geq 0$ such that $\left\vert |L_{\delta}-L_{0}|\right\vert
_{L^{2}\rightarrow L^{2}}\leq K\delta $ for small $\delta $. 
Furthermore, suppose there exist
 $\hat{f} \in V$ such that
 
 \begin{equation*}
\underset{\delta \rightarrow 0}{\lim }  \frac{(L_{\delta}-L_{0})}{%
\delta }f_0 =\hat{f}.
\end{equation*}

Under these assumptions, the following hold:
\begin{enumerate}
\item There exists a $\delta_2>0$ such that for each $0\leq \delta <\delta_2 $, the operators $L_\delta$ have unique invariant functions ${f}_\delta$ such that $\int {f}_\delta dm=1.$
\item The resolvent operator $({Id}-L_0)^{-1}:V\rightarrow $ $V$ is continuous.

\item \begin{equation*}
\lim_{\delta \rightarrow 0}\left\Vert \frac{f_{\delta }-f_{0}}{\delta }-({Id}%
-L_0)^{-1}\hat{f}\right\Vert _{2}=0;
\end{equation*}%
thus, $({Id}-L_0)^{-1}\hat{f}$ represents the first order term in the perturbation of the invariant function for the family of systems $L_{\delta}$.
\end{enumerate}
\end{enumerate}
\end{theorem}



\begin{proof}
\, 
   \paragraph{Claim (I):}
We start by proving the existence of the invariant functions $g_\delta$ for the operators $L_\delta$. 
Since the operators are compact and integral preserving, $L_\delta $ has an eigenvalue $1$ for each $\delta$. Indeed, let us consider the adjoint operators $L^*_\delta:L^2\to L^2$ defined by the duality relation $\langle L_\delta f,g\rangle=\langle f,L^*_\delta g \rangle $ for all $f,g\in L^2.$
Because of the integral-preserving assumption, we have
$\langle f, L^*_\delta \mathbf{1}\rangle = \langle L_\delta f,\mathbf{1}\rangle = \int L_\delta f\ dm = \int f\ dm = \langle f,\mathbf{1}\rangle$.\footnote{We use the notation $\mathbf{1}$ for the constant function and $\mathbf{1}_A$ for the indicator function of the set $A$.}
This implies $L^*_\delta \mathbf{1}=\mathbf{1}$ and thus, $1$ is in the spectrum of $L^*_\delta$ and $L_\delta$. 
Since $L_\delta$ is compact, its spectrum equals the eigenvalues and we have nontrivial fixed points for the 
operators $L_\delta$.

\paragraph{Claim (III)(a) for $\delta=0$:}
Now we prove the uniqueness of the normalized invariant function of $L_0$. 
Above we proved that $L_0 $ has some invariant function $g_0\neq 0$.
The mixing assumption $(A1)$ implies that
$\int g_{0}\ dm\neq 0$; to see this, we note that if $\int g_{0}\ dm=0$, then $g_0\in V$, and, by $(A1)$, $g_0$ cannot be a nontrivial fixed point of $L_0$.
We claim that $f_0=\frac{g_0}{\int g_{0}\ dm}$ is the unique normalized invariant function for $L_0$.
To see this, suppose there was a second normalized invariant function $f'_0$; then, $f'_0-f_0$ would be an invariant function in $V$, which is a contradiction. 

\paragraph{Claim (II):}
To show that $L_0$ has exponential contraction on $V$, we first note that for
$f\in L^2$, we can write $f=f_0\int f\ dm+[f-f_0\int f\ dm]$. Since $[f-f_0\int f\ dm]\in V$, it follows from $(A1)$ that $L_0^n f\to_{L^2} f_0\int f\ dm$. 
Thus, the spectrum of $L_0$ is contained in the unit disk by the spectral radius theorem.
Now suppose $\lambda$ is in the spectrum of $L_0$ and $|\lambda|=1$. By the compactness assumption, there is  an eigenvector $f_{\lambda}$ for $\lambda$ and then we have $||L_0^n(f_{\lambda})||_2=||f_{\lambda}||_2$. However, $L_0^n(f_{\lambda})\to_{L^2} f_0\int f_\lambda\ dm$, which is not possible unless $\lambda=1$. Hence, the spectrum of $L_0|_V$ is strictly contained in the unit disk. Thus, by the spectral radius theorem, there is an $n>0$ such that 
$||L_0^n|_V||_{L^2\rightarrow L^2}\leq \frac{1}{2}$ and we have exponential contraction of $L_0$ on $V$. 

\paragraph{Claim (III)(a) for $\delta\in [0,\bar{\delta}]$:}
From the assumption $||L_0-L_\delta||_{L^2\rightarrow L^2}\leq K\delta $, 
we have for small enough $\delta$ that  
$||L_\delta^n|_V||_{L^2\rightarrow L^2}\leq \frac{2}{3}$ and therefore, $L_\delta$ is also mixing. 
We can apply the argument above to the operators $L_\delta$ and obtain, for each small enough $\delta$, a unique normalized invariant function $f_\delta$.

\paragraph{Claim (III)(b):}
Using the exponential contraction of 
$L_0$ on $V$, we now show that $(\text{Id}-L_{0})^{-1}:V\rightarrow V$ is continuous. 
Indeed, for $f\in V$, we get $(\text{Id}-L_0)^{-1}f=f+\sum_{n=1}^{\infty }L_{0}^{n}f$. 
Since $L_{0}$ is exponentially
contracting on $V$, and $\sum_{n=1}^{\infty }Ce^{\lambda n}:=M<\infty ,$
the sum $\sum_{n=1}^{\infty }L_{0}^{n}f$ converges in $V$ with respect to
the $L^{2}$ norm. The resolvent $(\text{Id}-L_{0})^{-1}:V\rightarrow V$ is then a continuous operator and $||(\text{Id}-L_{0})^{-1}||_{V\rightarrow V}\leq 1+M.$ We remark
that since $\hat{f}\in V,$ the resolvent can be computed at $\hat{f}$. 

\paragraph{Claim (III)(c):}
Now we are ready to prove the linear response formula. 
Furthermore, we have 
%
\begin{eqnarray*}
\Vert f_{\delta }-f_{0}\Vert _{2} &\leq &\Vert L_{\delta }^{n}f_{\delta
}-L_{0}^{n}f_{0}\Vert _{2} \\
&\leq &\Vert L_{\delta }^{n}f_{0}-L_{0}^{n}f_{0}\Vert _{2}+\Vert L_{\delta
}^{n}f_{\delta }-L_{\delta }^{n}f_{0}\Vert _{2} \\
&\leq &\Vert L_{\delta }^{n}-L_{0}^{n}\Vert _{2}\Vert f_{0}\Vert_{2}+\Vert
L_{\delta }^{n}|_{V}\Vert_{L^2\rightarrow L^2}\Vert f_{\delta }-f_{0}\Vert _{2}
\\
&\leq &
\Vert L_{\delta }^{n}-L_{0}^{n}\Vert _{2}\Vert f_{0}\Vert_{2}+
\frac23\Vert f_{\delta }-f_{0}\Vert _{2},
\end{eqnarray*}
from which we get $\Vert f_{\delta }-f_{0}\Vert _{2}\leq 3 \Vert L_{\delta }^{n}-L_{0}^{n}\Vert _{L^2\rightarrow L^2}\Vert f_{0}\Vert_{2}$. 
Since $||L_0-L_\delta||_{L^2\rightarrow L^2}\leq K\delta $, we have 
$\Vert f_{\delta }-f_{0}\Vert _{2} \to 0$ as $\delta \to 0.$
Since $f_{0}$
and $f_{\delta }$ are the invariant functions of $L_0$ and $L_\delta$, we have
\begin{equation*}
(\text{Id}-L_{0})\frac{f_{\delta }-f_{0}}{\delta }=\frac{1}{\delta }(L_{\delta
}-L_{0})f_{\delta }.
\end{equation*}%
By applying the resolvent to both sides we obtain
\begin{equation*}
\begin{aligned} \frac{f_{\delta }-f_{0}}{\delta }& =
(\text{Id}-L_0)^{-1}\frac{L_{\delta }-L_{0}}{\delta }f_{\delta } \\
&=(\text{Id}-L_0)^{-1}\frac{L_{\delta }-L_{0}}{\delta
}f_{0}+(\text{Id}-L_0)^{-1}\frac{L_{\delta }-L_{0}}{\delta }(f_{\delta }-f_{0}).
\end{aligned}  \label{eq:derivation-linear-resp}
\end{equation*}
Moreover, from assumption $(A2)$, we have for sufficiently small $\delta$ that
\begin{equation*}
\left\Vert (\text{Id}-L_{0})^{-1}\frac{L_{\delta }-L_{0}}{\delta }(f_{\delta
}-f_{0})\right\Vert _{2}\leq \Vert (\text{Id}-L_{0})^{-1}\Vert _{V\rightarrow
V}K\Vert f_{\delta }-f_{0}\Vert _{2}.
\end{equation*}%
Since we already proved that $\lim_{\delta \rightarrow 0}\Vert f_{\delta }-f_{0}\Vert _{2}=0$, 
we are left with
\begin{equation*}
\lim_{\delta \rightarrow 0}\frac{f_{\delta }-f_{0}}{\delta }=(\text{Id}%
-L_{0})^{-1}\hat{f}
\end{equation*}%
converging in the $L^{2}$ norm.
\end{proof}

We remark that the strategy of proof of Theorem \ref{th:linearresponse copy(1)} is similar to the one of Theorem 3 of \cite{GG} although the assumptions made are quite different, here we consider a compact integral preserving operator on $L^2$, while in \cite{GG} several norms are considered to allow low regularity perturbations and the operator is required to be positive.

It is worth to remark that the above proof gives a description of the spectral picture of $L_0$. By Theorem \ref{th:linearresponse copy(1)}, if $L_0$ satisfies $(A1)$ then the invariant
function is unique, up to normalization; this shows that $1$ is a simple eigenvalue.  
Furthermore, $L_0$ preserves the direct sum $L^2=$ span$\{f_0\} \oplus V$ and the spectrum of $L_0$ is strictly inside the unit disk when $L_0$ is restricted to $V$. 
Hence, the spectrum of $L_0$ is contained in the unit disk and there is a spectral gap.

\begin{remark}\label{assumptions} The mixing assumption in $(A1)$ is required only for the 
{unperturbed operator} $L_{0}$. 
This assumption is satisfied, for example, if $L_0$ is an integral operator and an iterate of this  operator has a strictly positive
kernel, see Corollary 5.7.1 of \cite{LM}.
Later in Remark \ref{6.4} we show this assumption is verified for a wide range of examples of stochastic dynamical systems.
\end{remark}

\subsection{Existence of linear response of the dominant eigenvalues}

In this section, we consider the existence of linear response for the second
largest eigenvalues (in magnitude) and provide a formula for the linear
response. 
An important object needed to quantify linear response statements is a ``derivative'' of the transfer operator with respect to the perturbation.

\begin{definition}\label{dop}
 We define $\dot{L}:L^2\to V$ as the unique linear operator satisfying  
$$\lim_{\delta\to 0}\left\|\frac{(L_\delta-L_0)}{\delta}-\dot{L}\right\|_{L^2\to V}=0.$$
\end{definition}

Let $\mathcal{B}(L^2)$ denote the space of bounded linear operators from the Banach space $L^2$
to itself and $r(L)$ denote the spectral radius of an operator $L$; we begin with the 
following definition.

\begin{definition}[\protect\cite{HH}, Definition III.7]
\label{def:HH-dom-sim-eval} Let $s\in\mathbb{N}, s\ge 1$. 
We say that $L\in%
\mathcal{B}(L^2([0,1],\mathbb{C}))$ has $s$ dominating simple eigenvalues if there exists
closed subspaces $E$ and $\tilde{E}$ such that

\begin{enumerate}
\item $L^2([0,1],\mathbb{C}) = E\oplus \tilde{E}$,

\item $L(E)\subset E$, $L(\tilde{E})\subset \tilde{E}$,

\item dim$(E)=s$ and $L|_{E}$ has $s$ geometrically simple eigenvalues $\lambda_i$, $%
i=1,\dots, s$,

\item $r(L|_{\tilde{E}})<\min \{|\lambda_i|:i=1,\dots,s\}$.
\end{enumerate}
\end{definition}

Adapting  Theorem III.8 and Corollary III.11 of \cite{HH} to our situation, we can now state a linear response result for these eigenvalues.

\begin{proposition}
\label{th:lin-resp-2nd-eval} 
Let $L_\delta:L^2([0,1],{\mathbb{C}}%
)\rightarrow L^2([0,1],{\mathbb{C}})$, where $\delta\in [0,\bar\delta)=:I_0$, be integral-preserving (see equation \eqref{eq:int-pers-cond}) compact operators. 
Assume that the map $\delta\mapsto L_\delta$ is in 
$C^1(I_0,\mathcal{B}(L^2([0,1],\bC)))$ and $L_0$ is mixing (see $(A1)$ in Theorem \ref{th:linearresponse copy(1)}). Then, $\lambda_{1,0}:= 1\in\sigma (L_0)$ and $r(L_0)=1$. 
Let $\mathcal{I}\subset\sigma(L_0)\setminus\{1\}$
be the eigenvalue(s) of maximal modulus strictly inside the unit disk; assume they are geometrically simple 
and let $s:=|\mathcal{I}|+1$. 
Then there exists an interval $I_1:=[0,\delta_1) $, $I_1\subset I_0$ such that for 
$\delta\in I_1$, $L_\delta$ has $s$ dominating simple eigenvalues. Thus, there exists functions $e_{i, (\cdot)},\ \hat{e}_{i,(\cdot)}\in C^1(I_1,L^2([0,1],\bC))$ and
$\lambda_{i,(\cdot)}\in C^1(I_1,\bC)$ such that for $\delta\in I_1$ and $i,j = 2,\dots, s$
\begin{itemize}
\item[(i)] $L_\delta e_{i,\delta} = \lambda_{i,\delta} e_{i,\delta}$, $L^*_\delta\hat{e}_{i,\delta} = \lambda_{i,\delta}\hat{e}_{i,\delta}$,
\item[(ii)] $\langle e_{i,\delta},\hat{e}_{j,\delta}\rangle_{L^2([0,1],\bC)} = \delta_{i,j}$, where
$\delta_{i,j}$ is the Kronecker delta.
\end{itemize}
Furthermore, let  $\dot{\lambda}_i\in \bC$ satisfy
\begin{equation*}
\lim_{\delta\rightarrow 0}\bigg|\frac{\lambda_{i,\delta}-\lambda_{i,0}}{\delta}%
-\dot{\lambda}_{i}\bigg| = 0,
\end{equation*}
then 
\begin{equation}  \label{eq:lin-resp-2nd-eval-pre}
\begin{aligned} \dot{\lambda}_i = \langle \hat{e}_{i,0},\dot{L} e_{i,0}
\rangle_{L^2([0,1],\bC)}, \end{aligned}
\end{equation}
where $\dot{L}$ is as in Definition \ref{dop}.

\end{proposition}

\begin{proof}
From Theorem \ref{th:linearresponse copy(1)} and the discussion following it, $1\in\sigma(L_0)$ and 
$r(L_0)=1$.


We now use Theorem III.8 in \cite{HH} to obtain the
existence of linear response and Corollary III.11 \cite{HH} to obtain the formula.
We begin by verifying the two hypotheses of Theorem III.8 \cite{HH}.
We remark that our map $\delta\mapsto L_\delta$ belonging to $C^1([0,\bar\delta),\mathcal{B}(L^2([0,1],\bC)))$  can be extended to a map  $C^1((-\bar\delta,\bar\delta),\mathcal{B}(L^2([0,1],\bC)))$.

Doing so,  
hypothesis $(H1)$ of Theorem III.8 \cite{HH} is satisfied. Since $r(L_0)=1$, we just need to show that $L_0$ has $s$ dominating eigenvalues. 
 Since
$L_0 $ is a compact operator, the eigenvalues $\lambda_{i,0}\in\mathcal{I}$ are isolated. 
Let $\Pi_i$ be the eigenprojection onto the
eigenspace of $\lambda_{i,0}$ and $E_i:=\Pi_i (L^2([0,1],{\mathbb{C}}))$. 
Define the eigenspaces $E:=\bigoplus_{i=1}^s E_i$ and $\widetilde{E}: = (\text{Id}%
-\sum_{i=1}^s\Pi_i)(L^2([0,1],{\mathbb{C}}))$. We thus have:

\begin{itemize}
\item[(1)] $L^2([0,1],{\mathbb{C}}) = E\oplus \widetilde{E}$.

\item[(2)] $L_0\left(E\right)\subset E$ and $L_0(\widetilde{E}%
)\subset \widetilde{E}$.

\item[(3)] dim$\left(E\right) =s$ and $L_0|_{E}$ has $s$
simple eigenvalues $\lambda_{1,0}\cup\mathcal{I}$. This point 
follows from the assumption that the eigenvalues in $\mathcal{I}$ are geometrically
simple and the fact that $\lambda_{1,0}$ is simple 
(see Theorem \ref{th:linearresponse copy(1)}).

\item[(4)] $r(L_0|_{\widetilde{E}}) <|\lambda_{i,0}|$ where $\lambda_{i,0}\in \mathcal{I}$.
\end{itemize}
Thus, $L_0$ satisfies hypothesis (H2) of
Theorem III.8 since it has $s$ dominating simple eigenvalues and 
$r(L_0)=1$. Hence, from Theorem III.8 \cite{HH}, the map $%
\delta\mapsto\lambda_{i,\delta}$ is differentiable at $\delta=0$.

We can now apply the argument in Corollary III.11 \cite{HH} for $\lambda_{i,0}$ to obtain %
\eqref{eq:lin-resp-2nd-eval} (the result and proof of Corollary III.11 \cite{HH} is
for the top eigenvalue, however the argument still holds for any eigenvalue
$\lambda_{i,0}$, $\in\mathcal{I}$ by changing the index value in the proof of the corollary).
\end{proof}

\section{Application to Hilbert-Schmidt integral operators}\label{newsec3}

In this section we apply the results of the previous section to Hilbert-Schmidt integral operators
and suitable perturbations. 
The operators we consider are compact operators on $L^{2}([0,1],\bR)$ (or $L^{2}([0,1],\bC)$); for brevity we will denote\footnote{We will also denote $L^p:=L^p([0,1],\bR)$; this notation will not be used for $L^{2}([0,1],\bC)$.} $L^{2}:=L^{2}([0,1],\bR)$. 
To avoid confusion we point out that in the following we will also consider the space $L^{2}([0,1]^{2})$ of square integrable real functions on the unit square; 
this space contains the kernels of the operators we consider. 

Let $k\in L^{2}([0,1]^{2})$
and consider the operator $L:L^2 \to L^2$ defined in the following way:
for $f\in L^{2}$
\begin{equation}
Lf(x)=\int k(x,y)f(y)dy;  \label{kernelL}
\end{equation}%
such an operator is called a Hilbert-Schmidt integral operator. Such
operators may represent the annealed transfer operators of systems perturbed by
additive noise (see Section \ref{sec:sys-add-noise}). 

We now list some well-known and basic facts about Hilbert-Schmidt integral operators with kernels 
in $L^{2}([0,1]^{2})$:
\begin{itemize}
\item The operator $L:L^{2}\rightarrow L^{2}$ is bounded and
\begin{equation}
||Lf||_{2}\leq ||k||_{L^{2}([0,1]^{2})}||f||_{2}  \label{KF}
\end{equation}%
(see Proposition 4.7 in II.\S 4 \cite{C}).
\item If $k\in L^{\infty}([0,1]^{2})$, then
\begin{equation}
||Lf||_{\infty }\leq ||k||_{L^{\infty }([0,1]^{2})}||f||_{1}  \label{KF2}
\end{equation}%
and the operator $L:L^1\rightarrow L^{\infty }$ is bounded. Furthermore, 
$\|L\|_{L^p\rightarrow L^\infty}\le \|k\|_{L^\infty([0,1]^2)}$ for $1\le p\le \infty$.  
\item If for almost every $y\in[0,1]$ we have $$\int k(x,y) dx=1,$$ then the Hilbert-Schmidt 
integral operator associated to the kernel $k$ is integral preserving (satisfies 
\eqref{eq:int-pers-cond}).
\item The operator $L:L^2\rightarrow L^2$ is compact (see \cite{KF}).
\end{itemize}
Combining the last two points, we have from Theorem \ref{th:linearresponse copy(1)} that
such an operator has an invariant function in $L^{2}$.
Furthermore, for $k\in L^\infty([0,1]^2)$ we have an analogous result.
\begin{lemma}
Let $L:L^2\rightarrow L^2$ be an integral operator, with integral-preserving kernel $k\in L^\infty([0,1]^2)$, that is mixing (satisfies $(A1)$ of
Theorem \ref{th:linearresponse copy(1)}). Then, there exists a unique fixed point $f\in L^\infty$ of $L$ satisfying $\int f\ dm =1$.
Furthermore, if the kernel is nonnegative, then $f$ is nonnegative.
\end{lemma}
\begin{proof}
Since $k$ is an integral-preserving kernel, $L_0$ satisfies \eqref{eq:int-pers-cond}. Thus, we
can apply Theorem \ref{th:linearresponse copy(1)} to conclude that
there exists a unique $f\in L^2$,
$\int f\ dm=1$, such that $Lf=f$.
Noting that $k\in L^\infty([0,1]^2)$, we have from inequality \eqref{KF2} that $f\in L^\infty$.

We now assume $k$ is nonnegative. Let $k^j$ be the kernel of the operator $L^{j}$. Since $k$ is an integral-preserving kernel, we have 
\begin{equation*}
\begin{aligned}
    |k^2(x,y)| &= \bigg|\int k(x,z)k(z,y)dz\bigg|\le \int |k(x,z)k(z,y)|dz\\
    &\le \|k\|_{L^\infty([0,1]^2)}\int k(z,y)dz = \|k\|_{L^\infty([0,1]^2)};
\end{aligned}
\end{equation*}
 it easily follows that
$\|k^j\|_{L^\infty([0,1]^2)}\le \|k\|_{L^\infty([0,1]^2)}$.
Thus, for any probability density $g\in L^1$, we have
$\|L^jg\|_\infty\le\|k\|_{L^\infty([0,1]^2)}$; thus, by Corollary 5.2.2 in
\cite{LM}, there exists a probability density $\hat{f}\in L^1$ 
such that $L\hat{f} = \hat{f}$.
Since $f$ is the unique invariant function with integral $1$, we have $\hat{f}=f$; 
thus, $f$ is a probability density.
\end{proof}

\subsection{Characterising valid perturbations and the derivative of the transfer operator}
\label{subsubsec:lin-resp-form-ker-setting}

In this subsection we consider perturbations of integral-preserving
Hilbert-Schmidt integral operators such that assumption (A2) of Theorem \ref%
{th:linearresponse copy(1)} can be verified and the derivative operator $%
\dot{L}$ computed. We begin, however, by first characterizing
the set of perturbations for which the integral
preserving property of the operators is preserved.

Consider the set $V_{\ker }$ of kernels having zero average in the $x$
direction, defined as
\begin{equation*}
V_{\ker }:=\bigg\{k\in L^{2}([0,1]^{2}): \int k(x,y)dx=0~for~a.e.~y\bigg\}.
\end{equation*}

\begin{lemma}
\label{lem:charac-mean-0} Consider a kernel operator $A:L^{2}([0,1])
\rightarrow L^{2}([0,1])$ defined by $Af(x)=\int k(x,y)f(y)dy$. Then, the
following are equivalent

\begin{enumerate}
\item $A(L^{2}([0,1]))\subseteq V$,
\item $k\in V_{\ker}$.
\end{enumerate}
\end{lemma}

\begin{proof}
Clearly, the second condition implies the first. For the other direction we prove the contrapositive.
If $\int
k(x,y)dx\neq 0$ on a set of positive measure, then for a small $\epsilon >0$
there is a set $S$ of positive measure $m(S) >0$ such that $\int k(x,y)dx\geq \epsilon $ or $\int
k(x,y)dx\leq -\epsilon $ for each $y\in S$.
Suppose $\int k(x,y)dx\geq \epsilon $ in this set, consider $f:=\mathbf{1}_{S}$ and $%
g:=Af.$ Then, $g(x)=\int k(x,y)\mathbf{1}_{S}(y)dy$ and we have $\int g(x)dx= \int_S \int k(x,y) dx dy \geq \epsilon \ m(S) $ and $ g\notin V$. The other case $\int k(x,y)dx\leq -\epsilon $  is analogous.
\end{proof}

We now prove that $V_{\ker }$ is closed.
\begin{lemma}
\label{lemma:closedsub} The set $V_{\ker }$ is a closed vector subspace of $%
L^{2}([0,1]^{2}).$
\end{lemma}

\begin{proof}
The fact that $V_{\ker}$ is a vector space is trivial. For fixed $f\in
L^{2}([0,1])$, the set of $k\in L^2([0,1]^2)$ such that $\int k(x,y)f(y)dx\in V$ is closed.
To see this, define the function $K_{f}:L^{2}([0,1]^{2})\rightarrow L^{2}([0,1])$ as
\begin{equation}
K_{f}(k)=\int k(x,y)f(y)dy.  \label{kf}
\end{equation}%
By \eqref{KF}, $K_{f}$ is continuous. Since $V$ is closed in $L^{2}([0,1])$, this implies that
$K_{f}^{-1}(V)$ is closed in $L^{2}([0,1]^{2}).$ Finally, $V_{\ker}$ is closed in $L^2([0,1]^2)$
because $V_{\ker}=\cap_{f\in L^{2}([0,1])}K_{f}^{-1}(V)$.
\end{proof}

We now introduce the type of perturbations which we will investigate throughout
the paper. Let $L_{\delta}:L^{2}\rightarrow L^{2}$ be a family of integral operators,
with kernels $k_{\delta }\in L^{2}([0,1]^{2})$, given by
\begin{equation*}
L_{\delta }f(x)=\int k_{\delta }(x,y)f(y)dy.
\end{equation*}

\begin{lemma}
\label{ldot2}Let $k_{\delta }\in L^{2}([0,1]^{2})$ for
each $\delta \in [0,\bar{\delta}).$ Suppose that
\begin{equation}
k_{\delta }=k_{0}+\delta \cdot \dot{k}+r_\delta  \label{perturb1}
\end{equation}%
where $\dot{k},\ r_\delta \in L^{2}([0,1]^{2})$ and $%
||r_\delta||_{L^{2}([0,1]^{2})} = o(\delta).$
The bounded linear operator $\dot{L}:L^2\to V$ defined by
\begin{equation}
    \label{dotL}
    \dot{L}f(x):=\int \dot{k}(x,y)f(y)dy
\end{equation} satisfies
\begin{equation*}
\lim_{\delta \rightarrow 0}\bigg\|\frac{L_{\delta }-L_{0}}{\delta }-\dot{L}\bigg\|_{L^2\to V}=0.
\end{equation*}%
If additionally the derivative of the map $\delta\mapsto k_\delta$  with respect to $\delta$ varies continuously in a neighborhood of $\delta=0$, then $\delta\mapsto L_\delta$  has a continuous derivative in a neighborhood of  $\delta=0$.
\end{lemma}

\begin{proof}
By integral preservation of $L_\delta$ and the fact that $\dot{k}\in L^2([0,1]^2)$, one sees that $\dot{L}:L^2\to V$ and is bounded. 
By (\ref{perturb1}),
\begin{eqnarray*}
\left\|\frac{L_{\delta }-L_{0}}{\delta }-\dot{L}\right\|_{L^2\to V} &=&\sup_{\|f\|_{L^2}=1}\left\|\int \frac{k_{\delta }(x,y)-k_{0}(x,y)%
}{\delta }f(y)\ dy - \int \dot{k}(x,y)f(y)\ dy\right\|_{L^2} \\
&=&\sup_{\|f\|_{L^2}=1}\left\|\int r_\delta(x,y)f(y)\ dy\right\|_{L^2}\\
&\le&\|r_\delta\|_{L^2([0,1]^2)}=o(\delta).
\end{eqnarray*}%
Proceeding similarly, one shows that if the map $\delta\mapsto k_\delta$ has a continuous derivative with respect to $\delta$ in a neighborhood of $\delta=0$, then $\delta\mapsto L_\delta$  has a continuous derivative. 
Indeed we are supposing that for each $\delta \in [0,\overline{ \delta})$ there is $\dot{k}_\delta $ such that for small enough $h$  $$k_{\delta+h}=k_\delta+h\cdot \dot{k}_\delta +r_{\delta,h} $$
where $\dot{k}_\delta,\ r_{\delta,h} \in L^{2}([0,1]^{2})$, $
||r_{\delta,h}  ||_{L^{2}([0,1]^{2})} = o(h)$ and furthermore $\delta\mapsto \dot{k}_\delta$ is continuous. We have then by \eqref{KF} that the  associated    operators $\dot{L}_\delta$ defined as
\begin{equation}
    \dot{L}_\delta f(x):=\int \dot{k}_\delta (x,y)f(y)dy
\end{equation}
also varies in a continuous way as $\delta$ increases.

\end{proof}


\subsection{A formula for the linear response of the invariant measure and its continuity}

Now we apply Theorem \ref{th:linearresponse copy(1)} to Hilbert-Schmidt integral
operators to obtain a linear response formula for $L^2$ perturbations.

\begin{corollary}[Linear response formula for kernel operators]
\label{LR}Suppose $L_{\delta}:L^2\rightarrow L^2$ are integral-preserving (satisfying \eqref{eq:int-pers-cond}) integral operators with stochastic kernels $k_{\delta}\in L^2([0,1]^{2})$ as in \eqref{perturb1}. Suppose $L_0$ satisfies assumption $(A1)$ of Theorem \ref{th:linearresponse copy(1)}. Then $\dot{k}\in V_{\ker }$, the system has linear
response for this perturbation and an explicit formula for it is given by%
\begin{equation}
\lim_{\delta \rightarrow 0}\frac{f_{\delta }-f_{0}}{\delta }=(\text{Id}%
-L_{0})^{-1}\int \dot{k}(x,y)f_{0}(y)dy  \label{kerrsep}
\end{equation}%
with convergence in $L^{2}.$
\end{corollary}

\begin{proof}
Since $L_\delta$, $\delta\in[0,\bar\delta)$, is integral preserving, we have
$(L_\delta-L_0)(L^2)\subset V$ and therefore, $k_\delta-k_0\in V_{\ker}$ by Lemma 
\ref{lem:charac-mean-0}, i.e.\ $\delta\k+r_\delta\in V_{\ker}$. Then, for a.e. $y\in[0,1]$ 
and $\delta\not= 0$, we have
\begin{equation*}
\begin{aligned}
    \bigg|\int\k(x,y)dx\bigg|\le \frac{1}{\delta}\int|r_\delta(x,y)|dx \le \frac{1}{\delta}\|r_\delta\|_{L^2([0,1]^2)}.
\end{aligned}    
\end{equation*}
As $\delta\rightarrow0$, the right hand side approaches $0$ and, since $\int\k(x,y)dx$ is 
independent of $\delta$, we have $\int\k(x,y)dx=0$ for a.e. $y\in[0,1]$, i.e.\ $\k\in V_{\ker}$.

Furthermore by $(\ref{perturb1})$  there is a $K\geq 0$ such that
\begin{equation}
\left\vert |L_{0}-L_{\delta }|\right\vert _{L^{2}\rightarrow L^{2}}\leq
K\delta .  \label{1of3}
\end{equation}%
Hence the family of operators satisfy the first part of assumption $(A2)$. The second part of this assumption is established by  the results of Lemma \ref{ldot2}.

Since the operators $L_\delta$ are compact, integral preserving, and satisfy  assumptions $(A1)$
and $(A2)$ we can conclude by applying
Theorem \ref{th:linearresponse copy(1)} to this family of operators, obtaining
\begin{equation*}
\lim_{\delta \rightarrow 0}\left\Vert \frac{f_{\delta }-f_{0}}{\delta }-(%
\text{Id}-L_{0})^{-1}\int \dot{k}(x,y)f_{0}(y)dy\right\Vert _{2}=0.
\end{equation*}

\end{proof}

Now we show that the linear response of the invariant measure is continuous with respect to the kernel perturbation. This will be used in Section \ref{sec4} for the proof of the existence of solutions of our main optimization problems.

Consider the transfer operator $L_{0}$, having a kernel $k_{0}\in L^{2}([0,1]^2)$, and a set of 
infinitesimal perturbations $P\subset V_{\ker}$ of $k_{0}$. We will endow $P$ with the topology induced by its inclusion in $L^2([0,1]^2)$. Suppose $L_{\delta}$ is a perturbation
of $L_0$ satisfying the assumptions of Lemma \ref{ldot2}.
By Corollary \ref{LR}, the linear response will depend on the first-order term of the perturbation, $\dot k \in P$, allowing us to define the function $R:P\rightarrow V$ by
\begin{equation}
R(\dot{k}):=(\text{Id}%
-L_{0})^{-1}\int \dot{k}(x,y)f_{0}(y)dy.
\label{R}
\end{equation}%
By \eqref{KF} and the continuity of the resolvent operator it follows directly that the response function $R:(P,\|\cdot\|_{L^2([0,1]^2)})\to (V,\|\cdot\|_{L^2})$ is continuous.


\subsection{A formula for the linear response of the dominant eigenvalues and its continuity}

We apply Proposition \ref{th:lin-resp-2nd-eval} to Hilbert-Schmidt integral
operators and obtain a linear response formula for the dominant eigenvalues in the case of $L^2$ perturbations.

\begin{corollary}\label{cor:lin-resp-2nd-eval-HS}
Suppose $L_{\delta}:L^2([0,1],\bC)\rightarrow L^2([0,1],\bC)$ are integral-preserving (satisfying \eqref{eq:int-pers-cond}) integral operators with kernels $k_{\delta}\in L^2([0,1]^{2})$ satisfying $\delta\mapsto k_\delta\in C^1([0,\bar{\delta}),L^2([0,1]^{2}))$. Suppose $L_0$ 
satisfies $(A1)$ of Theorem 
\ref{th:linearresponse copy(1)}.
Let $\lambda_0\in \bC$ be an eigenvalue of $L_0$ with the largest magnitude strictly inside the unit circle and assume that $\lambda_0$ is geometrically simple. Then, there exists
$\dot{\lambda}\in\mathbb{C}$ such
that
\begin{equation*}
\lim_{\delta\rightarrow 0}\bigg|\frac{\lambda_{\delta}-\lambda_{0}}{\delta}%
-\dot{\lambda}\bigg| = 0.
\end{equation*}
Furthermore, 
\begin{equation}  \label{eq:lin-resp-2nd-eval}
\begin{aligned} \dot{\lambda} &= \int_0^1\int_0^1\dot{k}(x,y)\left(\Re(\hat{e})(x)\Re(e)(y) + \Im(\hat{e})(x)\Im(e)(y)%
\right) dydx\\ &\qquad +
i\int_0^1\int_0^1\dot{k}(x,y)\left(\Im(\hat{e})(x)\Re(e)(y) - \Re(\hat{e})(x)\Im(e)(y)
\right) dydx,
\end{aligned}
\end{equation}
where $e\in L^2([0,1],\mathbb{C})$ is the eigenvector of $L_0$ associated to
the eigenvalue $\lambda_0$, $\hat{e}\in L^2([0,1],\mathbb{C})$ is the
eigenvector of $L_0^*$ associated to the eigenvalue $\lambda_{0}$ and $\dot{L}$
is the operator in Lemma \ref{ldot2}.
\end{corollary}
\begin{proof}
Since $k_\delta\in L^2([0,1]^2)$, the operator $L_\delta:L^2([0,1],\bC)\rightarrow L^2([0,1],\bC)$
is compact; by assumption, it also satisfies \eqref{eq:int-pers-cond}. 
From Lemma \ref{ldot2}, the map $\delta\mapsto L_\delta$ is $C^1$. 
Hence,
by Proposition \ref{th:lin-resp-2nd-eval}, we have $\dot{\lambda} = \langle \hat{e},\dot{L} e
\rangle_{L^2([0,1],\mathbb{C})}$. Finally, we compute 
\begin{equation*}
\begin{aligned} \dot{\lambda}=\langle \hat{e},\dot{L} e
\rangle_{L^2([0,1],\mathbb{C})} &= \int_0^1\hat{e}(x) \overline{\dot{L}e}(x)
dx \\ &= \int_0^1\hat{e}(x)\overline{\int_0^1 \k(x,y)e(y)dy}dx\\ &=
\int_0^1\int_0^1\dot{k}(x,y)\hat{e}(x)\bar{e}(y) dydx\\ &=
\int_0^1\int_0^1\dot{k}(x,y)\left(\Re(\hat{e})(x)\Re(e)(y) + \Im(\hat{e})(x)\Im(e)(y)%
\right) dydx\\ &\qquad +
i\int_0^1\int_0^1\dot{k}(x,y)\left(\Im(\hat{e})(x)\Re(e)(y) - \Re(\hat{e})(x)\Im(e)(y)
\right) dydx. \end{aligned}
\end{equation*}
\end{proof}
From the expression in the final line of the proof above, it is clear that if we consider $\dot{\lambda}$ as a function of $\dot{k}$, the map $\dot{\lambda}:(V_{\ker},\|\cdot\|)_{L^2([0,1]^2)}\to\mathbb{C}$ is continuous.

\section{Optimal response:  optimising the expectation of observables and mixing rate} \label{sec4}

Having described the responses of our dynamical systems to perturbations, 
it is natural to consider the optimisation problem of finding perturbations that provoke  \emph{maximal} responses.
We consider the problems of finding the 
infinitesimal perturbation that maximises the expectation of a given observable and the
 infinitesimal perturbation that maximally enhances mixing.
In doing so, we extend the approach in \cite{ADF} from the setting of finite-state Markov chains to the integral operators considered in the present paper.

We show that at an abstract level these problems reduce to the optimization of a linear continuous functional $\mathcal{J}$ on a convex set $P$ of feasible perturbations; this problem has a solution and the solution is unique if the set $P$ of allowed infinitesimal perturbations is strictly convex. 
The convexity assumption on $P$ is natural because if two different perturbations of
the system are possible, then their convex combination (applying the two perturbations with different intensities) will also be possible.
After introducing the abstract setting, we construct the objective functions for our two optimal response problems and state general existence and uniqueness results for the optima.
Later, in Section \ref{sec:explicit} we focus on the construction of the set of feasible perturbations and provide explicit formulae for the maximising perturbations.

\subsection{General optimisation setting, existence and uniqueness}\label{subsec:abst-opt-setng}

We recall some general results (adapted for our purposes) on optimizing a linear\
continuous function on convex sets;  see also Lemma 6.2 \cite{FKP}.
The abstract problem is to find $\dot{k}$ such that
\begin{equation}
\mathcal{J}(\dot{k})=\max_{\dot{h}\in P}\mathcal{J}(\dot{h}),  \label{gen-func-opt-prob}
\end{equation}%
where $\mathcal{J}:\mathcal{H}\rightarrow {\mathbb{R}}$ is a continuous linear
function, $\mathcal{H}$ is a separable Hilbert space and $P\subset \mathcal{H%
}$.

\begin{proposition}[Existence of the optimal solution]
\label{prop:exist} Let $P$ be bounded, convex, and closed in $\mathcal{H}$.
Then, problem considered at \eqref{gen-func-opt-prob} has at least one solution.
\end{proposition}

\begin{proof}
Since $P$ is bounded and $\mathcal{J}$ is continuous, we have that $\sup_{k\in P}\mathcal{J}(k)<\infty $.
Consider a maximizing sequence $k_{n}$ such that $%
\lim_{n\rightarrow \infty }\mathcal{J}(k_{n})=\sup_{k\in P}\mathcal{J}(k)$. Then, $k_{n}$ has a
subsequence $k_{n_{j}}$ converging in the weak topology. Since $P$ is
strongly closed and convex in $\mathcal{H}$, we have that it is weakly closed. This
implies that $\overline{k}:=\lim_{j\rightarrow \infty }k_{n_{j}}\in P.$
Also, since $\mathcal{J}(k)$ is continuous and linear, it is continuous in the weak
topology. Then we have that $\mathcal{J}(\overline{k})=\lim_{j\rightarrow \infty
}\mathcal{J}(k_{n_{j}})=\sup_{k\in P}\mathcal{J}(k)$ and we realise a maximum.
\end{proof}

Uniqueness of the optimal solution will be provided by strict convexity of the feasible set.
\begin{definition}
\label{stconv}We say that a convex closed set $A\subseteq \mathcal{H}$ is
\emph{strictly convex} if for each pair $x,y\in A$ and for all $0<\gamma<1$, the points $\gamma x+(1-\gamma)y\in \mathrm{int}(A)$, where the relative interior\footnote{The relative interior of a closed convex set $C$ is the interior of $C$ relative to the closed affine hull of $C$, see e.g.\ \cite{borwein}.} is meant.

\end{definition}

\begin{proposition}[Uniqueness of the optimal solution]
\label{prop:uniqe} 
Suppose $P$ is closed, bounded, and strictly convex subset of $\cal{H}$, and that $P$ contains the zero vector in its relative interior.
If $\mathcal{J}$ is not uniformly vanishing on $P$
then the optimal
solution to \eqref{gen-func-opt-prob} is unique.
\end{proposition}


\begin{proof}
Suppose that there are two distinct maxima $\dot{k}_1,\dot{k}_2\in P$ with $\mathcal{J}(\dot{k}_1)=\mathcal{J}(\dot{k}_2)=\alpha$.
Let $0<\gamma<1$ and set $z=\gamma \dot{k}_1+(1-\gamma)\dot{k}_2$.
By strict convexity of $P$, $z\in \mathrm{int}(P)$, and by linearity of $\mathcal{J}$, $\mathcal{J}(z)=\alpha$.
Let $B_r(z)$ denote a (relative in $P$) open ball of radius $r$ centred at $z$, with $r>0$ chosen small enough so that $B_r(z)\subset \mathrm{int}(P)$.
Because the zero vector lies in the relative interior of  $P$, and $\mathcal{J}$ does not uniformly vanish on $P$, 
there exists a vector $v\in B_r(z)$ such that $\mathcal{J}(v)>0$.
Now $z+\frac{rv}{2\|v\|}\in \mathrm{int}(P)$ and $\mathcal{J}(z+\frac{rv}{2\|v\|})>\alpha$, 
contradicting maximality of $\dot{k}_1$.
\end{proof}

In the following subsections we apply the general results of this section to our specific optimisation problems.

\subsection{Optimising the response of the expectation of an observable}
\label{subsubsec1}
Let $c\in L^2$ be a given observable. 
We consider the problem of finding an infinitesimal perturbation that maximises the expectation of $c$.
The perturbations we consider are perturbations to the kernels of Hilbert-Schmidt integral operators, of the form \eqref{perturb1}.
If we denote the average of $c$ with respect to the perturbed invariant
density $f_{\delta }$ by
\begin{equation*}
\mathbb{E}_{c,\delta }:=\int c~f_{\delta }~dm,
\end{equation*}
we have
\begin{equation*}
\frac{d\mathbb{E}_{c,\delta }}{d\delta }\bigg|_{\delta =0}=\lim_{\delta \rightarrow 0}%
\frac{\mathbb{E}_{c,\delta }-\mathbb{E}_{c,0}}{\delta }=\lim_{\delta \rightarrow 0}\int c~%
\frac{f_{\delta }-f_{0}}{\delta }~dm=\int c~R(\dot{k})~dm,
\end{equation*}%
where the last equality follows from Corollary \ref{LR}. 

The function $\mathcal{J}(\dot{k})=\langle c,R(\dot{k})\rangle$ is clearly  continuous as a map from $(V_{\ker},\|\cdot\|)_{L^2([0,1]^2)}$ to $\mathbb{R}$.  
Suppose that $P$ is a closed, bounded, convex subset of $V_{\ker}$ 
containing the zero perturbation, and that $\mathcal{J}$ is not uniformly vanishing on $P$.
We wish to solve the following problem:
\begin{gproblem}
\label{prob1}
Find $\dot{k}\in P$ such that
\begin{equation}
\big\langle c,R(\dot{k})\big\rangle_{L^{2}([0,1],\bR)}=\max_{\dot{h}\in P}%
\big\langle c,R(\dot{h})\big\rangle_{L^{2}([0,1],\bR)}.  \label{P1}
\end{equation}
\end{gproblem}
We may immediately apply Proposition \ref{prop:exist} to obtain that \textit{there exists a solution to (\ref{P1}).}
If, in addition, $P$ is strictly convex, then by Proposition \ref{prop:uniqe} the solution to (\ref{P1}) \emph{is unique}.

To end this subsection we note that without loss of generality, we may assume that $c\in
$ span$\{f_0\}^\perp$.  This is because for $c\in L^{2}$, we have
\begin{equation*}
\langle c,R(\dot{k})\rangle_{L^{2}([0,1],\bR)}
=\langle c- \langle c,f_0\rangle_{L^{2}([0,1],\bR)}\mathbf{1},R(\dot{k}%
)\rangle_{L^{2}([0,1],\bR)},
\end{equation*}
since $R(\dot{k})\in V$. 
From $\int f_0(x) dx = 1,$ we have that 
$f\mapsto \langle f,f_0\rangle_{L^{2}([0,1],\bR)}\mathbf{1}$ is a projection onto span$\{\mathbf{1}\}$ and so
$f\mapsto
f- \langle f,f_0\rangle_{L^{2}([0,1],\bR)}\mathbf{1}$ is a projection onto span$%
\{f_0\}^\perp $. 



\subsection{Optimising the response of the rate of mixing}

We now consider the linear response problem of optimising the rate of mixing.
Let $\lambda_0\in \bC$ denote an eigenvalue of $L_0$ strictly inside the unit circle with largest magnitude. 
From now on, whenever discussing the linear response of eigenvalues to kernel perturbations we assume  the conditions of Corollary \ref{cor:lin-resp-2nd-eval-HS}. 
We recall that $e$ and $\hat{e}$ are the eigenfunctions of $L_0$ and $L_0^*$, respectively, corresponding
to the eigenvalue $\lambda_0$.

%
To find the kernel perturbations that enhance mixing, we follow the approach taken in
\cite{ADF} (see also \cite{FS17,FKP} in the continuous time setting), namely perturbing our original dynamics $L_0$ in such a way that the modulus of the second eigenvalue of the perturbed dynamics decreases. 
Equivalently, we want to decrease the real part of the logarithm of the perturbed second eigenvalue.
The following result provides an explicit formula for this instantaneous rate of change.
Define
\begin{equation}  \label{eq:almost-soln-2nd-evalue}
 E(x,y):= \left(\Re(\hat{e})(x)\Re(e)(y) + \Im(\hat{e})(x)\Im(e)(y)\right)\Re(\lambda_{0})
+ \left(\Im(\hat{e})(x)\Re(e)(y) - \Re(\hat{e})(x)\Im(e)(y)\right)\Im(\lambda_{0}).
\end{equation}

\begin{lemma}\label{lem:2-eval-conj-uncess}
One has 
\begin{equation*}
\frac{d}{d\delta}\Re\left(\log\lambda_{\delta}\right)\bigg|_{\delta=0} = \frac{%
\big\langle\dot{k}, E\big\rangle_{L^2([0,1]^2,\bR)}}{|\lambda_{0}|^2}.
\end{equation*}
\end{lemma}

\begin{proof}
From \eqref{eq:lin-resp-2nd-eval}, we have that
\begin{equation}  \label{eq:lin-resp-2nd-eval-2-Re}
\Re(\dot{\lambda}_{0}) = \int_0^1\int_0^1\dot{k}(x,y)\left(\Re(\hat{e})(x)\Re(e)(y) + \Im(\hat{e})(x)\Im(e)(y)%
\right) dydx
\end{equation}
and
\begin{equation}  \label{eq:lin-resp-2nd-eval-2-Im}
\Im(\dot{\lambda}_{0}) = \int_0^1\int_0^1\dot{k}(x,y)\left(\Im(\hat{e})(x)\Re(e)(y) - \Re(\hat{e})(x)\Im(e)(y)
\right) dydx.
\end{equation}

Next, we note that
\begin{equation}  \label{eq:dif-re-log-mu-del}
\frac{d}{d\delta} \Re(\log \lambda_{\delta}) = \Re\left(\frac{d}{d\delta}%
\log\lambda_{\delta}\right) = \Re\left(\frac{d\lambda_{\delta}}{d\delta}\frac{1}{%
\lambda_{\delta}}\right).
\end{equation}

From \eqref{eq:lin-resp-2nd-eval-2-Re}-\eqref{eq:dif-re-log-mu-del}, we
obtain
\begin{equation*}
\frac{d}{d\delta}\Re\left(\log\lambda_{\delta}\right)\bigg|_{\delta=0}=
\Re\left(\frac{\dot{\lambda}_{0}}{\lambda_{0}}\right)=
\Re\left(\frac{\dot{\lambda}_{0}}{\lambda_{0}}\frac{\overline{\lambda_{0}}}{\overline{%
\lambda_{0}}}\right)=
\frac{\Re(\dot{\lambda}_{0})\Re(\lambda_{0})+\Im(\dot{\lambda}_{0})\Im(\lambda_{0})}{|%
\lambda_{0}|^2}=
\frac{\big\langle\k, E\big\rangle_{L^2([0,1]^2,\bR)}}{|\lambda_{0}|^2}.
\end{equation*}
\end{proof}

The function $\mathcal{J}(\dot{k})=\langle \dot{k},E\rangle$ is clearly continuous as a map from $(V_{\ker},\|\cdot\|_{L^2([0,1]^2)})$ to $\mathbb{R}$.
As in subsection \ref{subsubsec1}, suppose that $P$ is a closed, bounded, strictly convex subset of $V_{\ker}$ containing the zero element, and that $\mathcal{J}$ is not uniformly vanishing on $P$.
We wish to solve the following problem:

\begin{gproblem}
Find $\dot{k}\in P$ such that
\begin{equation}
\label{P2}
\langle \dot{k},E\rangle_{L^2([0,1]^2,\bR)}=\min_{\dot{h}\in P}
\langle \dot{k},E\rangle_{L^2([0,1]^2,\bR)}.  
\end{equation}
\end{gproblem}

We may immediately apply Proposition \ref{prop:exist} to obtain that \textit{there exists a solution to (\ref{P1}).}
If, in addition, $P$ is strictly convex, then by Proposition \ref{prop:uniqe} the solution to (\ref{P2}) \emph{is unique}.

\section{Explicit formulae for the optimal perturbations}
\label{sec:explicit}

Thus far we have not been specific about the feasible set $P$;  we take up this issue in this and the succeeding subsections to provide explicit formulae for the optimal responses in both problems (\ref{P1}) and (\ref{P2}).
First, we have not required that the perturbed
kernel $k_\delta$ in \eqref{perturb1} be nonnegative for $\delta>0$,
however, this is a natural
assumption.
To facilitate this, for $0<l<1$, define \begin{equation}
    \label{FSeqn}
F_l:=\{(x,y)\in[0,1]^2:k_0(x,y)\ge l\}\quad\mbox{ and }\quad
S_{k_0,l}:= \{k\in L^2([0,1]^2): \text{supp}(k)\subseteq F_l\}.
\end{equation}
 The set of allowable perturbations that we will  consider in the sequel is
\begin{equation}\label{PL}
P_l := V_{\ker}\cap S_{k_0,l}\cap B_1,
\end{equation}
where $B_1$ is the closed unit ball in
$L^2([0,1]^2)$.

We now begin verifying the conditions on $P_l$ and $\mathcal{J}$ required by Proposition \ref{prop:uniqe}.
First, $P_l$ is clearly bounded in $L^2([0,1]^2)$.
Second, we note that as long as $F_l$ has positive Lebesgue measure, the zero kernel is in the relative interior of $P_l$.
Third, the following lemma handles closedness of $P_l$.
Fourth, from this, since $V_{\ker}$ and $S_{k_0,l}$ are closed subspaces, $%
V_{\ker}\cap S_{k_0,l}$ is itself a Hilbert space, and hence, $P_l$ is strictly
convex.
Finally, sufficient conditions for the objective function to not uniformly vanish are given in Lemma \ref{lem:nonzeroobj}.



\begin{lemma}\label{Sk0lem}
The set $S_{k_0,l}$ is a closed subspace of $L^2([0,1]^2)$.
\end{lemma}

\begin{proof}
The fact that $S_{k_0,l}$ is a subspace is trivial. Let $\{k_n\}\subset
S_{k_0,l}$ and suppose $k_n\rightarrow_{L^2} k\in L^2([0,1]^2)$.
Further suppose $\{(x,y)\in [0,1]^2: k_0(x,y)< l\}$ is not a null set;  otherwise $S_{k_0,l}=L^2([0,1]^2)$ and the result immediately follows. 
Then, we have
\begin{equation*}
\int_{\{k_0\ge l\}} (k_n(x,y)-k(x,y))^2dydx + \int_{\{k_0 < l\}}
k(x,y)^2 dxdy\rightarrow 0.
\end{equation*}
Since $\int_{\{k_0\ge l\}} (k_n(x,y)-k(x,y))^2dydx\ge0$,
if $\int_{\{k_0< l\}} k(x,y)^2 dxdy>0$ then we obtain a contradiction; thus, $%
\int_{\{k_0< l\}} k(x,y)^2 dxdy=0$ and therefore $k=0$ a.e. on $\{(x,y)\in
[0,1]^2: k_0(x,y)< l\}$. Hence, $S_{k_0,l}$ is closed.
\end{proof}

Let 
\begin{equation}
    \label{Fly}
F_l^y:=\{x\in[0,1]:(x,y)\in F_l\},
\end{equation}
 and for $F_l\subset [0,1]^2$, define 
$$\Xi(F_l)=\{y\in [0,1]: m(F_l^y)>0\}.$$
The following lemma provides sufficient conditions for a functional of the general form we wish to optimise to not uniformly vanish.
The general objective has the form $\mathcal{J}(\dot{k})=\int\int \dot{k}(x,y)\mathcal{E}(x,y)\ dy\ dx$; in our first specific objective (optimising response of expectations) we put $\mathcal{E}(x,y)=((\text{Id}-L_0^*)^{-1}c)(x)\cdot f_0(y)$ and in our second specific objective (optimising mixing) we put $\mathcal{E}(x,y)=E(x,y)$ from (\ref{eq:almost-soln-2nd-evalue}).
Let $\mathcal{E}^+$ and $\mathcal{E}^-$ denote the positive and negative parts of $\mathcal{E}$.
For $y\in \Xi(F_l)$, let $A(y)=\int_{F_l^y} \mathcal{E}^+(x,y)\ dx$ and $a(y)=\int_{F_l^y} \mathcal{E}^-(x,y)\ dx$.
\begin{lemma}
\label{lem:nonzeroobj}
Assume that there is $\Xi'\subset \Xi(F_l)$ such that $m(\Xi')>0$ and $A(y),a(y)>0$ for $y\in \Xi'$.
Then there is a $\dot{k}\in P_l$ such that $\mathcal{J}(\dot{k})>0$.
\end{lemma}
\begin{proof}
For $y\in\Xi(F_l)$, set $\dot{k}(x,y)=\mathbf{1}_{F_l^y}(x)\left(a(y)\mathcal{E}^+(x,y)-A(y)\mathcal{E}^-(x,y)\right)$.
To show $\dot{k}\in P_l$ we need to check that (i) the support of $\dot{k}$ is contained in $F_l$ and (ii) $\int_{F_l^y} \dot{k}(x,y)\ dx=0$ for a.e. $y\in \Xi(F_l)$;  these points show $\dot{k}\in S_{k_0,l}\cap V_{\ker}$ and by trivial scaling we may obtain $\dot{k}\in B_1$.
Item (i) is obvious from the definition of $\dot{k}$.
For item (ii) we compute
\begin{equation*}
\int_{F_l^y} \dot{k}(x,y)\ dx=\int_{F_l^y}(a(y)\mathcal{E}^+(x,y)-A(y)\mathcal{E}^-(x,y))\ dx=a(y)A(y)-A(y)a(y)=0.
\end{equation*}
Finally, we check that $\mathcal{J}(\dot{k})>0$.
One has
\begin{eqnarray*}
\lefteqn{\int_{F_l} \dot{k}(x,y)\mathcal{E}(x,y)\ dx\ dy}\\
&=&\int_{F_l} \left(a(y)\mathcal{E}^+(x,y)-A(y)\mathcal{E}^-(x,y)\right)\cdot \mathcal{E}(x,y)\ dx\ dy\\
&=&\int_{F_l} a(y)(\mathcal{E}^+(x,y))^2+A(y)(\mathcal{E}^-(x,y))^2\ dx\ dy\\
&=&\int_{\Xi(F_l)}\left[\left(\int_{F_l^y} \mathcal{E}^-(x,y)\ dx\right)\cdot\left(\int_{F_l^y} (\mathcal{E}^+(x,y))^2\ dx\right)+\left(\int_{F_l^y} \mathcal{E}^+(x,y)\ dx\right)\cdot\left(\int_{F_l^y} (\mathcal{E}^-(x,y))^2\ dx\right)\right]\ dy.
\end{eqnarray*}
This final expression is positive due by the hypotheses of the Lemma.
\end{proof}
\begin{remark}
We note that in the situation where $\mathcal{E}(x,y)$ is in separable form $\mathcal{E}(x,y)=h_1(x)h_2(y)$---as in the case of optimising the derivative of the expectation of an observable $c$ , and in the case of optimising the derivative of a real eigenvalue---then $A(y)=h_2(y)\int_{F_l^y} h_1^+(x)\ dx$ and $a(y)=h_2(y)\int_{F_l^y} h_1^-(x)\ dx$. 
Because $h_2=f_0$ and $h_2=e$ are not the zero function, and $h_1=(\text{Id}-L_0^*)^{-1}c$ and $h_1=\hat{e}$ are both nontrivial signed functions, the conditions of Lemma \ref{lem:nonzeroobj} are relatively easy to satisfy.
\end{remark}



\subsection{Maximising the expectation of an observable}
\label{sec:explicit-formula}

In this section we provide an explicit formula for the optimal kernel perturbation to increase the expectation of an observation function $c$ by the greatest amount.
Since the objective function in \eqref{P1} is linear in $\dot{k}$, a
maximum will occur on $\partial B_1\cap V_{\ker}\cap S_{k_0,l}$ (i.e.\ we
only need to consider the optimization over the unit sphere and not the unit
ball). 
Thus, we consider the
following reformulation of the general problem \ref{prob1}:

\begin{problem}
\label{prob3}
Given $l > 0 $ and $c\in $ span$\{f_0\}^\perp$, solve
\begin{eqnarray}
\min_{\dot{k}\in V_{\ker }\cap S_{k_0,l}} &&-\big\langle c,R(\dot{k})%
\big\rangle_{L^{2}([0,1],\bR)}  \label{obj_lin_fun-min} \\
\mbox{subject to} &&\Vert \dot{k}\Vert _{L^{2}([0,1]^{2})}^{2}-1=0.
\label{constraint-bound}
\end{eqnarray}

\end{problem}

Our first main result is:
\begin{theorem}
\label{thm:explicit-formula} 
Let $L_0:L^2\rightarrow L^2$ be an integral operator 
with the stochastic kernel $k_0\in L^2([0,1]^2)$. 
Suppose that $L_0$ satisfies $(A1)$ of Theorem 
\ref{th:linearresponse copy(1)} and 
that there is a $\Xi'\subset \Xi(F_l)$ with $m(\Xi')>0$ and $f_0(y)>0, \int_{F_l^y} ((\text{Id}-L_0^*)^{-1}c)^+(x)\ dx>0$, and $\int_{F_l^y} ((\text{Id}-L_0^*)^{-1}c)^-(x)\ dx>0$ 
for $y\in\Xi'$.
Then the unique
solution to  Problem \ref{prob3}   is
\begin{equation}  \label{opt-soln}
\dot{k}(x,y)=
\begin{cases}
\frac{f_0(y)}{\alpha}\left(((\text{Id}-L_{0}^{\ast })^{-1}c)(x)-\frac{%
\int_{F_l^y}((\text{Id}-L_{0}^{\ast })^{-1}c)(z)dz}{m(F_l^y)} \right) & (x,y)\in F_l, \\
0 & \text{otherwise},%
\end{cases}%
\end{equation}
where $\alpha>0$ is selected so that $\|\dot{k}%
\|_{L^2([0,1]^2)}=1$. 
Furthermore, if $c\in W:=$ span$\{f_0\}^\perp\cap L^\infty$, $k_0\in L^\infty([0,1]^2)$, and $k_0$ is such that $%
L_0:L^1\rightarrow L^1$ is compact, then $\dot{k}\in
L^\infty([0,1]^2)$.
\end{theorem}
\begin{proof}
See Appendix \ref{apx:thm:explicit-formula}.
\end{proof}
Note that the expression for the optimal perturbation $\dot{k}$ in (\ref{opt-soln}) depends only on $k_0$ and $c$.
This is in part a consequence of the fact that the linear response formula (\ref{kerrsep}) depends only on the first order term $\dot{k}$ (the ``direction'' of the perturbation) in the expansion of $k_\delta$.
Thus, in order to find the unique perturbation that optimises our linear response, we seek the best ``direction'' for  the perturbation. 
Similar comments hold for our other three optimal linear perturbation results in later sections.




\begin{remark}
In certain situations we may desire to make non-infinitesimal perturbations $k_\delta:= k_0 + \delta\cdot \dot{k}$ that remain stochastic for small $\delta>0$.
If
$\dot{k}\in L^\infty ([0,1]^2)\cap V_{\ker}\cap S_{k_0,l}$,
clearly 
 $k_\delta= k_0 + \delta\cdot \dot{k}$ satisfies $\int k_\delta(x,y) dx =1$
for a.e. $y$. Also, as we are only perturbing at values where $k_0\ge
l>0 $, and since $\dot{k}$ is essentially bounded, there exists a $\bar{\delta}>0$ such that $k_\delta \ge 0$ a.e. for all $%
\delta\in (0,\bar\delta)$.
In summary, for $\delta \in (0,\bar{\delta})$, $k_\delta
$ is a stochastic kernel.

The compactness condition on $L_0:L^1\to L^1$ required for essential boundedness of $\dot{k}$ can be addressed as follows.
A criterion for $L_0$ to be compact on $L^1([0,1])$ is the following (see \cite{E}): Given $%
\varepsilon>0$ there exists $\beta>0$ such that for a.e. $y\in [0,1]$ and $%
\gamma\in {\mathbb{R}}$ with $|\gamma|<\beta$,
\begin{equation*}
\int_{{\mathbb{R}}}\big|\tilde{k}(x+\gamma,y)-\tilde{k}(x,y)\big|%
dx<\varepsilon,
\end{equation*}
where $\tilde{k}:\mathbb{R}\times [0,1]\to\mathbb{R}$ is defined by
\begin{equation*}
\tilde{k}(x,y) =
\begin{cases}
k_0(x,y) & x\in[0,1], \\
0 & \text{otherwise}.%
\end{cases}%
\end{equation*}
A class of kernels that satisfy this are
essentially bounded kernels $k_0:[0,1]\times[0,1]\rightarrow\bR$ that are uniformly
continuous in the first coordinate.
Such a class naturally arises in our dynamical systems settings.
\end{remark}

\subsection{Maximally increasing the mixing rate}
\label{seceig}

Let $\lambda_0\in \bC$ denote a geometrically simple eigenvalue of $L_0$ strictly inside the unit circle and $e$ and $\hat{e}$ denote the corresponding eigenvectors of $L_0$ and $L_0^*$, respectively.
Our results concerning optimal rate of movement of $\lambda_0$ under system perturbation work for any $\lambda_0$ as above, but eigenvalues of largest magnitude inside the unit circle have the additional significance of controlling the exponential rate of mixing. 
We therefore primarily focus on these eigenvalues and in this section we consider again the linear response problem for enhancing the rate of mixing, now providing explicit formulae for optimal perturbations and the response.
%

Since we
are again interested in kernel perturbations that will ensure that the perturbed
kernel $k_\delta$ is nonnegative, 
we consider the constraint set $P_l$, as in Section \ref{subsec:abst-opt-setng},
where $0<l<1$. 
The objective function of \eqref{P2} is linear and therefore, we only need
to consider the optimization problem on $V_{\ker}\cap S_{k_0,l}\cap \partial
B_1$. Thus, to obtain the perturbation $\dot{k}$ that will enhance the
mixing rate, we solve the following optimization problem:
\begin{problem}\label{prob3.1}
Given $l > 0$, solve
\begin{eqnarray}
\min_{\dot{k}\in V_{\ker }\cap S_{k_0,l}} &&\big\langle\dot{k},E\big\rangle%
_{L^{2}([0,1]^{2},\bR)\label{obj_lin_fun-min-2nd-eval}} \\
\mbox{such that} &&\Vert \dot{k}\Vert _{L^{2}([0,1]^{2},\bR)}^{2}-1=0,
\label{constraint-bound-2nd-eval}
\end{eqnarray}
where $E$ is defined in \eqref{eq:almost-soln-2nd-evalue}.
\end{problem}
\begin{theorem}
\label{thm:explicit-formula-eval} Let $L_0:L^2([0,1],\bC)\rightarrow L^2([0,1],\bC)$ be an integral operator 
with the stochastic kernel $k_0\in L^2([0,1]^2,\bR)$. 
Suppose that $L_0$ satisfies $(A1)$ of Theorem 
\ref{th:linearresponse copy(1)} and 
that there is a $\Xi'\subset \Xi(F_l)$ with $m(\Xi')>0$, and $\int_{F_l^y} E(x,y)^+\ dx>0$ and $\int_{F_l^y} E(x,y)^-\ dx>0$ 
for $y\in\Xi'$.
Then, the unique solution to Problem \ref{prob3.1}  is
\begin{equation}  \label{eq:soln-2nd-eval}
\dot{k}(x,y)=
\begin{cases}
\frac{1}{\alpha}\left(\frac{1}{m(F_l^y)}\int_{F_l^y} E(x,y)dx - E(x,y)\right) & (x,y)\in F_l\\
0 & \text{otherwise},%
\end{cases}%
\end{equation}
where $E$ is given in \eqref{eq:almost-soln-2nd-evalue} and $\alpha>0$ is
selected so that $\|\dot{k}\|_{L^2([0,1]^2,\bR)}=1$.
Furthermore, if $k_0\in L^\infty([0,1]^2,\bR)$ then $\dot{k}\in L^\infty([0,1]^2,\bR)$.
\end{theorem}
\begin{proof}
See Appendix \ref{apx:thm:explicit-formula-eval}.
\end{proof}

If $\lambda_0$ is real, the optimal kernel has a simpler form:
\begin{corollary}
\label{cor:maxmix}
If $\lambda_{0}$ is real and $%
k_0\ge l$, then the solution to Problem \ref{prob3.1}  is
\begin{equation}
\label{optkernelevaleqn}
\dot{k}(x,y) = sgn(\lambda_0) \frac{e(y)}{\|e\|_2}\left( \frac{\langle\hat{e},\mathbf{1}\rangle_{L^2([0,1],\bR)}\mathbf{1} -%
\hat{e}(x)}{\|\langle\hat{e},\mathbf{1}\rangle_{L^2([0,1],\bR)}\mathbf{1}-\hat{e}\|_2}\right).
\end{equation}
\end{corollary}

\begin{proof}
We have $E(x,y) = \lambda_0\hat{e}(x)e(y)$; thus, 
the solution to the
optimization problem \eqref{obj_lin_fun-min-2nd-eval}-%
\eqref{constraint-bound-2nd-eval} is
\begin{equation*}
\dot{k}(x,y) = (\lambda_{0}/\alpha)\left(\int_0^1 \hat{e}(x)dx - \hat{e}%
(x)\right)e(y),
\end{equation*}
where $\alpha>0$ is the normalization constant such that $\|\dot{k}%
\|_{L^2([0,1]^2,\bR)}^2=1$. 
\end{proof}


\section{Linear response for map perturbations}

\label{sec:sys-add-noise}

In this section we consider
random dynamics governed
by the composition of a deterministic map $%
T_{\delta }$, $\delta\in[0,\bar{\delta})$, and additive i.i.d.\ perturbations, or ``additive noise''.
We will assume that
the noise is distributed according to a certain Lipschitz kernel $\rho 
$ 
and impose a
reflecting boundary condition that ensures that the dynamics remain in the interval $[0,1]$. More precisely, we consider a random
dynamical system 
whose trajectories are given
by%
\begin{equation}
x_{n+1}=T_{\delta }(x_{n})\ \hat{+}\ \omega _{n}, \label{systm}
\end{equation}%
where $\hat{+}$ is the ``boundary reflecting" sum, defined by $a\hat{+}b:=\pi (a+b)$, and
 $\pi :\mathbb{R}\rightarrow \lbrack 0,1]$  is the piecewise
linear map
$\pi (x)=\min_{i\in \mathbb{Z}}|x-2i|$.
%
We assume throughout that 
\begin{itemize}
\item[(T1)] $T_{\delta }:[0,1]\rightarrow \lbrack 0,1]$ is a Borel-measurable map for each $\delta\in [0,\bar\delta)$,
\item[(T2)] $\omega _{n}$ is an i.i.d. process distributed according to a
probability density $\rho \in Lip(\mathbb{R})$, supported on $[-1,1]$ with Lipschitz constant $K$.
\end{itemize}

\subsection{Expressing the map perturbation as a kernel perturbation}

In this subsection we describe precisely the kernel of the transfer operator of the system (\ref{systm}).
Associated with the  process \eqref{systm} is an integral-type transfer
operator $L_\delta$, which we will derive (following the method of \S 10.5 in \cite{LM}). 
Noting that $|\pi'(z)|=1$ for all $z\in\mathbb{R}$, the Perron-Frobenius operator $P_\pi:L^1(\mathbb{R})\rightarrow L^1([0,1])$
associated to the map $\pi$ is given by 
\begin{equation}\label{eq:PF-pi}
P_\pi f(x) = \sum_{z\in\pi^{-1}(x)}f(z)=\sum_{i\in 2\mathbb{Z}}(f(i+x)+f(i-x)).
\end{equation}
For $b\in\bR$ consider the shift operator $\tau_b$ defined by
$(\tau_b g)(y):=g(y+b)$ for $g\in Lip(\mathbb{R})$. 
For the process \eqref{systm}, suppose that $x_n$ has the distribution $f_n:[0,1]\to\mathbb{R}^+$ (i.e.\ $f_n\in L^1,\ f_n\ge 0$ and $\int f_n\ dm =1$).
We note that
$T_\delta(x_n)$ and $\omega_{n}$ are independent
and thus the joint density of $(x_n,\omega_n)\in[0,1]\times[-1,1]$ is $f_n\cdot\rho$. 
Let $h:[0,1]\rightarrow\bR$ be a bounded, measurable function and let $\mathbb{E}$ denote expectation with respect to Lebesgue measure; we then 
compute 
\begin{equation*}
\begin{aligned}
\mathbb{E}(h(x_{n+1})) 
&= \int_{-\infty}^{\infty}\int_0^1 h(\pi(T_\delta(y)+z))f_n(y)\rho(z)dydz\\
&= \int_0^1\int_{-\infty}^{\infty} h(\pi(z'))f_n(y)\rho(z'-T_\delta(y))dz'dy\\
&= \int_0^1 f_n(y)\int_{-\infty}^{\infty} h(\pi(z'))(\tau_{-T_\delta(y)}\rho)(z')dz'dy\\
&= \int_0^1 f_n(y)\int_{0}^{1} h(z') (P_\pi\tau_{-T_\delta(y)}\rho)(z')dz' dy,
\end{aligned}
\end{equation*}
where the last equality 
follows from the duality of the Perron-Frobenius and the Koopman 
operators for $\pi$.
Since $\mathbb{E}(h(x_{n+1})) = \int_{0}^1 h(x) f_{n+1}(x)dx$, and $h$ is arbitrary, the map
$f_n\mapsto f_{n+1}$ is given by 
\begin{equation*}
f_{n+1}(z')=\int_0^1 (P_\pi\tau_{-T_\delta(y)}\rho)(z') f_n(y)dy
\end{equation*}
for all $z'\in[0,1]$.
Thus, for $\delta\in[0,\bar{\delta})$ the integral operator $L_\delta:L^2([0,1])\rightarrow L^2([0,1])$
associated to the process \eqref{systm} is given by
\begin{equation}
L_{\delta }f(x)=\int k_{\delta }(x,y)f(y)dy,  \label{kernelL2}
\end{equation}%
where
\begin{equation}
k_{\delta }(x,y)=(P_\pi\tau_{-T_\delta(y)}\rho)(x)
\label{eq:ker-add-noise-explicit}
\end{equation}%
and $x,y\in [0,1]$. 

\begin{lemma}\label{lem:prop-add-noise-ker}
The kernel \eqref{eq:ker-add-noise-explicit} is a stochastic kernel in $L^\infty([0,1]^2)$. 
\end{lemma}
\begin{proof}
Stochasticity and nonnegativity of $k_\delta$ follow from stochasticity and nonnegativity of $\rho$ and the fact that Perron-Frobenius operators preserve these properties.
Essential boundedness of $k_\delta$ follows from the facts that $\rho$ is Lipschitz (thus essentially bounded), $\tau$ is a shift, and $P_\pi$ is constructed from a finite sum because $\rho$ has compact support. 
\end{proof}

\begin{proposition}
\label{kdottt}Assume that $k_{\delta }$  arising from the system $(T_{\delta },\rho
)$ is given by %
\eqref{eq:ker-add-noise-explicit}. Suppose that the family of interval maps $\{T_\delta\}_{\delta\in[0,\bar\delta)}$ satisfies
\begin{equation*}
T_{\delta }=T_{0}+\delta \cdot \dot{T} +t_\delta,
\end{equation*}
where $\dot{T},t_\delta\in L^{2}$ and $\|t_\delta\|_{2} =o(\delta)$.
Then
\begin{equation*}
k_{\delta }=k_{0}+\delta \cdot \dot{k}+r_\delta
\end{equation*}%
where $\k\in L^2([0,1]^2)$ is given by
\begin{equation}\label{derivk}
\dot{k}(x,y)=-\left(P_\pi\left(\tau_{-T_0(y)}\frac{d\rho}{dx}\right) \right)(x)\cdot \dot{T}(y)  
\end{equation}
and $r_\delta\in L^{2}([0,1]^{2})$ satisfies $\|r_\delta\|_{L^2([0,1]^2)}=o(\delta)$.

If additionally, $d\rho/dx$ is Lipschitz and the derivative of the map $\delta\mapsto T_\delta$ with respect to $\delta$ varies continuously in $L^2$ in a neighborhood of $\delta=0$, then $\delta\mapsto k_\delta$ has a continuous derivative with respect to $\delta$ in a neighborhood of $\delta=0$.
\end{proposition}

\begin{proof}

We show that $\|k_{\delta }(x,y)-k_{0}(x,y) - \delta\cdot\k(x,y)\|_{L^2([0,1]^2)}=o(\delta)$, where 
$\dot{k}$ is as in (\ref{derivk}). We have
\begin{eqnarray}
\label{prop62eqn}\lefteqn{\left\|k_{\delta }(x,y)-k_{0}(x,y) - \delta\cdot\k(x,y)\right\|_{L^2([0,1]^2)}}\\
\nonumber&\le&\left\|(P_\pi\tau_{-T_\delta(y)}\rho)(x) - (P_\pi\tau_{-(T_0(y)+\delta\cdot\dot{T}(y))}\rho)(x)\right\|_{L^2([0,1]^2)}\\
\nonumber&&+\left\|(P_\pi\tau_{-(T_0(y)+\delta\cdot\dot{T}(y))}\rho)(x) - (P_\pi\tau_{-T_0(y)}\rho)(x)- \delta \left(-\left(P_\pi\left(\tau_{-T_0(y)}\frac{d\rho}{dx}\right) \right)(x)\cdot \dot{T}(y)\right)\right\|_{L^2([0,1]^2)}.
\end{eqnarray}
We begin by showing that the first term on the right hand side of (\ref{prop62eqn}) is $o(\delta)$. 
Since $\rho$ is Lipschitz with constant $K$, one has
\begin{equation}
    \label{Keqn}
\big|(\tau_{-(T_\delta(y))}\rho)(x)-(\tau_{-(T_0(y)+\delta\cdot\dot{T}(y))}\rho)(x)\big| = \big|\rho(x-T_\delta(y))-\rho(x-T_0(y)-\delta\cdot\dot{T}(y))\big|\le K|t_\delta(y)|.
\end{equation}
Because the support of $\tau_{-(T_\delta(y))}\rho-\tau_{-(T_0(y)+\delta\cdot\dot{T}(y))}\rho$ is contained in 2 intervals, each of length 2, by (\ref{Keqn}) and Lemma \ref{lem:upbndP}, we therefore see that
 $$\left\|(P_\pi\tau_{-T_\delta(y)}\rho)(x) - (P_\pi\tau_{-(T_0(y)+\delta\cdot\dot{T}(y))}\rho)(x)\right\|_{L^2([0,1]^2)}\le 6K\|t_\delta\|_{L^2}=o(\delta).$$

Next we show that the second term on the right hand side of (\ref{prop62eqn}) is $o(\delta)$. 
Using the definition of the derivative and the fact that $\rho$ is differentiable a.e.\ we see that 
\begin{equation}\label{eq:deriv-rho}
\lim_{\delta \rightarrow 0}D(\delta):=\lim_{\delta \rightarrow 0}\left[\frac{\rho(x-T_{0}(y)-\delta\cdot \dot{T}%
(y))-\rho(x-T_{0}(y))}{\delta }-\left(-\frac{d\rho}{dx}(x-T_{0}(y))\dot{T%
}(y)\right)\right]=0
\end{equation}
for a.e.\ $x,y$. Since $\bigg|\frac{\rho(x-T_{0}(y)-\delta\cdot \dot{T}%
(y))-\rho(x-T_{0}(y))}{\delta }\bigg|\leq K\dot{T}(y)$, by
dominated convergence the limit (\ref{eq:deriv-rho}) also converges in $L^{2}.$ 
Hence, applying Lemma \ref{lem:upbndP} to the second term on the right hand side of (\ref{prop62eqn}), noting that $D(\delta)$ in (\ref{eq:deriv-rho}) is square-integrable and supported in at most 3 intervals of length at most 2, we obtain
\begin{eqnarray*}
&&{\left\|(P_\pi\tau_{-(T_0(y)+\delta\cdot\dot{T}(y))}\rho)(x) - (P_\pi\tau_{-T_0(y)}\rho)(x)- \delta \left(-\left(P_\pi\left(\tau_{-T_0(y)}\frac{d\rho}{dx}\right) \right)(x)\cdot \dot{T}(y)\right)\right\|_{L^2([0,1]^2)}}\\
\qquad&\le&9\delta D(\delta)=o(\delta).
\end{eqnarray*}
Regarding the final statement, suppose that $\delta\mapsto T_\delta$ has a continuous derivative with respect to $\delta$ at a neighborhood of $\delta=0$.
This implies that $\dot{T}$ exists and varies continuously on a small interval $[0,\delta^*]$, with $0<\delta^*\le \bar{\delta}$. 
Denote the derivative $dT_\delta/d\delta$ at $\delta$ by $\dot{T}_\delta$, and similarly for $\dot{k}$.
One has
\begin{eqnarray*}
\|\dot{k}_\delta-\dot{k}_0\|_{L^2([0,1]^2)}&=&\left\|\left(P_\pi\left(\tau_{-T_\delta(y)}\frac{d\rho}{dx}\right) \right)(x)\cdot \dot{T}_\delta(y)-\left(P_\pi\left(\tau_{-T_0(y)}\frac{d\rho}{dx}\right) \right)(x)\cdot \dot{T}_0(y)\right\|_{L^2([0,1]^2)}\\
&\le&\left\|\left(P_\pi\left(\tau_{-T_\delta(y)}\frac{d\rho}{dx}\right) \right)(x)\cdot (\dot{T}_\delta(y)-\dot{T}_0(y))\right\|_{L^2([0,1]^2)}\\
&&\quad+\left\|\left[\left(P_\pi\left(\tau_{-T_\delta(y)}\frac{d\rho}{dx}\right) \right)(x)-\left(P_\pi\left(\tau_{-T_0(y)}\frac{d\rho}{dx}\right) \right)(x)\right]\cdot \dot{T}_0(y)\right\|_{L^2([0,1]^2)}\\
&\le&3\|d\rho/dx\|_2\|\dot{T}_\delta-\dot{T}_0\|_2+6{\rm Lip}(d\rho/dx)\|\delta\cdot\dot{T}_0+r_\delta\|_2\|\dot{T}_0\|_2,
\end{eqnarray*}
where the final inequality follows from Lemma \ref{lem:upbndP} applied to each term in the previous line, noting that $\rho$ is supported in a single interval of length 2.
The first term in the final inequality goes to zero as $\delta\to 0$ by continuity of $\dot{T}$, and the second term goes to zero as $\delta\to 0$ since $\|r_\delta\|_2\to 0$.
\end{proof}

\subsection{A formula for the linear response of the invariant measure and continuity with respect to map perturbations}

By considering the kernel form of map perturbations, we can apply  
Corollary \ref{LR} to obtain the following.
\begin{proposition}
\label{prop:lin-resp-formula-add-noise}
Let $L_\delta:L^2\rightarrow L^2$, $\delta\in [0,\bar{\delta})$,
be the integral operators in \eqref{kernelL2} with the
kernels $k_\delta$ as in \eqref{eq:ker-add-noise-explicit}. Suppose that $L_0$ satisfies $(A1)$ of Theorem \ref{th:linearresponse copy(1)}.
Then 
the kernel $\k$ in \eqref{derivk} is in $V_{\ker}$ and
\begin{equation*}
\lim_{\delta \rightarrow 0}\frac{f_{\delta }-f_{0}}{\delta }%
= -(\text{Id}-L_{0})^{-1}\int_0^1 \left(P_\pi\left(\tau_{-T_0(y)}\frac{d\rho}{dx}\right) \right)(x)
\dot{T}(y)f_{0}(y)dy,
\end{equation*}
with convergence in $L^{2}.$
\end{proposition}

\begin{proof}
The result is a direct application of Corollary \ref{LR}; we verify its
assumptions. From Lemma \ref{lem:prop-add-noise-ker}, $k_\delta\in L^2([0,1]^2)$ is a stochastic 
kernel and so $L_\delta$ is an integral-preserving compact operator. From Proposition \ref{kdottt}, $k_\delta$ has the form \eqref{perturb1}.
Thus, we can apply Corollary \ref{LR} to obtain the result.
\end{proof}

\begin{remark}\label{6.4}
If $T$ is 
covering\footnote{We say $T$ is covering if for each small open interval $I\subseteq [0,1]$ there is $n=n(I)$ such that $T^n(I)=[0,1]$.} 
and  $\rho$ is 
strictly positive in a neighbourhood of zero one can show the corresponding transfer operator $L_0$  satisfies assumption $(A1)$ of Theorem \ref{th:linearresponse copy(1)}, using arguments similar to e.g.\ \cite{ZH07} Proposition 8.1, \cite{Fr13} Lemmas 3 and 10, or \cite{GG}, Lemma 41.
Let $f\in L^1$ have zero average: $\int_{[0,1]} f=0$.
If $f$ is $0$ almost everywhere, $L_0^n(f)= 0$ and we are done. 
Otherwise, given $\epsilon<0$, we can find an $f_1$ such that $\|f-f_1\|_1<\epsilon$ and $f_1$ is positive in some small interval $I\subset [0,1]$. 
Since $\rho$ is positive in a neighbourhood of zero,  
$\supp(L_0(f_1^+))\supset T(I)$. 
By the covering condition there is some $n'\in {\mathbb N}$ such that $\supp(L_0^{n'}(f_1^+))=[0,1]$.
It is then standard to deduce that there is an $n_0\ge n'$ such that $\|L_0^n(f_1)\|_1<\epsilon$ for $n\ge n_0$.
%
Since the transfer operator contracts the $L^1$ norm, then $||L_0^n f ||_1 \leq 2 \epsilon$ for $n\ge n_0$ and since $\epsilon$ was arbitrary, this implies that $L_0$ satisfies $(A1)$.
\end{remark}

Let the linear response $\widehat{R}:L^2\rightarrow L^2$ of the invariant density
be defined as
\begin{equation}  \label{Linreap2}
 \hspace*{-.5cm}\widehat{R}(\dot{T}):= -(\text{Id}-L_{0})^{-1}\int_0^1 \left(P_\pi\left(\tau_{-T_0(y)}\frac{d\rho}{dx}\right) \right)(x)
\dot{T}(y)f_{0}(y)dy.
\end{equation}
\begin{lemma}\label{lem:cont-resp-ref-bnd}
The function $\widehat{R}:L^2\rightarrow L^2$ is continuous.
\end{lemma}

\begin{proof}
We have
\begin{equation*}
\widehat{R}(\dot{T}_1)-\widehat{R}(\dot{T}_2)= -(\text{Id}-L_0)^{-1}\int_0^1
\tilde{k}(x,y)\left(\dot{T}_1(y)-\dot{T}_2(y)\right)dy,
\end{equation*}
where $\tilde{k}(x,y) := \left(P_\pi\left(\tau_{-T_0(y)}\frac{d\rho}{dx}\right) \right)(x)f_0(y)$.
Since $\frac{d\rho}{dx}\in L^\infty$, we have 
$\left(P_\pi\left(\tau_{-T_0(y)}\frac{d\rho}{dx}\right) \right)(x)\in L^\infty([0,1]^2)$. 
From inequality \eqref{KF2}, we then have $f_0\in L^\infty$
and so
$\tilde{k}\in L^\infty([0,1]^2)$. 
We finally have
\begin{equation*}
\|\widehat{R}(\dot{T}_1)-\widehat{R}(\dot{T}_2)\|_2 \le l \|(\text{Id}%
-L_0)^{-1}\|_{V\rightarrow V} \|\tilde{k}\|_{L^2([0,1]^2)}\cdot\|\dot{T}_1-\dot{T}_2\|_2.
\end{equation*}

\end{proof}

\subsection{A formula for the linear response of the dominant eigenvalues and continuity with respect to map perturbations}
\label{sec:responsemapeig}

We are also able to
express the linear 
response of the dominant eigenvalues as a function
of the perturbing map $\dot{T}$.  
Define 
$$H(y)=-\bar{e}(y)\int_0^1\left(P_\pi\left(\tau_{-T_0(y)}\frac{d\rho}{dx}\right)\right)(x)\hat{e}(x)dx.$$ 
\begin{proposition}
\label{prop:lin-resp-formula-add-noise-eval} Let $L_\delta:L^2([0,1],\bC)\rightarrow L^2([0,1],\bC)$, $\delta\in [0,\bar{\delta})$,
be integral operators generated by the
kernels $k_\delta $ as in \eqref{eq:ker-add-noise-explicit},  assume that $d\rho/dx$ is Lipschitz and $\delta\mapsto T_\delta$ is $C^1$.
Let $\lambda_{\delta}$ be an eigenvalue of $L_\delta$ with second
largest magnitude strictly inside the unit disk.
Suppose that $L_{0}$ satisfies $(A1)$ of Theorem 
\ref{th:linearresponse copy(1)}
and $\lambda_{0}$ is geometrically simple. Then
\begin{equation}  \label{eq:lin-resp-formula-add-noise-eval}
\frac{d\lambda_{\delta}}{d\delta}\bigg|_{\delta =0} = \langle H, \dot{T}
\rangle_{L^2([0,1],\mathbb{C})},
\end{equation}
where $e$ is the eigenvector of $L_0$ associated to the eigenvalue $%
\lambda_{0} $ and $\hat{e}$ is the eigenvector of $L_0^*$ associated to the
eigenvalue $\lambda_{0}$.
\end{proposition}

\begin{proof}
Since $k_\delta\in L^2([0,1]^2,\bR)$, 
$L_\delta:L^2([0,1],\bC)\rightarrow L^2([0,1],\bC)$ is compact. From Lemma \ref{lem:prop-add-noise-ker}
we have that $k_\delta$ is a stochastic kernel and so $L_\delta$ preserves the integral 
(i.e.\ it satisfies \eqref{eq:int-pers-cond}).
By Proposition \ref{kdottt} the kernel $k_\delta$ is in the form \eqref{perturb1} and the map $\delta\mapsto k_\delta$ is $C^1$.
By Lemma \ref{ldot2} we see that
$\delta\mapsto L_\delta$ is $C^1$,
%
where the 
derivative operator $\dot{L}$ is the integral operator with the kernel 
$\k$. Using the 
assumption that $L_0$ is mixing and $\lambda_{0}$ is
geometrically simple, we apply Proposition \ref%
{th:lin-resp-2nd-eval} to obtain $\frac{d\lambda_{\delta}}{d\delta}\big|_{\delta =0} = \langle \hat{e}, \dot{L}
e\rangle_{L^2([0,1],\mathbb{C})}$. Finally, we compute 
\begin{equation*}
\begin{aligned}
\langle \hat{e}, \dot{L}e\rangle_{L^2([0,1],\mathbb{C})} &= \int_0^1\hat{e}(x)\overline{\int_0^1 \k(x,y)e(y)dy}dx\\
&= \int_0^1\int_0^1\hat{e}(x)\k(x,y)\bar{e}(y)dxdy\\
&= -\int_0^1 \bar{e}(y)\int_0^1\left(P_\pi\left(\tau_{-T_0(y)}\frac{d\rho}{dx}\right) \right)(x)\hat{e}(x)dx\ \dot{T}(y)dy\\
&= \langle H, \dot{T}\rangle_{L^2([0,1],\mathbb{C})}.
\end{aligned}
\end{equation*}
\end{proof}

From \eqref{eq:lin-resp-formula-add-noise-eval}, the linear response of the 
dominant eigenvalues is continuous with respect to map perturbations. 

\begin{lemma}
\label{mapevalctslem}
The eigenvalue response function $\check{R}:L^2\to\mathbb{C}$ given by  $\check{R}(\dot{T})=\langle H,\dot{T}\rangle$ is continuous.
\end{lemma}
\begin{proof}
This follows from
Cauchy-Schwarz and the fact that $H\in L^2([0,1],\bC)$; the latter claim
follows from the fact that
$\left(P_\pi\left(\tau_{-T_0(y)}\frac{d\rho}{dx}\right) \right)(x)\in L^\infty([0,1]^2,\bR)$ (see proof of 
Lemma \ref{lem:cont-resp-ref-bnd}) and
that $e,\hat{e}\in L^\infty([0,1],\bC)$ 
(which follows from \eqref{KF2} and the fact that $k_0\in L^\infty([0,1]^2,\bR)$, see Lemma \ref{lem:prop-add-noise-ker}). 
\end{proof}

\section{Optimal linear response for map perturbations}
\label{subsec:exp-obs-add-noise}

In this section we derive formulae for the map perturbations that maximise our two types of linear response.
We begin by formalising the set of allowable map perturbations then state the formulae.

\subsection{The feasible set of perturbations}

Before we formulate the optimization problem, we note that in this setting, we
require some restriction on the space of allowable perturbations to $T_0$
if we are to interpret $T_0+\delta\dot{T}$ as a map of the unit interval for some $\delta$ strictly greater than 0 (a non-infinitesimal perturbation).
With this in mind, let
$\ell>0$ and $\widetilde{F}_\ell:=\{x\in[0,1]:\ell\le T_0(x)\le 1-\ell\}$;  it will turn out that we obtain for free that $\dot{T}\in L^\infty$.
Note that in principle, $\ell>0$ can be taken as small as one likes, and indeed if one wishes to consider only infinitesimal perturbations $\dot{T}$ then one may set $\widetilde{F}_\ell=\widetilde{F}_0=[0,1]$.
Of course if $T:S^1\to S^1$ then may may use $\widetilde{F}_\ell=\widetilde{F}_0=[0,1]$ even for non-infinitesimal perturbations.
Recalling that 
in Proposition \ref{kdottt} we are considering $L^2$ perturbations $\dot{T}$ of the map $T_0$,  we define
\begin{equation}\label{lem:rest-map-supp-closed-subspc}
S_{T_0,\ell}:= \{T\in L^2: \text{supp}(T)\subseteq\widetilde{F}_\ell\}.
\end{equation}
\begin{lemma}\label{Lem7.1}
$S_{T_0,\ell}$ is a closed subspace of $L^2$.
\end{lemma}
\begin{proof}
It is clear that $S_{T_0,\ell}$ is a subspace. To show it is closed,
let $\{f_n\}\subset S_{T_0,\ell}$ and suppose that $f_n\rightarrow_{L^2} f\in L^2$.
Further, suppose that $\widetilde{F}_\ell$ is not $[0,1]$ up to measure zero;  otherwise $S_{T_0,\ell}=L^2$, which is closed.
Then, we have
\begin{equation*}
\|f_n-f\|_2^2=\int_{\widetilde{F}_\ell}(f_n(x)-f(x))^2dx + \int_{\widetilde{F}_\ell^c}f(x)^2\ dx\rightarrow0.
\end{equation*}
If $\int_{\widetilde{F}_\ell^c}f(x)^2dx>0$, we obtain a contradiction since
$\int_{\widetilde{F}_\ell}(f_n(x)-f(x))^2dx\ge 0$; thus, $\int_{\widetilde{F}_\ell^c}f(x)^2dx=0$
and so $f=0$ a.e. on $\widetilde{F}_\ell^c$. Hence, $S_{T_0,\ell}$ is closed.
\end{proof}

For the remainder of this section, the set of allowable perturbations that we consider is
\begin{equation}\label{eq:allow-map-pert}
P_\ell:= S_{T_0,\ell}\cap B_1,
\end{equation}
where $B_1$ is the unit ball in $L^2$.
Since $S_{T_0,\ell}$ is a closed subspace of $L^2$, it is itself a Hilbert space
and so $P_\ell$ is strictly convex.
The following lemma concerns the existence of a perturbation $\dot{T}$ for which our objectives will be nonzero;  that is, our objective $\mathcal{J}$ is not uniformly vanishing.
Denote $\mathcal{P}(x,y):=P_\pi\left(\tau_{-T_0(y)}\frac{d\rho}{dx}\right)(x)$ 
and
let $$\mathcal{J}(\dot{T}):=\int_{\Xi(\widetilde{F}_\ell)}\int_0^1 \mathcal{P}(x,y)\dot{T}(y)\mathcal{E}(x,y)\ dx\ dy$$ be our objective.
In our first specific objective (optimising response of expectations) we will insert $\mathcal{E}(x,y)=((\text{Id}-L_0^*)^{-1}c)(x)f_0(y)$ and in our second specific objective (optimising mixing) we will insert $\mathcal{E}(x,y)=E(x,y)$ from (\ref{eq:almost-soln-2nd-evalue}).

\begin{lemma}
\label{lem:nonzeroobjmap}
Assume that there is $F'\subset \widetilde{F}_\ell$ such that $m(F')>0$ and $\mathcal{E}(\cdot,y)\notin \spn\{\mathcal{P}(\cdot,y)\}^\perp$ for all $y\in F'$.
 Then there is a $\dot{T}\in P_\ell$ such that $\mathcal{J}(\dot{T})>0$.
\end{lemma}
\begin{proof}
Because $$\mathcal{J}(\dot{T})=\int_{\Xi(\widetilde{F}_\ell)}\dot{T}(y)\left(\int_0^1\mathcal{P}(x,y)\mathcal{E}(x,y)\ dx\right)\ dy,$$ we may set $\dot{T}(y)=\int_0^1\mathcal{P}(x,y)\mathcal{E}(x,y)\ dx$ for $y\in F'$ and $\dot{T}(y)=0$ otherwise to obtain $\mathcal{J}(\dot{T})>0$. Trivial scaling yields $\dot{T}\in B_1$.
\end{proof}
We expect the hypotheses of Lemma \ref{lem:nonzeroobjmap} to be satisfied ``generically''.

\subsection{Explicit formula for the optimal map perturbation that maximally increases the expectation
of an observable}
\label{sssec:for-opt-pert-obs-add-nos}

In this section we consider the problem of finding 
the optimal map perturbation that maximizes the expectation of some observable $c\in L^2$.
We first present a result that ensures a unique solution exists and then derive an explicit
expression for the optimal perturbation. 

We begin by noting that $\widehat{R}(\dot{T})\in V$; this follows from the fact that
$\left(P_\pi\left(\tau_{-T_0(y)}\frac{d\rho}{dx}\right) \right)(x)f_0(y)\in V_{\ker}$
(since $\k\in V_{\ker}$, see Proposition \ref{prop:lin-resp-formula-add-noise})
and therefore $\int_0^1 \left(P_\pi\left(\tau_{-T_0(y)}\frac{d\rho}{dx}\right) \right)(x)f_0(y) g(y)dy\in V$ for $g\in L^2$ (see Lemma \ref{lem:charac-mean-0}).
Hence, we only need to consider $c\in$ span$\{f_0\}^\perp$
(see the discussion at the end of Section \ref{subsubsec1}).

\begin{proposition}
\label{solutionT}
Let $c\in$ span$\{f_0\}^\perp$ and $P_{\ell}$ be the set in \eqref{eq:allow-map-pert}. 
Assume that the function $\mathcal{J}(\dot{T}):=\big\langle c,\widehat{R}(\dot{T})\big\rangle_{L^{2}([0,1],\bR)}$ is not uniformly vanishing on $P_\ell$.
Then the optimisation problem
\begin{equation}  \label{P1-add}
\big\langle c,\widehat{R}(\dot{T})\big\rangle_{L^{2}([0,1],\bR)}=\max_{\dot{h}%
\in P_{\ell}}\big\langle c,\widehat{R}(\dot{h})\big\rangle_{L^{2}([0,1],\bR)},
\end{equation}
where $\widehat{R}$ is as in \eqref{Linreap2}, has a unique solution $\dot{T}\in L^2$. 
\end{proposition}

\begin{proof}
Let $\mathcal{H} = L^2$, $P=P_\ell$
and $\mathcal{J}(\dot{h}) = \langle c,\widehat{R}(\dot{h})\rangle_{L^2([0,1],\bR)}$.
Using Lemma \ref{Lem7.1} we note that $P_\ell$ is closed, as well as bounded, strictly convex and that it contains the zero element of $\mathcal{H}$.
From Lemma \ref{lem:cont-resp-ref-bnd}, it follows that $\langle c,\widehat{R}(\dot{h}%
)\rangle_{L^2([0,1],\bR)}$ is continuous as a function of $\dot{h}$; note that it is also linear in $\dot{h}$.
By hypothesis, $\mathcal{J}$ is not uniformly vanishing on $P_\ell$.
We can therefore apply  Propositions \ref{prop:exist} and \ref{prop:uniqe} to conclude that 
\eqref{P1-add} has a unique solution. 
\end{proof}

Before we present the explicit formula for the optimal solution, we will 
reformulate the optimization problem \eqref{P1-add} to simplify the analysis. 
We first note that since the objective
function in \eqref{P1-add} is linear in $\dot{T}$, the maximum will occur
on $S_{T_0,\ell}\cap \partial B_1$. Combining this with the fact that we only need 
$c\in$ span$\{f_0\}^\perp$, we consider the following reformulation of \eqref{P1-add}:
\begin{problem} Given $\ell \geq0 $ and $c\in$ span$\{f_0\}^\perp$ solve\label{problem7.4}
\begin{eqnarray}
\min_{\dot{T}\in S_{T_0,\ell}} &&-\big\langle c,\widehat{R}(\dot{T})\big\rangle_{L^2([0,1],\bR)}
\label{obj_lin_fun-min2} \\
\mbox{subject to} &&\Vert \dot{T}\Vert^2_2-1=0.
\label{constraint-bound2}
\end{eqnarray}
\end{problem}

\begin{theorem}

\label{thm:explicit-formula2} 
Suppose the transfer operator $L_0$ associated with the system $(T_0,\rho)$ has a kernel $k_0$ as in \eqref{eq:ker-add-noise-explicit}, which satisfies $(A1)$ of Theorem \ref{th:linearresponse copy(1)}, and there is a $F'\subset \widetilde{F}_\ell$ such that $m(F')>0$, and $f_0(y)>0$ and $(\text{Id}-L_0^*)^{-1}c\notin \spn\{\mathcal{P}(\cdot,y)\}^\perp$ for all $y\in F'$.
Let $\mathcal{G}:L^2\rightarrow L^2$ be defined as
\begin{equation}  \label{eq:def-op-G}
\mathcal{G}f(y) := \int_0^1 \left(P_\pi\left(\tau_{-T_0(y)}\frac{d\rho}{dx}\right) \right)(x)f(x)dx.
\end{equation}
Then, the unique solution to Problem \ref{problem7.4} is
\begin{equation}  \label{opt-soln2}
\dot{T} (y) =  \begin{cases}
      -f_0(y)\mathcal{G}((\text{Id}-L_0^*)^{-1}c)(y)/\|f_0\mathcal{G}((\text{Id}-L_0^*)^{-1}c)\mathbf{1}_{\widetilde{F}_\ell}\|_2  & y\in \widetilde{F}_\ell, \\
      0 & \text{otherwise}.
   \end{cases}
\end{equation}
Furthermore, $\dot{T}\in L^\infty$.
\end{theorem}
\begin{proof}
See Appendix \ref{apx:thm:explicit-formula2}.
\end{proof}

\subsection{Explicit formula for the optimal map  perturbation that maximally increases the mixing rate}
\label{sec:expliciteigmap}

In this section we set up the optimisation problem for mixing enhancement and derive a formula for the optimal map perturbation.
We remark that related spectral approaches to mixing enhancement for continuous-time flows were developed in \cite{FS17,FKP}.

Recall that to enhance mixing in Section \ref{seceig}, we 
perturbed $k_0$ so that the logarithm of the real part of the second eigenvalue decreases. 
From Lemma \ref{lem:2-eval-conj-uncess}, we have 
\begin{equation}
\label{loglameqn}
\frac{d}{d\delta}\Re(\log \lambda_\delta)\bigg|_{\delta =0}= \frac{\langle%
\dot{k}, E\rangle_{L^2([0,1]^2,\bR)}}{|\lambda_{0}|^2},
\end{equation}
where $\lambda_\delta$ denotes the second largest
eigenvalue in magnitude (assumed to be simple) of the integral operator $L_\delta$ with the kernel 
$k_\delta = k_0+\delta\cdot \dot{k} + o(\delta)$, where $\delta\mapsto k_\delta$ is $C^1$ at $\delta=0$.
Since we want to perturb $T_0$ by $\dot{T}$, we  reformulate the above inner product.
Define
\begin{equation}\label{eq:pre-opt-sol-eval}
\widehat{E}(y)= 
-\int_0^1 \left(P_\pi\left(\tau_{-T_0(y)}\frac{d\rho}{dx}\right) \right)(x)E(x,y) dx,
\end{equation}
where $E(x,y)$ is as in (\ref{eq:almost-soln-2nd-evalue}).
\begin{proposition}\label{prop:add-noise-eval-obj-mix}
Let $L_\delta:L^2([0,1],\bC)\rightarrow L^2([0,1],\bC)$, $\delta\in [0,\bar{\delta})$,
be integral operators generated by the
kernels $k_\delta $ as in \eqref{eq:ker-add-noise-explicit},  assume that $d\rho/dx$ is Lipschitz and $\delta\mapsto T_\delta$ is $C^1$.
Let $\lambda_{\delta}$ be an eigenvalue of $L_\delta$ with second
largest magnitude strictly inside the unit disk.
Suppose that $L_{0}$ satisfies $(A1)$ of Theorem 
\ref{th:linearresponse copy(1)}
and $\lambda_{0}$ is geometrically simple.
Let $e$ and $\hat{e}$ be the eigenvectors of $L_0$ and $L_0^*$,
respectively, corresponding to the eigenvalue $\lambda_{0}$.
Then
$\widehat{E}\in L^\infty([0,1],\bR)$ and
\begin{equation*}
\big\langle\dot{k},E\big\rangle_{L^2([0,1]^2,\bR)} = \big\langle\dot{T},
\widehat{E}\big\rangle_{L^2([0,1],\bR)}.
\end{equation*}
\end{proposition}
\begin{proof}
We first show that $\widehat{E}\in L^\infty([0,1],\bR)$. We can write
\begin{equation*}
    \begin{aligned}
    -\int_0^1 \left(P_\pi\left(\tau_{-T_0(y)}\frac{d\rho}{dx}\right) \right)(x)E(x,y) dx  &= -\sum_{i=1}^4\beta_ih_i(y)\int_0^1 \left(P_\pi\left(\tau_{-T_0(y)}\frac{d\rho}{dx}\right) \right)(x)g_i(x)dx\\
    &= -\sum_{i=1}^4\beta_ih_i(y)(\mathcal{G}g_i)(y),
    \end{aligned}
\end{equation*}
where $\beta_1=\beta_2 = \Re(\lambda_{0})$, $\beta_3=-\beta_4 =
\Im(\lambda_{0}), g_1=g_4 = \Re(\hat{e}), g_2=g_3=\Im(\hat{e})$, $h_1=h_3 =
\Re(e),h_2=h_4 = \Im(e)$.
From the proof of Theorem \ref{thm:explicit-formula2}, we have $\mathcal{G}g_i \in L^\infty([0,1],\bR)$.
Also, from Lemma \ref{lem:prop-add-noise-ker}, we have that $k_0\in L^\infty([0,1]^2)$ and therefore
$h_i\in L^\infty([0,1],\bR)$; thus, $\widehat{E}\in L^\infty([0,1],\bR)$.

Finally, we compute
\begin{eqnarray*}
\langle\dot{k},E\rangle_{L^2([0,1]^2,\bR)}&=&\int_0^1\int_0^1 \dot{k}(x,y)E(x,y) dx dy\\ 
&=&-\int_0^1\int_0^1 \left(P_\pi\left(\tau_{-T_0(y)}\frac{d\rho}{dx}\right) \right)(x)\dot{T}(y)E(x,y)dxdy\\ 
&=&\int_0^1 \dot{T}(y)
\widehat{E}(y) dy = \big\langle\dot{T}, \widehat{E}\big\rangle_{L^2([0,1],\bR)}.
\end{eqnarray*}
\end{proof}


From equation (\ref{loglameqn}), in order to maximally increase the spectral gap, by Proposition \ref{prop:add-noise-eval-obj-mix}, we should choose the map perturbation $\dot{T}$ to minimise $\langle \dot{T},\widehat{E}\rangle$.
We first show this optimisation problem has a unique solution.
\begin{proposition}
\label{solutionT2}
Let $P_\ell$ be the set in \eqref{eq:allow-map-pert} and assume that $\mathcal{J}(\dot{T})=\langle \dot{T},\widehat{E}\rangle$ does not uniformly vanish on $P_\ell$.
Then, 
the problem of finding $\dot{T}\in P_\ell$ such that
\begin{equation}
\big\langle\dot{T}, \widehat{E}\big\rangle_{L^2([0,1],\bR)}=\min_{\dot{h}\in P_\ell}%
\big\langle \dot{h},\widehat{E}\big\rangle_{L^2([0,1],\bR)} \label{P2-add}
\end{equation}
has a unique solution.
\end{proposition}

\begin{proof}
Note that $P_\ell$ is closed (by Lemma \ref{Lem7.1}), bounded, strictly convex and contains the zero element of $L^2$.
Now, since $\mathcal{J}(\dot{h}) := \langle \dot{h},\widehat{E}\rangle_{L^2([0,1],\bR)}$ is linear and
continuous and by hypothesis does not vanish everywhere on $P_\ell$, we may apply 
Propositions \ref{prop:exist} and \ref{prop:uniqe} to obtain the result.
\end{proof}


Since the objective function in %
\eqref{P2-add} is linear, all optima will lie in $S_{T_0,\ell}\cap\partial B_1$.
Hence, we equivalently consider the following optimization problem:
\begin{problem} \label{prob7.8} Given $\ell \geq 0$, solve
\begin{eqnarray}
\min_{\dot{T}\in S_{T_0,\ell}} &&\big\langle\dot{T}, \widehat{E}\big\rangle_{L^2([0,1],\bR)}\label{obj_lin_fun-min-2nd-eval-add} \\
\mbox{such that} &&\Vert \dot{T}\Vert _{2}^{2}-1= 0.
\label{constraint-bound-2nd-eval-add}
\end{eqnarray}

\end{problem}

We now state a formula for the unique optimum.
\begin{theorem}
\label{extremiser2}
Let $(T_0,\rho)$
be a deterministic system with additive noise satisfying $(T1)$ and $(T2)$.
Suppose the associated transfer operator $L_0:L^2([0,1],\bC)\rightarrow L^2([0,1],\bC)$,
with the kernel $k_0$ as in \eqref{eq:ker-add-noise-explicit}, satisfies $(A1)$ of Theorem \ref{th:linearresponse copy(1)}, and that there is a $F'\subset \tilde{F}_\ell$ with $m(F')>0$ and $E(\cdot,y)\notin\spn\{\mathcal{P}(\cdot,y)\}^\perp$ for all $y\in F'$.
Suppose $\lambda_0$
is geometrically simple. 
Then, the unique solution to the optimization problem \ref{prob7.8} is
\begin{equation}  \label{eq:soln-2nd-eval-add}
\dot{T}(y) = \begin{cases}
\frac{1}{\alpha}\int_0^1 \left(P_\pi\left(\tau_{-T_0(y)}\frac{d\rho}{dx}\right) \right)(x)E(x,y) dx & y\in \widetilde{F}_\ell,\\
0 & \text{otherwise},
\end{cases}
\end{equation}
where $E(x,y)$ is as in (\ref{eq:almost-soln-2nd-evalue})
and $\alpha>0$ 
is selected so that $\|\dot{T}\|_2 =1$. 
Furthermore, $\dot{T}\in L^\infty$.
\end{theorem}
\begin{proof}
See Appendix \ref{apx:extremiser2}.
\end{proof}

\begin{corollary}
If $\lambda_{0}$ is real, then
\begin{equation*}
\dot{T}(y) = \begin{cases}
\text{sgn}(\lambda_0)\frac{e(y)(\mathcal{G}\hat{e})(y)}{\|e\mathcal{G}\hat{e}\mathbf{1}_{\widetilde{F}_\ell}\|_2} & y\in \widetilde{F}_\ell,\\
0 & \text{otherwise},
\end{cases}
\end{equation*}
where $\mathcal{G}$ is the operator in \eqref{eq:def-op-G}.
Furthermore, if there exists
an $\ell>0$ such that $\ell\le T_0(x)\le 1-\ell$ for $x\in[0,1]$, then
\begin{equation}\label{eq:simp-opt-soln-mix-add-noise}
\dot{T} = \text{sgn}(\lambda_0)\frac{e\cdot \mathcal{G}\hat{e}}{\|e\cdot \mathcal{G}\hat{e}\|_2}.
\end{equation}
\end{corollary}
\begin{proof}
Since $e, \hat{e}$ and $\lambda_0$ are real, we have $E(x,y) = \hat{e}(x)e(y)\lambda_0$
and
the expression for $\dot{T}$ follows from \eqref{eq:soln-2nd-eval-add}.
Finally, if $\ell\le T_0(x)\le 1-\ell$,
then $\widetilde{F}_\ell = [0,1]$ and we have \eqref{eq:simp-opt-soln-mix-add-noise}. 
\end{proof}

\section{Applications and numerical experiments}

\label{sec:numerics}

In this section we will consider two stochastically perturbed deterministic systems, namely the Pomeau-Manneville map and a weakly mixing interval exchange map. 
For each of these maps we numerically estimate:
\begin{enumerate}
    \item The unique kernel perturbation that maximises the change in expectation of a prescribed observation function (see Problem \ref{prob3}).  An expression for this optimal kernel is given by (\ref{opt-soln}).
    \item The unique kernel perturbation that maximally increases the mixing rate (see Problem \ref{prob3.1}).
    An expression for this optimal kernel is given by (\ref{eq:soln-2nd-eval}) and (\ref{optkernelevaleqn}).
    \item The unique map perturbation that maximises the change in expectation of a prescribed observation function (see Problem \ref{problem7.4}). An expression for this optimal map perturbation is given by (\ref{opt-soln2}).
    \item The unique map perturbation that maximally increases the mixing rate (see Problem \ref{prob7.8}). An expression for this optimal map perturbation is given by (\ref{eq:soln-2nd-eval-add}) and (\ref{eq:simp-opt-soln-mix-add-noise}).
\end{enumerate}
The numerics will be explained as we proceed through these four optimisation problems.
We refer the reader to \cite{ADF} for additional details on the implementation and related experiments.


\subsection{Pomeau-Manneville map}
\label{sec:PM}
We consider the Pomeau-Manneville map \cite{LSV} 
\begin{equation}
    \label{PMmap}
T_0(x)=\left\{
\begin{array}{ll}
x(1+(2x)^\alpha),&\hbox{$x\in[0,1/2)$;}\\
2x-1,&\hbox{$x\in[1/2,1]$}
\end{array}
    \right.,
\end{equation}
with parameter value $\alpha=1/2$.
For this parameter choice it is known that the map $T_0$ admits a unique absolutely continuous invariant probability measure, but only algebraic decay of correlations \cite{LSV}.
With the addition of noise as per (\ref{systm}), the transfer operator defined by (\ref{kernelL2}) and (\ref{eq:ker-add-noise-explicit}) for $\delta=0$ becomes compact as an operator on $L^2$. 
In our numerical experiments we will use the smooth noise kernel $\rho_\epsilon:[-\epsilon,\epsilon]\to\mathbb{R}$, defined by $\rho_\epsilon(x)=N(\epsilon)\exp(-\epsilon^2/(\epsilon^2-x^2))$, where $N(\epsilon)$ is a normalisation factor ensuring $\int \rho_\epsilon(x)\ dx=1$.

We now begin to set up our numerical procedure for estimating $L_0$, which is a standard application of Ulam's method \cite{ulam}.
Let $B_n = \{I_1,\dots, I_n\}$ denote an equipartition of $[0,1]$ into $n$ subintervals, and set $\mathcal{B}_n = $ span$\{\mathbf{1}%
_{I_1},\dots,\mathbf{1}_{I_n}\}$. 
We define the (Ulam) projection $\pi_n:L^2([0,1])
\rightarrow\mathcal{B}_n$ by $\pi_n(g) = \sum_{i=1}^n\left(\frac{1}{m(I_i)}
\int_{I_i}g(x)dx\right) \mathbf{1}_{I_i}$. 
The finite-rank
transfer operator $L_{n}:=\pi_n L_0:L^2([0,1])\rightarrow \mathcal{B}_n$ can be computed numerically.
We use MATLAB's built-in functions \verb"integral.m" and  \verb"integral2.m" to perform the $\rho$-convolution (using an explicit form of $\rho_\epsilon$) and the Ulam projections, respectively.
Figure \ref{fig:PMtranspt1} displays the nonzero entries in the column-stochastic matrix corresponding to $L_n$ for $\epsilon=0.1$. 
\begin{figure}[h!]
    \begin{center}
        \includegraphics[width=0.5\textwidth]{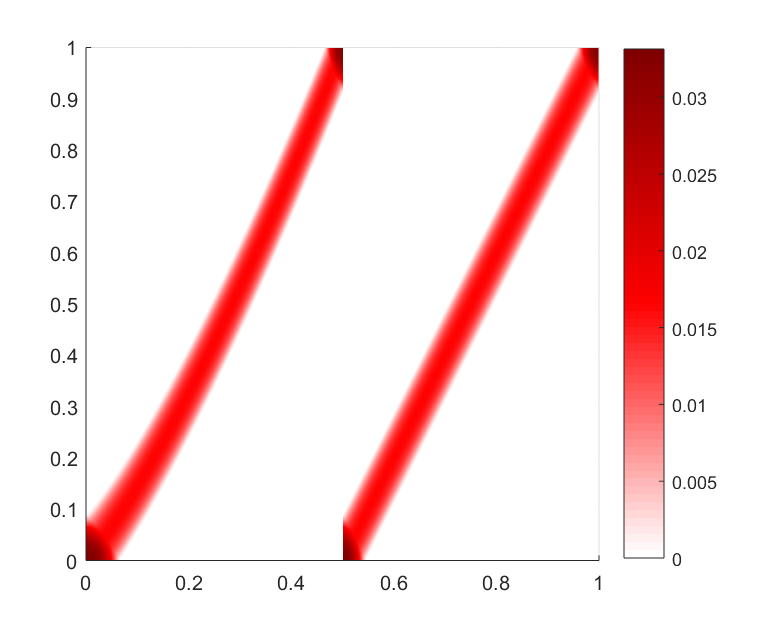}
         \caption{Transition matrix $L_n$ for the system (\ref{systm}) generated by the Pomeau-Manneville map $T_0$ (\ref{PMmap}) using $n=500$ subintervals of equal length. The matrix entries are located according to the subinterval positions in the domain $[0,1]$, so that the image appears as a ``blurred'' version of the graph of $T_0$.  The additive noise in (\ref{systm}) is drawn according to $\rho_\epsilon$ with $\epsilon=1/10$.}
        \label{fig:PMtranspt1}
        \end{center}
        \end{figure}
        
Approximations to the invariant probability densities for our stochastic dynamics are displayed in Figure \ref{fig:PMdens} (left) for large and small noise supports.  
  \begin{figure}[h!]
    \begin{center}
        \includegraphics[width=\textwidth]{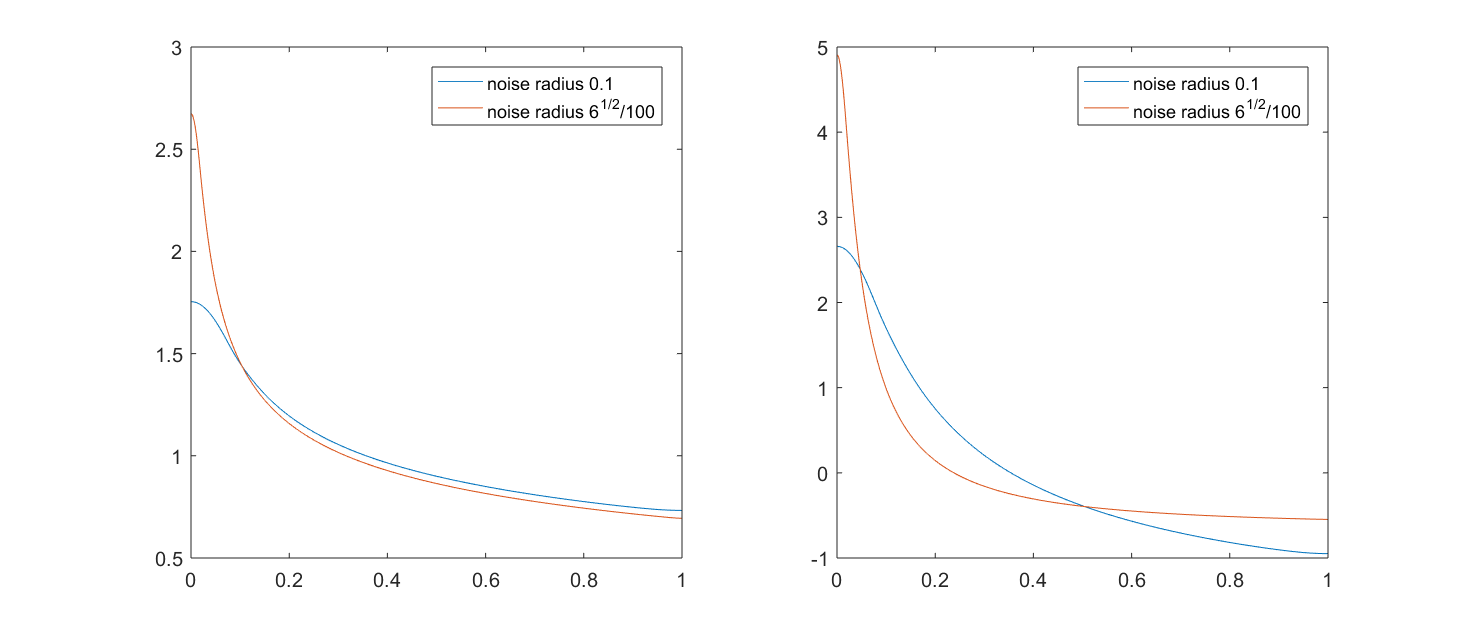}
         \caption{Approximate invariant densities (left) and eigenfunctions corresponding to the 2nd largest eigenvalue of $L_0$ (right) for the  system (\ref{systm}) with $T_0$ given by the Pomeau-Manneville map (\ref{PMmap}).  The additive noise in (\ref{systm}) is drawn according to $\rho_\epsilon$ with $\epsilon$ taking the values 1/10 (blue) and $\sqrt{6}/100$ (red). The Ulam matrix $L_n$ is constructed with 500 subintervals.}
        \label{fig:PMdens}
        \end{center}
\end{figure}
A lower level of noise permits greater concentration of invariant probability mass near the fixed point $x=0$ of the map $T_0$.
Also shown in Figure \ref{fig:PMdens} (right) are the estimated eigenfunctions corresponding to the second-largest eigenvalue of $L_n$.
The signs of these second eigenfunctions split the interval $[0,1]$ into left and right hand portions, broadly indicating that the slow mixing is due to positive mass near $x=0$ and negative mass away from $x=0$ \cite{DFS00};  see \cite{FMS11} for further discussion of this point in the Pomeau-Manneville setting.

\subsubsection{Kernel perturbations}
\label{sec:kerper}
In the framework of Problems \ref{prob3} and  \ref{prob3.1} we use the (arbitrarily chosen) monotonically increasing observation function $c(x)=-\cos(x)$.
In order to estimate $\dot{k}$ as in (\ref{opt-soln}) we use the code from Algorithm 3 \cite{ADF};
the inputs are the Ulam matrix $L_n$ and $c_n$ (obtained as $\pi_n(c)$).
Equivalently, directly using (\ref{opt-soln}) one may substitute $f_n$ (obtained as the leading eigenvector of $L_n$) for $f$, $L_n$ for $L$, $c_n$ as above for $c$, and solve $(Id-L_n^*)^{-1}c_n$ (obtained as a vector $y\in\mathbb{R}^n$ by numerically solving the linear system $(Id-L_n^*)y=c_n, f_n^\top y=0$).
Figure \ref{fig:PM_kernel_obs} shows the optimal kernel perturbations $\dot{k}_n$ for $n=500$.
\begin{figure}[h!]
    \begin{center}
        \includegraphics[width=0.45\textwidth]{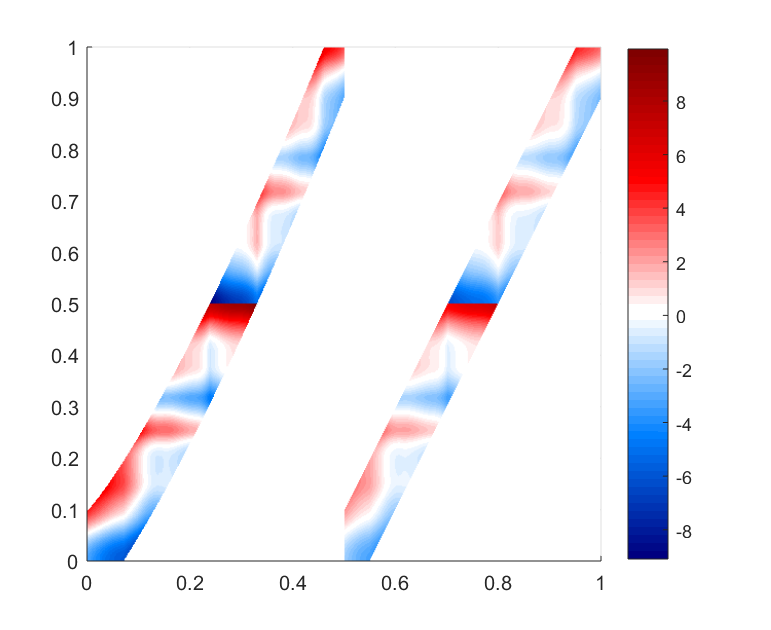}
        \includegraphics[width=0.45\textwidth]{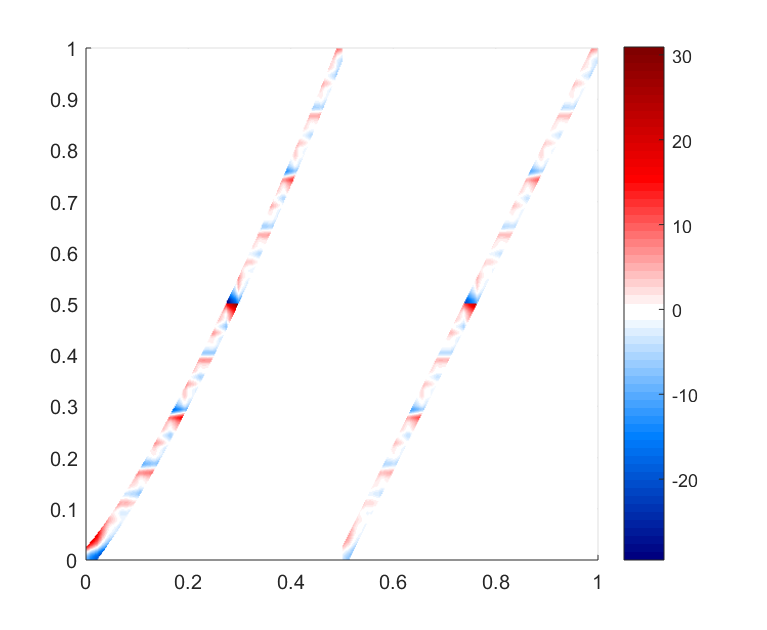}
        \caption{Optimal kernel perturbations for the Pomeau-Manneville map to maximise the change in expectation of $c(x)=-\cos(x)$, based on an Ulam approximation of (\ref{opt-soln}) with $n=500$ subintervals.  Left: $\epsilon=1/10$, Right: $\epsilon=\sqrt{6}/100$.}
        \label{fig:PM_kernel_obs}
        \end{center}
        \end{figure}
        Because $c$ is an increasing function, intuitively one might expect the kernel perturbation to try to shift mass in the invariant density from left to right.
        Broadly speaking, this is what one sees in the high-noise case in Figure \ref{fig:PM_kernel_obs} (left): vertical strips typically have red above blue, corresponding to a shift of mass to the right in $[0,1]$.
        The main exception to this is around the $y$-axis value of 1/2, where red is strongly below blue along vertical strips.
        This is because at the next iteration, these red regions will be mapped near $x=1$ and achieve the highest value of $c$, while the blue regions will be mapped near to $x=0$ with the least value of $c$.
        In the low-noise case of Figure \ref{fig:PM_kernel_obs} (right), we see a similar solution with higher spatial frequencies, and strong perturbations near the critical values of $x=0$ and $T_0(x)=1/2$.
        
To investigate the optimal kernel perturbation to maximally increase the rate of mixing in the stochastic system, we use the expression $\dot{k}$ in (\ref{eq:soln-2nd-eval}).
A natural approximate version (\ref{eq:soln-2nd-eval}) requires estimates of the left and right eigenfunctions of $L_0$ corresponding to the second largest eigenvalue $\lambda_2$;  these are obtained directly as eigenvectors of $L_n$.
Figure \ref{fig:PM-mix} shows the resulting optimal kernel perturbations, computed using the code from Algorithm 4 \cite{ADF} with input $L_n$.
 \begin{figure}[h!]
    \begin{center}
            \includegraphics[width=0.45\textwidth]{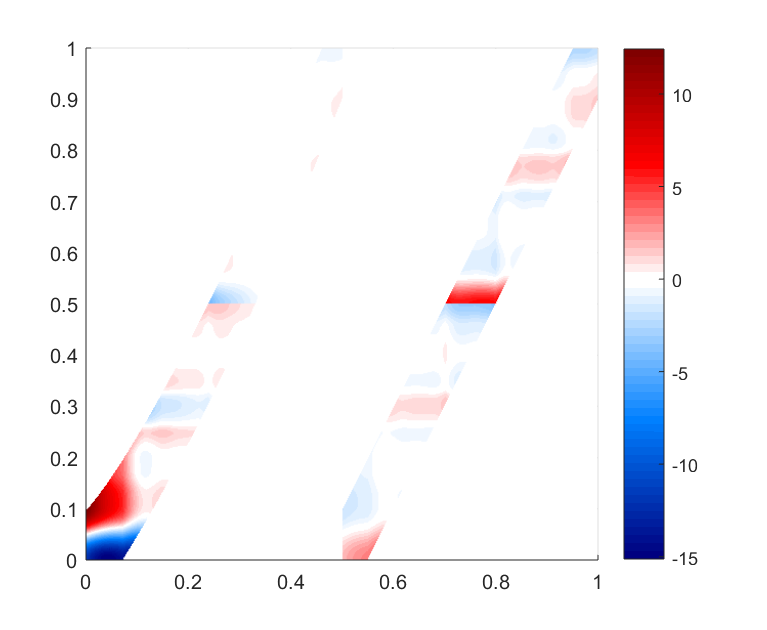}
        \includegraphics[width=0.45\textwidth]{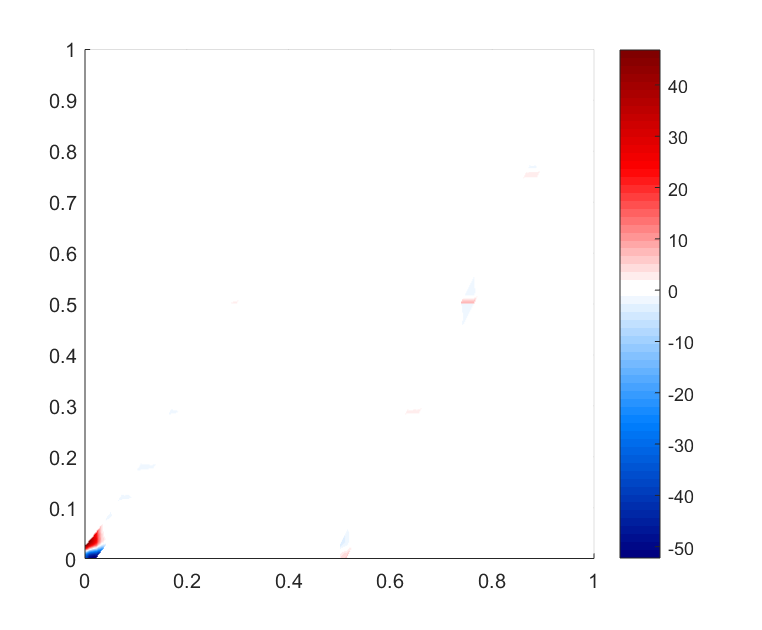}
            \caption{Optimal kernel perturbation for the Pomeau-Manneville map to maximally increase the mixing rate, computed with $n=500$ subintervals. Left: $\epsilon=1/10$, Right: $\epsilon=\sqrt{6}/100$. }
        \label{fig:PM-mix}
        \end{center}
        \end{figure}
Because the fixed point at $x=0$ is responsible for the slow algebraic decay of correlations for the deterministic dynamics of $T_0$, the fixed point will also play a dominant role in the mixing rate of the stochastic system for low to moderate levels of noise.
Indeed, Figure \ref{fig:PM-mix} shows that the optimal perturbation concentrates its effort in a neighbourhood of the fixed point, and pushes mass away from the fixed point as much as possible.
This is particularly extreme in the low noise case of Figure \ref{fig:PM-mix} (right) with the perturbation almost exclusively concentrated in a small neighbourhood of $x=0$.


\subsubsection{Map perturbations}
\label{sec:mapper}

We now turn to the problem of finding the unique map perturbation $\dot{T}$ that maximises the change in expectation of the observation $c(x)=-\cos(x)$ (see Problem \ref{problem7.4} for a precise formulation) and maximises  the speed of mixing (see Problem \ref{prob7.8}).
We use the natural Ulam discretisation of the expression\footnote{Note that since $T_0^{-1}(\{0,1\})$ is a finite set, we may take $\ell>0$ as small as we like.  In the computations we set $\ell=0$, so that $\widetilde{F}_\ell=[0,1]$ mod $m$.} (\ref{opt-soln2}). 
The objects $f_n$ and $(Id-L_n^*)^{-1}c_n$ are computed exactly as before in Section \ref{sec:kerper}.
The action of the operator $\mathcal{G}$ in (\ref{opt-soln2}) is computed using MATLAB's built-in function \verb"integral.m" using an explicit form of $d\rho_\epsilon/dx$ for $d\rho/dx$ in (\ref{opt-soln2}).

Figure \ref{fig:PM-map-obs} (left) shows the optimal $\dot{T}$ for the two noise amplitudes $\epsilon=1/10$ and $\epsilon=\sqrt{6}/100$.
 \begin{figure}[h!]
    \begin{center}
            \includegraphics[width=0.45\textwidth]{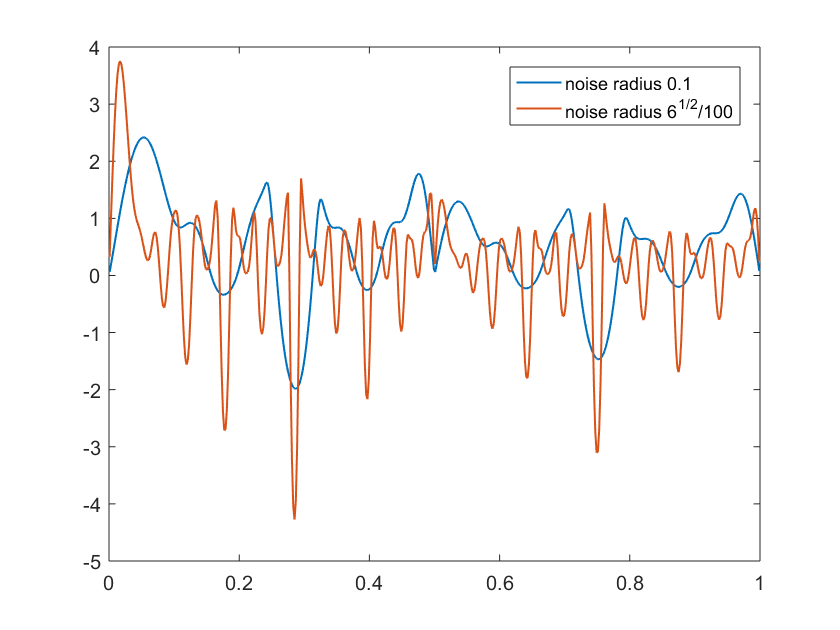}
        \includegraphics[width=0.45\textwidth]{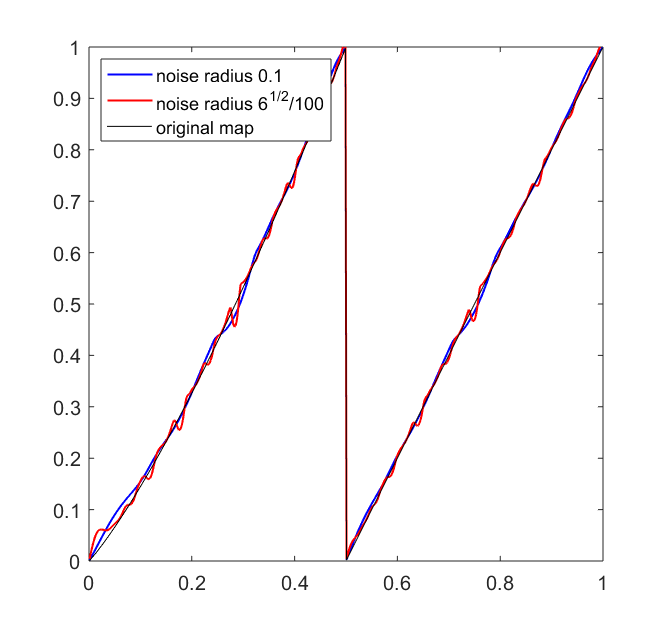}
            \caption{Left: Optimal map perturbation $\dot{T}$ for the Pomeau-Manneville map to maximise the change in expectation of $c(x)=-\cos(x)$, computed using (\ref{opt-soln2}) with $n=500$.  Right: Illustration of $T_0+\dot{T}/100$.}
        \label{fig:PM-map-obs}
        \end{center}
        \end{figure}
Note that for the noise amplitude $\epsilon=0.1$ (blue curve in Figure \ref{fig:PM-map-obs}) the map perturbation $\dot{T}$ is mostly positive, corresponding to moving probability mass to the right, as expected because we are maximising the change in expectation of an increasing observation function $c$.
The blue curve is most negative in neighbourhoods of the two preimages of $x=1/2$, corresponding to moving probability mass to the left.
The reason for this is identical to the discussion of the ``blue above red'' effect in Figure \ref{fig:PM_kernel_obs}, namely moving mass to the left creates a very large increase in the objective function value at the next iterate.
This ``look ahead'' effect is even more pronounced in the low noise case (red curve of Figure \ref{fig:PM-map-obs}), where $\dot{T}$ is mostly positive, but has deep negative perturbations at multiple preimages of $x=1/2$ reaching further into the past.

Figure \ref{fig:PM-map-obs} (right) illustrates the Pomeau-Manneville map (black) with perturbed maps $T_0+\dot{T}/100$.
We have chosen a scale factor of 1/100 for visualisation purposes;  one should keep in mind we have optimised for an infinitesimal change in the map.
Figure \ref{fig:PM_kernelequiv_obs} shows the kernel derivatives $\dot{k}$ corresponding to the optimal map derivatives $\dot{T}$ for the two noise levels.       
\begin{figure}[h!]
    \begin{center}
        \includegraphics[width=0.45\textwidth]{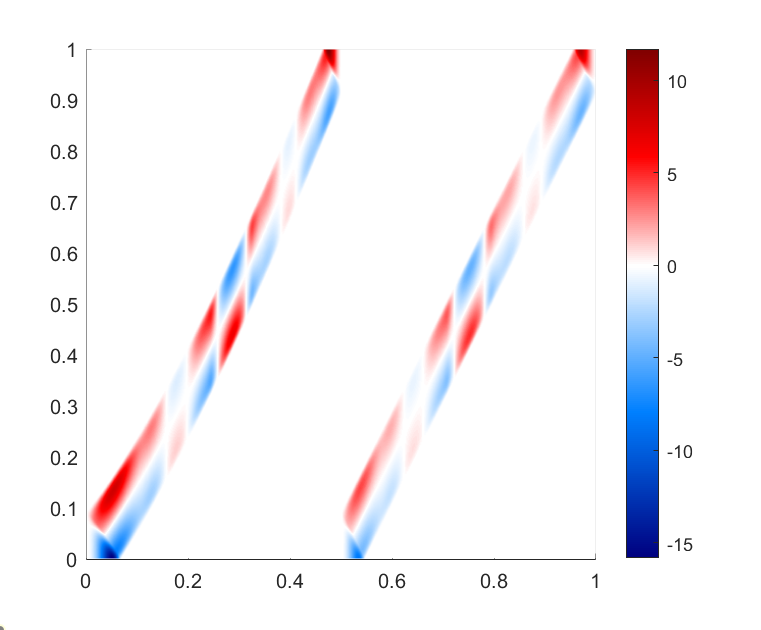}
        \includegraphics[width=0.45\textwidth]{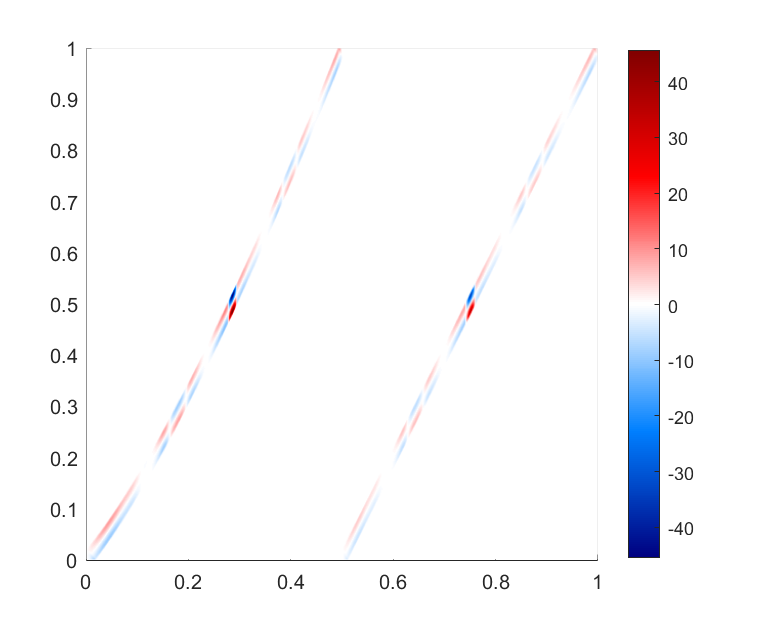}
        \caption{Kernel perturbations corresponding to the optimal map perturbations in Figure \ref{fig:PM-map-obs}. Left: $\epsilon=1/10$, Right: $\epsilon=\sqrt{6}/100$.}
        \label{fig:PM_kernelequiv_obs}
        \end{center}
        \end{figure}
These kernel derivatives have a restricted form because they arise purely from a derivative in the map.
One may compare Figure \ref{fig:PM_kernelequiv_obs} with Figure \ref{fig:PM_kernel_obs}  and note that the kernel derivative in Figure \ref{fig:PM_kernelequiv_obs} (left) attempts to follow the general structure of the kernel derivative in Figure \ref{fig:PM_kernel_obs} (left), while obeying its structural restrictions arising from the less flexible \textit{map} perturbation.
Broadly speaking, in Figure \ref{fig:PM_kernelequiv_obs} (left), red lies above blue (mass is shifted to the right). 
Exceptions are near $y=1/2$ because at the next iteration these red points will land near $x=1$, achieving very high objective value, while the blue region will get mapped to near $x=0$, encountering the lowest value of $c$.
Note that the perturbation decreases from a peak to very close to zero near $x=0$. 
This is because in a small neighbourhood of $x=0$ there is already some stochastic perturbation away from $x=0$ ``for free'' due to the reflecting boundary conditions.
Thus, the map perturbation $\dot{T}$ does not need to invest energy in large perturbations very close to $x=0$.

The map perturbation that maximally increases the rate of mixing is a particularly interesting question.
Our computations use the natural Ulam discretisation of (\ref{eq:simp-opt-soln-mix-add-noise}).
The computations follow as in Section \ref{sec:kerper} with the action of $\mathcal{G}$ computed as above.
Figure \ref{fig:PM-map-mix} (left) shows the optimal $\dot{T}$ for the two noise amplitudes $\epsilon=1/10$ and $\epsilon=\sqrt{6}/100$.
  \begin{figure}[h!]
    \begin{center}
            \includegraphics[width=0.45\textwidth]{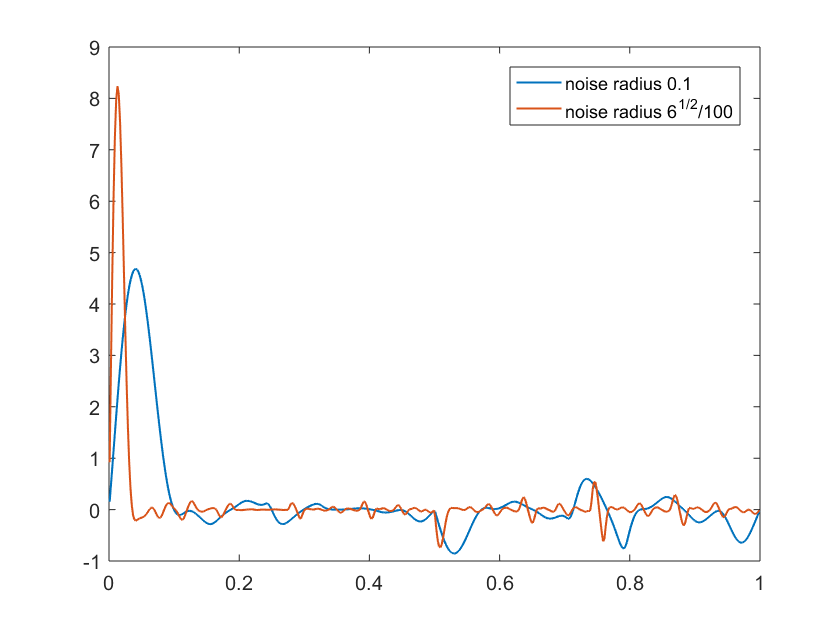}
        \includegraphics[width=0.45\textwidth]{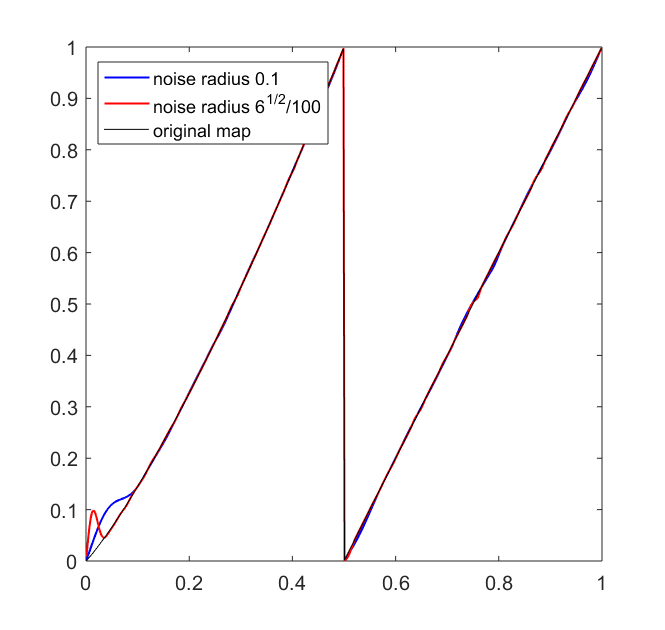}
            \caption{Left: Optimal map perturbation $\dot{T}$ for the Pomeau-Manneville map to maximise the change in the mixing rate, computed using (\ref{eq:simp-opt-soln-mix-add-noise}) with $n=500$.  Right: Illustration of $T_0+\dot{T}/100$.}
        \label{fig:PM-map-mix}
        \end{center}
        \end{figure}
A sharp map perturbation away from $x=0$ is seen for both noise levels, with the perturbation sharper for the lower noise case.
In both cases, the perturbations far from $x=0$ are weak (low magnitude values of $\dot{T}$).
This result corresponds well with the results seen for the optimal kernel perturbations in Figure \ref{fig:PM-mix}, where mass was primarily moved away from $x=0$.
As in the optimal solution shown in Figure \ref{fig:PM-map-obs} (left), the optimal perturbation in Figure \ref{fig:PM-map-mix} decreases from a sharp peak down to zero near $x=0$. 
This is again because in a small neighbourhood of $x=0$ the system experiences ``free'' stochastic perturbations away from $x=0$ due to the reflecting boundary conditions, and thus the map perturbation $\dot{T}$ need not need invest energy in large perturbations very close to $x=0$.
Figure \ref{fig:PM-map-mix} (right) illustrates the Pomeau-Manneville map (black) with perturbed maps $T_0+\dot{T}/100$, where again the factor $1/100$ is just for illustrative purposes.
When inspecting the kernel derivatives $\dot{k}$ corresponding to the optimal map perturbations $\dot{T}$ in Figure \ref{fig:PM_kernelequiv_mix}, we see similar behaviour to those in Figure \ref{fig:PM-map-mix}. 
        \begin{figure}[h!]
    \begin{center}
        \includegraphics[width=0.45\textwidth]{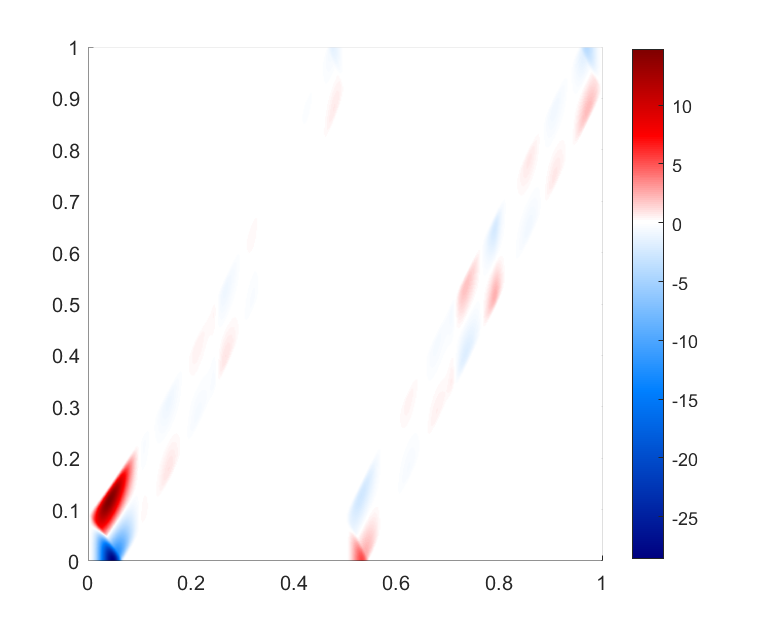}
        \includegraphics[width=0.45\textwidth]{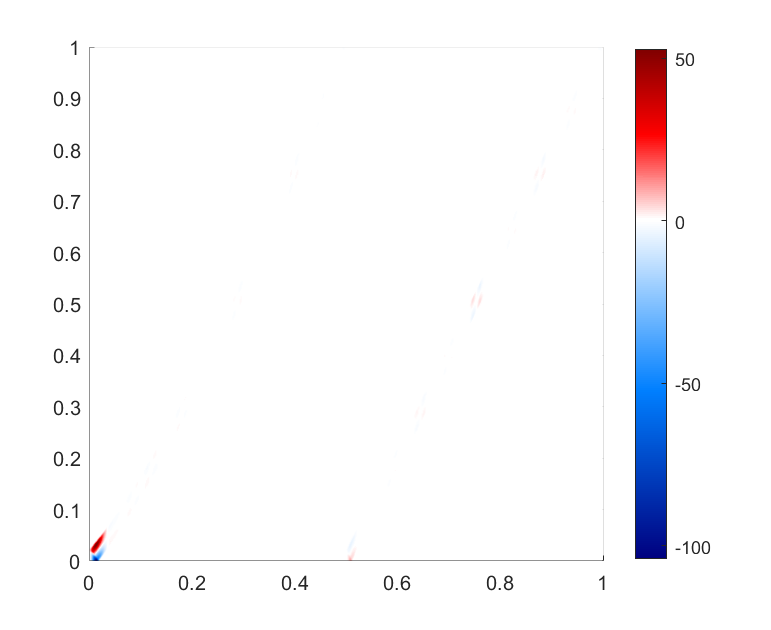}
        \caption{Kernel perturbations corresponding to the optimal map perturbations in Figure \ref{fig:PM-map-mix}. Left: $\epsilon=1/10$, Right: $\epsilon=\sqrt{6}/100$.}
        \label{fig:PM_kernelequiv_mix}
        \end{center}
        \end{figure}

\subsection{Interval exchange map}

In our second example, we consider a weak-mixing interval exchange map.
This is because of an existing literature in mixing optimisation for these classes of maps with the addition of noise.
Avila and Forni \cite{avilaforni} prove that a typical interval exchange is either weak mixing or an irrational rotation.
We use a specific weak-mixing \cite{ulcigraisinai} interval exchange map $T_0$ with interval permutation $(1234)\mapsto (4321)$ and interval lengths given by the normalised entries of the leading eigenvector of the matrix $\left(
  \begin{array}{cccc}
    13&37&77&47\\
    10&30&60&37\\
    3&10&24&14\\
    4&10&19&12
  \end{array}
\right)$;
see equation (51) in \cite{ulcigraisinai}.
We again form a stochastic system using the same noise kernels as for the Pomeau-Manneville map in Section \ref{sec:PM}.
The mixing properties of this map  have been studied in \cite{FGTW16}.
Figure \ref{fig:IEtranspt1} shows the column-stochastic matrix corresponding to $L_n$ for $n=500$ and $\epsilon=0.1$.
\begin{figure}[h!]
    \begin{center}
        \includegraphics[width=0.5\textwidth]{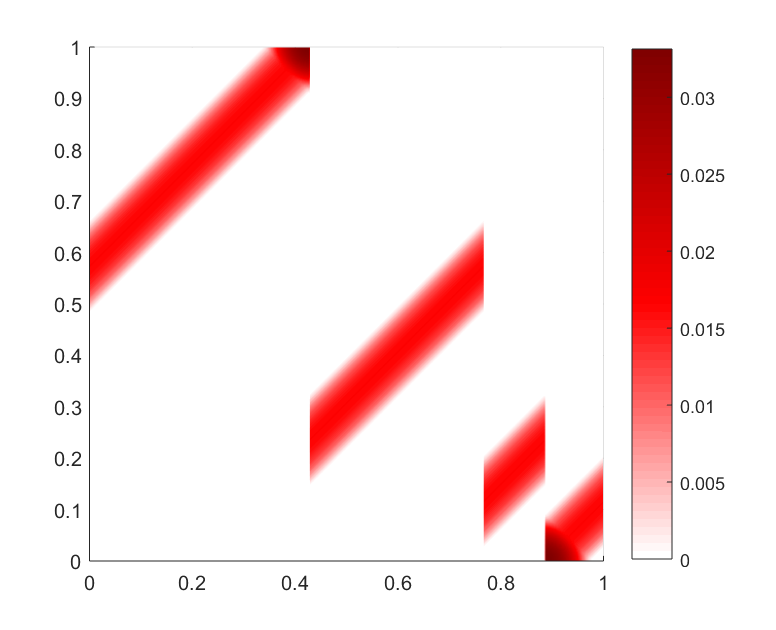}
        \caption{Transition matrix for the system (\ref{systm}) for $\delta=0$ and $T_0$ given by the interval exchange map above using $n=500$ subintervals.  The additive noise is drawn from the density $\rho_\epsilon$ with $\epsilon=1/10$.}
        \label{fig:IEtranspt1}
        \end{center}
        \end{figure}

\subsubsection{Kernel perturbations} \label{sec:kerper2}

In the framework of Problem \ref{prob3}, we use the same observation function $c(x)=-\cos(x)$ as in the Pomeau-Manneville case study, and estimate the optimal kernel perturbation $\dot{k}$ that maximally increases the expectation of $c$ in an identical fashion.
In broad terms, one again sees that $\dot{k}$ attempts to shift invariant probability mass to the right in $[0,1]$.
In Figure \ref{fig:IEobs} (left), in each smooth part of the support of $\dot{k}$, red is ``above'' blue, meaning mass is pushed to the right.
   \begin{figure}[h!]
    \begin{center}
        \includegraphics[width=0.45\textwidth]{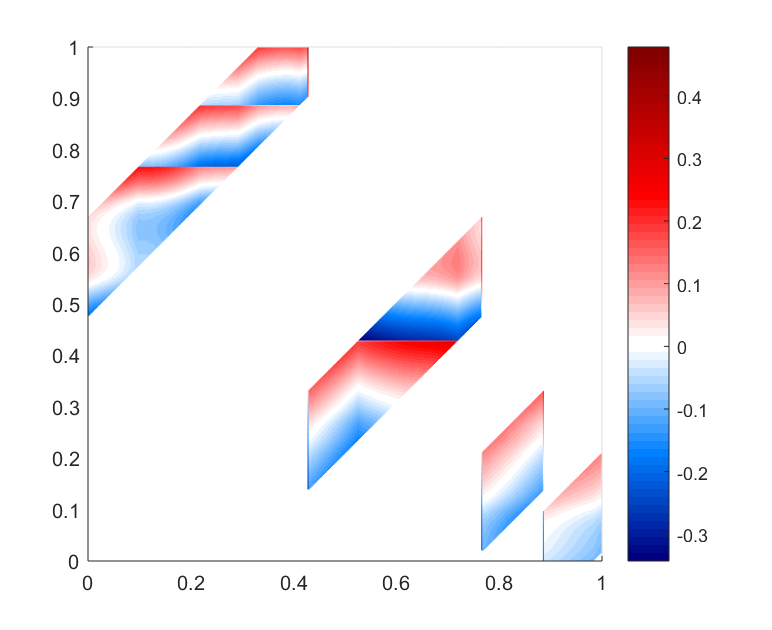}
        \includegraphics[width=0.45\textwidth]{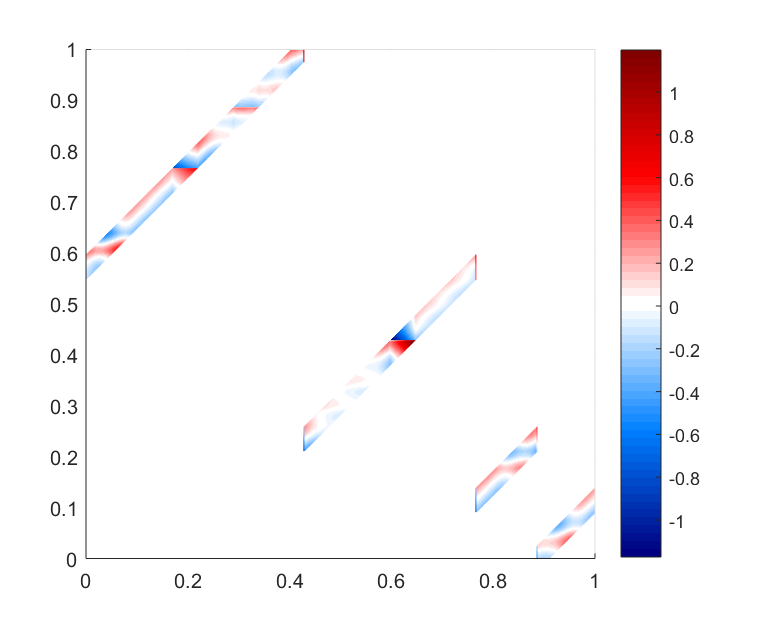}
             \caption{Optimal kernel perturbation for the interval exchange map to maximise the change in expectation of $c(x)=-\cos(x)$, computed with $n=500$ Ulam subintervals. Left: $\epsilon=1/10$, Right: $\epsilon=\sqrt{6}/100$.}
        \label{fig:IEobs}
        \end{center}
        \end{figure}
        
Clear exceptions to the ``red above blue'' scheme are seen as three sharp horizontal lines.
The $y$-coordinates of these three sharp horizontal lines coincide with the three points of discontinuity in the domain of the interval exchange at approximately $x=0.43, 0.77, 0.89$.
Consider the sharp horizontal ``blue above red'' line at $y\approx 0.43$. 
According to Figure \ref{fig:IEtranspt1}, under the action of the kernel $k_0$, mass in the vicinity of $x=0.6$ will be transported near to $x=0.43$.
The perturbation $\dot{k}$ shown in Figure \ref{fig:IEobs} will then tend to push this mass to the left of $x=0.43$.
Thus, on the next iteration there will be a bias for mass to be mapped near to $x=1$ rather than near $x=0.25$, achieving a much larger objective value at this iterate.
A similar reasoning applies to the ``blue above red'' horizontal lines at $y\approx 0.77$ and $0.89$;  the contrast is a little weaker because the potential gain at the next iterate is also weaker.
In the low noise case, Figure \ref{fig:IEobs} (right), displays similar behaviour to the higher noise case of Figure \ref{fig:IEobs} (left). With lower noise, the deterministic dynamics plays a greater role and additional preimages are taken into account, leading to a more oscillatory optimal $\dot{k}$.
      
To investigate the optimal kernel perturbation to maximally increase the rate of mixing in the stochastic system (in the framework of Problem \ref{prob3.1}) we use the expression $\dot{k}$ in (\ref{eq:soln-2nd-eval}).
The method of numerical approximation is identical to that used for the Pomeau-Manneville map.
Figure \ref{fig:IEefns} shows the signed distribution of mass that is responsible for the slowest real\footnote{In our numerical experiments the largest magnitude real eigenvalue appears as the sixth (resp.\ fourth) eigenvector of $L_{500}$ for $\epsilon=1/10$ (resp.\ $\epsilon=\sqrt{6}/100$).  Slightly larger complex eigenvalues are present, but we do not investigate these in order to make the dynamic interpretation more straightforward.} exponential rate of decay in the stochastic system. 
     \begin{figure}[h!]
        \begin{center}
        \includegraphics[width=0.45\textwidth]{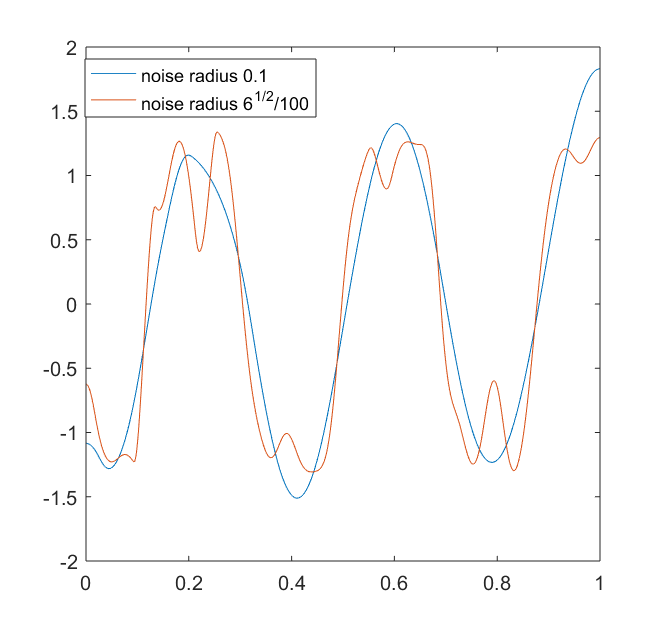}
        \caption{Approximate second eigenfunctions of the transfer operator $L_0$ of the system (\ref{systm}) with $T_0$ given by the interval exchange map above.  The additive noise in (\ref{systm}) is drawn from the density $\rho_\epsilon$ with $\epsilon$ taking the values 1/10 (blue) and $\sqrt{6}/100$ (red).}
            \label{fig:IEefns}
        \end{center}
        \end{figure}
This eigenfunction becomes more oscillatory as the level of noise decreases, and as must be the case, the magnitude of the corresponding eigenvalue increases from $\lambda\approx -0.7476$ ($\epsilon=1/10$) to $\lambda\approx -0.9574$ ($\epsilon=\sqrt{6}/100$).
Because the sign of these eigenvalues is negative, one expects a pair of almost-2-cyclic sets \cite{DJ}, consisting of three subintervals each, given by the positive and negative supports of the eigenfunctions.

Figure \ref{fig:IEmix} shows the approximate optimal kernel perturbations.
 \begin{figure}[h!]
    \begin{center}
        \includegraphics[width=0.45\textwidth]{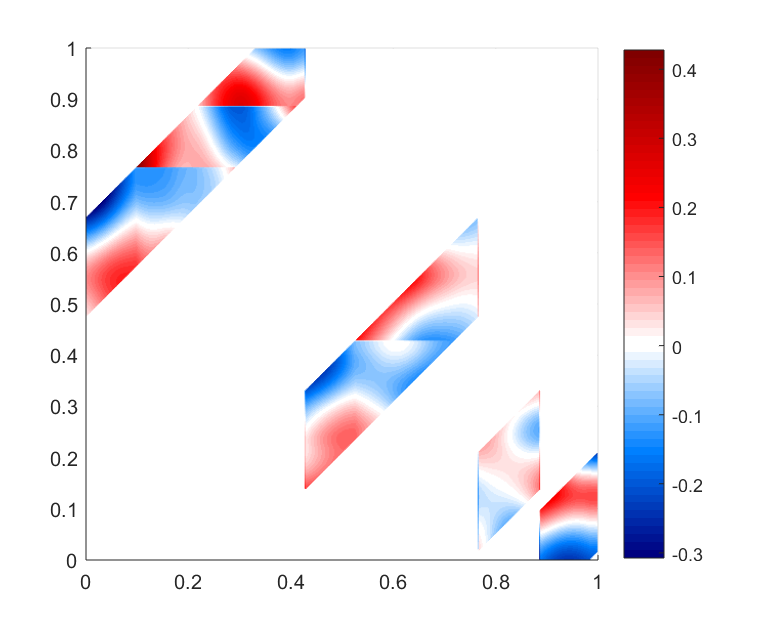}
        \includegraphics[width=0.45\textwidth]{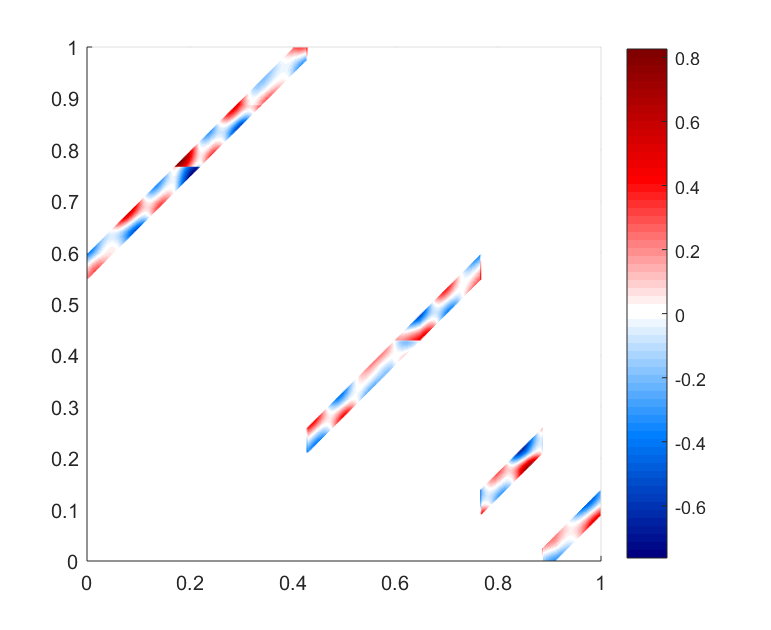}
         \caption{Optimal kernel perturbation for the interval exchange map to maximally increase the mixing rate, computed with $n=500$ Ulam subintervals. Left: $\epsilon=1/10$, Right: $\epsilon=\sqrt{6}/100$. }
        \label{fig:IEmix}
        \end{center}
        \end{figure}
In the high-noise situation of Figure \ref{fig:IEmix} (left), the sharp horizontal changes are present at preimages of the deterministic dynamics, as they were in to Figure \ref{fig:IEobs} (left).
The importance of the break points to the overall mixing rate is thus clearly borne out in the optimal $\dot{k}$;  a precise interpretation of the optimal $\dot{k}$ is not very straightforward.
For the low noise case  (Figure \ref{fig:IEmix} (right)) it appears that there is an alternating shifting of mass left and right with alternating ``red above blue'' and ``blue above red''.
This leads to greater mixing at smaller spatial scales than is possible in a single iteration of the deterministic interval exchange.
We anticipate that decreasing the noise amplitude further will result in more rapid alternation of ``red above blue'' and ``blue above red''.
As the diffusion amplitude decreases, the efficient large-scale diffusive mixing is no longer possible and so a transition is made to small-scale mixing, accessed by
increasing oscillation in the kernel.
        
\subsubsection{Map perturbations} \label{sec:mapper2}

The computations in this section follow those of Section \ref{sec:mapper}.
Figure \ref{fig:IE-map-obs} (left) shows the optimal map perturbations $\dot{T}$ at two different noise levels.
\begin{figure}[h!]
    \begin{center}
            \includegraphics[width=0.45\textwidth]{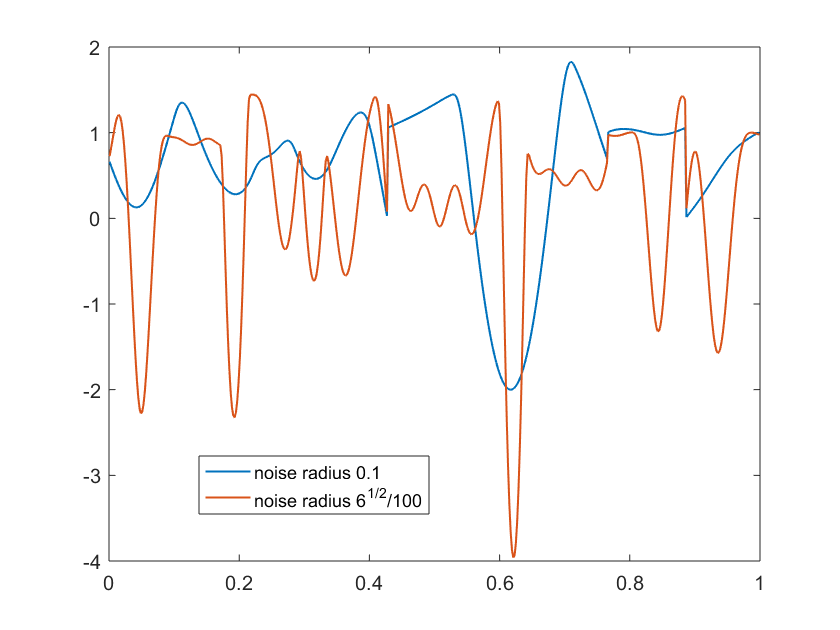}
        \includegraphics[width=0.45\textwidth]{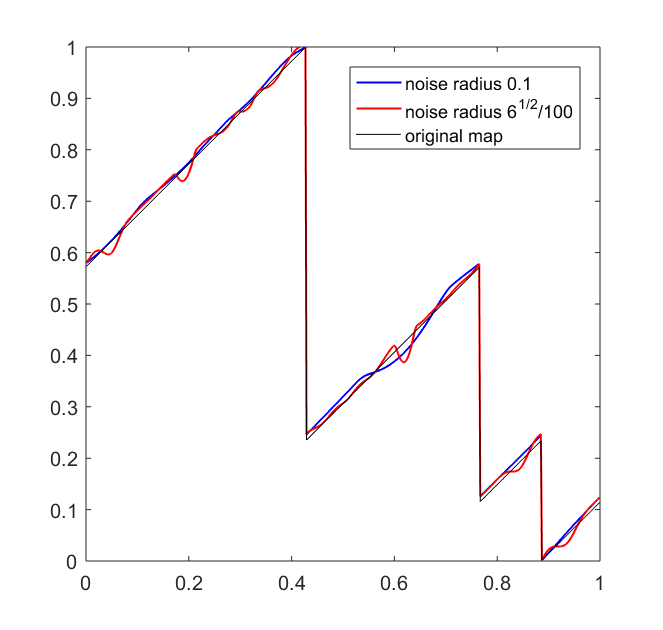}
            \caption{Left: Optimal map perturbation $\dot{T}$ for the interval exchange map to maximise the change in expectation of $c(x)=-\cos(x)$, computed using (\ref{opt-soln2}) with $n=500$.  Right: Illustration of $T_0+\dot{T}/100$.}
        \label{fig:IE-map-obs}
        \end{center}
        \end{figure}
Figure \ref{fig:IE-map-obs} (right) illustrates $T_0+\dot{T}/100$ for the two different levels of noise.
The kernel perturbations generated by these optimal map perturbations are displayed in Figure \ref{fig:IE_kernelequiv_obs}.
        \begin{figure}[h!]
    \begin{center}
        \includegraphics[width=0.45\textwidth]{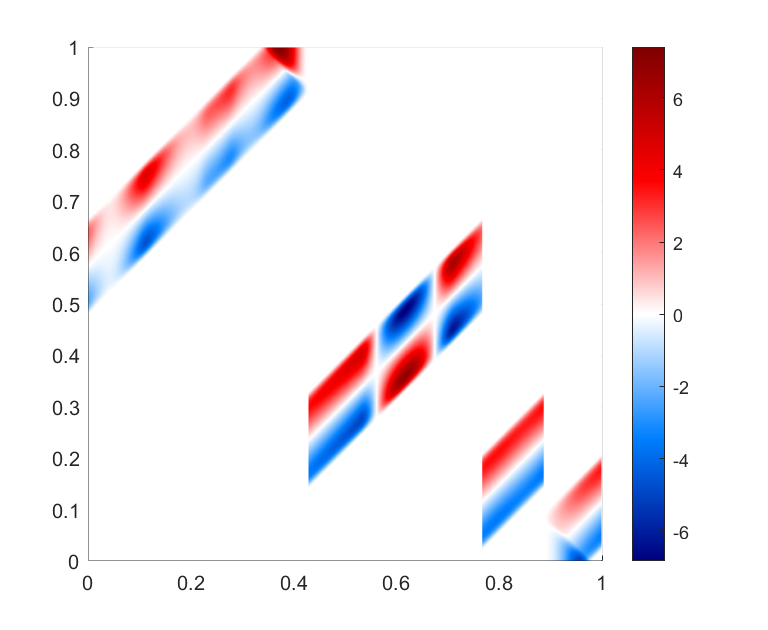}
        \includegraphics[width=0.45\textwidth]{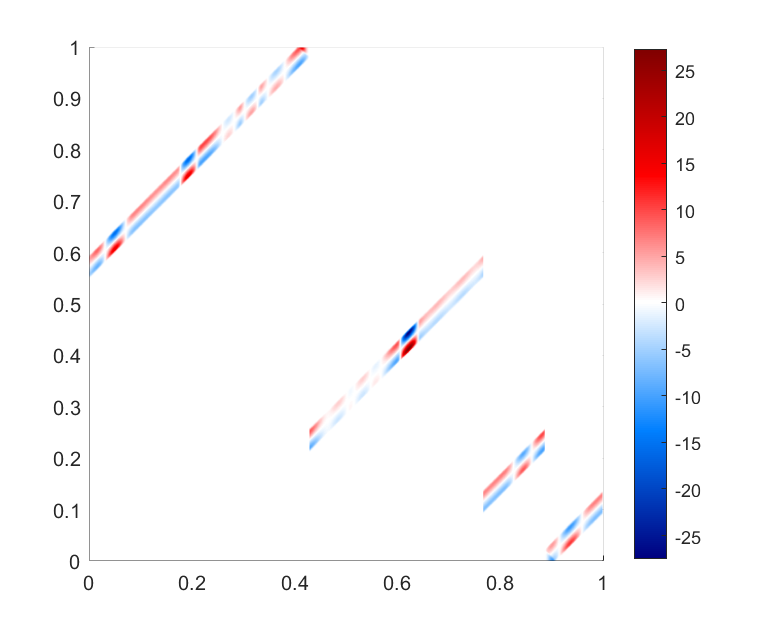}
        \caption{Kernel perturbations corresponding to the optimal map perturbations in Figure \ref{fig:IE-map-obs}. Left: $\epsilon=1/10$, Right: $\epsilon=\sqrt{6}/100$.}
        \label{fig:IE_kernelequiv_obs}
        \end{center}
        \end{figure}
If one compares the kernel perturbations in Figure \ref{fig:IE_kernelequiv_obs} with those more flexible kernel perturbations in Figure \ref{fig:IEobs}, one sees that the two sets of kernel perturbations are broadly equivalent with one another in terms of the relative positions of the positive and negative (red and blue) perturbations. Note that the more restrictive kernel derivative in Figure \ref{fig:IE_kernelequiv_obs} by construction cannot replicate the sharp horizontal red-blue switches in Figure \ref{fig:IEobs}.
It turns out that the strongest of these red-blue switches, namely the one at $y\approx 0.43$ in Figure \ref{fig:IEobs}~(left) is approximated as best as is allowed by a map perturbation, see Figure \ref{fig:IE_kernelequiv_obs}~(left), while the other two (weaker) horizontal red/blue switches seen in \ref{fig:IEobs} are ignored.

We now turn to optimal map perturbations for the mixing rate.
The combined effect of the ``cutting and shuffling'' of interval exchanges with diffusion on mixing rates has been widely studied, e.g.\ \cite{ashwin02,sturman12,FGTW16,kreczak17,wang2018}, including investigations of the impact of changing the diffusion or the interval exchange on mixing.  
The very general type of formal map optimisation we consider here has not been attempted before, and we hope that our novel techniques will stimulate interesting new research questions and motivate more sophisticated experiments in the field of mixing optimisation. 
  
Under repeated iteration, the original interval exchange $T_0$ cuts and shuffles the unit interval into an increasing number of smaller pieces, assisting the small scale mixing of diffusion.
Our results in Figure \ref{fig:IE-map-mix} (left) show an oscillatory $\dot{T}$, with increasing oscillations as the noise amplitude decreases.
 \begin{figure}[h!]
    \begin{center}
            \includegraphics[width=0.45\textwidth]{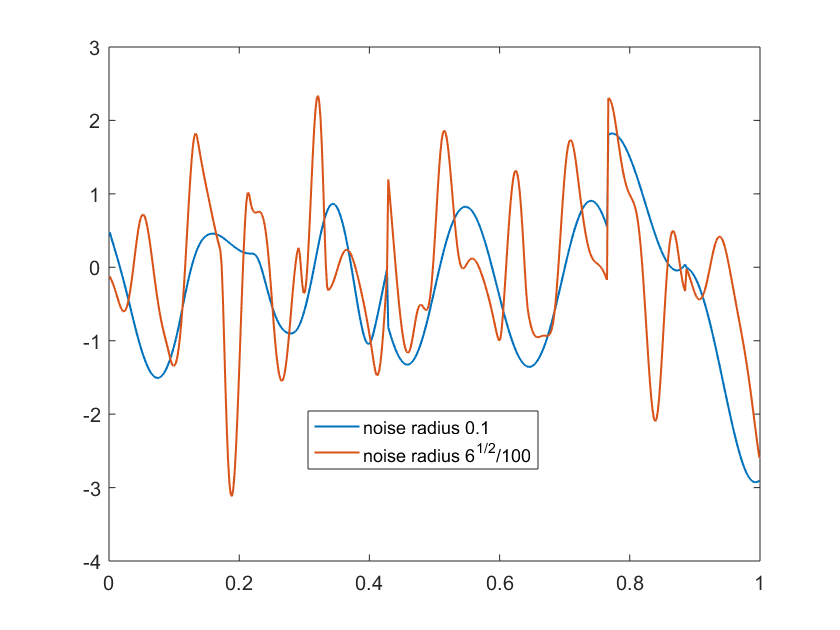}
        \includegraphics[width=0.45\textwidth]{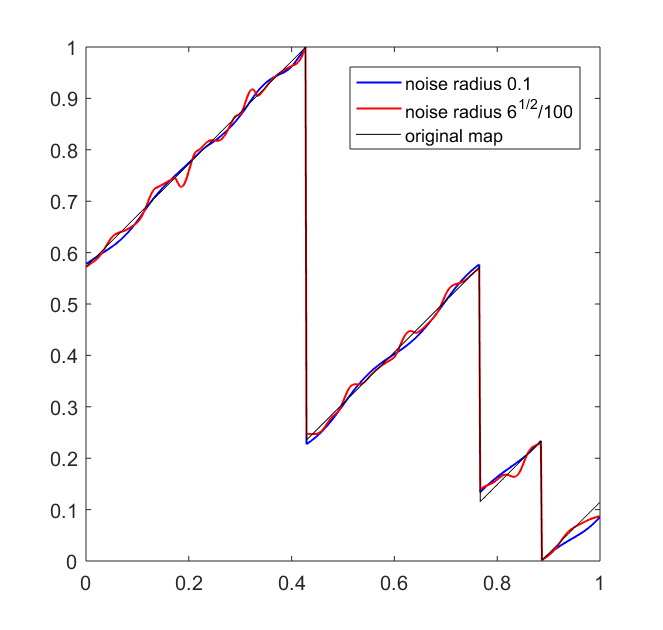}
            \caption{Left: Optimal map perturbation $\dot{T}$ for the interval exchange map to maximise the change in the mixing rate, computed using (\ref{eq:simp-opt-soln-mix-add-noise}) with $n=500$.  Right: Illustration of $T_0+\dot{T}/100$.}
        \label{fig:IE-map-mix}
        \end{center}
        \end{figure}
This increased oscillation effect is also seen when comparing the left and right panes of Figure \ref{fig:IE_kernelequiv_mix}.
Thus, the optimisation attempts to include some additional mixing by rapid local warping of the phase space.
It is plausible that this additional warping effect enhances mixing beyond the rigid shuffling of the interval exchange.
An illustration of $T_0+\dot{T}/100$ is given in Figure \ref{fig:IE-map-mix}. 
  \begin{figure}[h!]
    \begin{center}
        \includegraphics[width=0.45\textwidth]{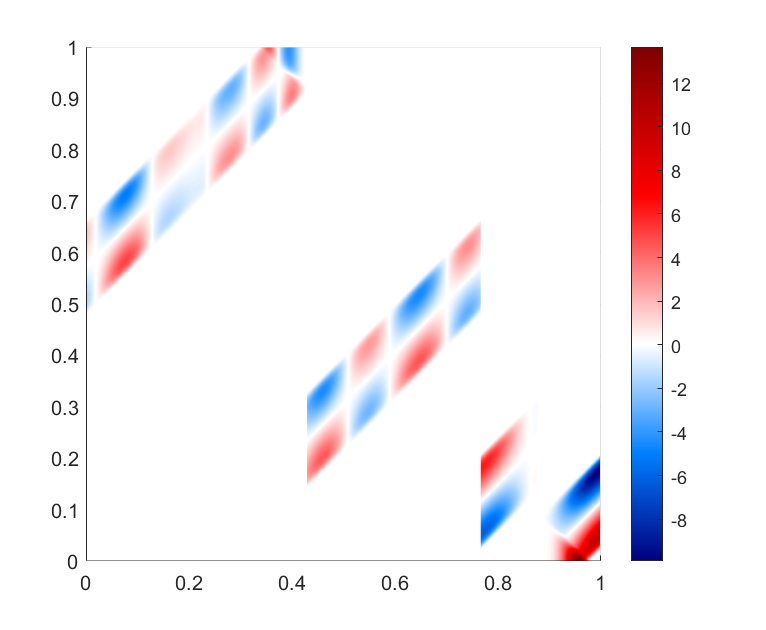}
        \includegraphics[width=0.45\textwidth]{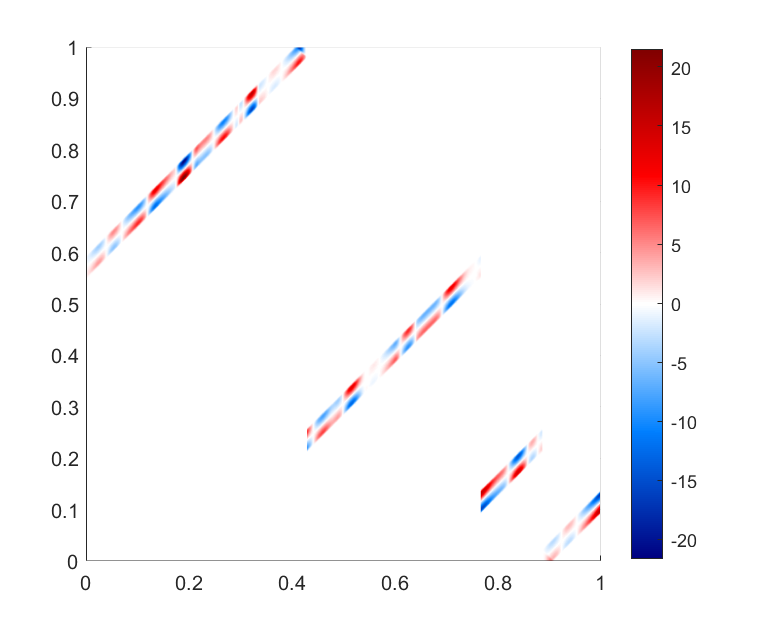}
        \caption{Kernel perturbations corresponding to the optimal map perturbations in Figure \ref{fig:IE-map-mix}. Left: $\epsilon=1/10$, Right: $\epsilon=\sqrt{6}/100$.}
        \label{fig:IE_kernelequiv_mix}
        \end{center}
        \end{figure}
We emphasise that the factor $1/100$ is only for visualisation purposes and for smaller factors, the perturbed map would remain a piecewise homeomorphism (modulo small overshoots at the boundaries, which are taken care of by the reflecting boundary conditions on the noise).
      



\section{Acknowledgments} FA is supported by a UNSW University Postgraduate Award. GF is partially supported by an ARC Discovery Project.  FA and GF thank the Department of Mathematics at the University of Pisa for generous support and hospitality. SG is partially supported by the research project \verb"PRIN 2017S35EHN_004" ``Regular and stochastic behaviour in dynamical systems'' of the Italian Ministry of Education and Research. 

\appendix

\section{Proof of Theorem \ref{thm:explicit-formula}}\label{apx:thm:explicit-formula}

First we need a technical lemma.
We note that the statement of the lemma is analogous to the continuity of $(\text{Id}-L_0)^{-1}$, which was treated in the proof of Theorem \ref{th:linearresponse copy(1)}.


\begin{lemma}
\label{sublem:adj-res-bnd} Consider the closed subspace span$\{f_0\}^\perp\subset L^2$ equipped with the $L^2$ norm. Then, the operator $(\text{Id}-L_0^*)^{-1}:$ span$%
\{f_0\}^\perp\rightarrow$ span$\{f_0\}^\perp$ is bounded.
\end{lemma}

\begin{proof}
We begin by finding the kernel and range of the operator $\text{Id}-L_0^*$.
Recall that $L_0(V)\subset V$ and that $L_0$ preserves a one-dimensional eigenspace span$\{f_0\}$,
with eigenvalue $1$. Thus, we have $\ker(\text{Id}-L_0)=$ span$\{f_0\}$ and
ran$(\text{Id}-L_0)\subset V$. Recalling that $L_0:V\rightarrow V$ is compact and $f_0\not\in V$, we have
by the Fredholm alternative (see \cite{DS}, VII.11) that
for any $g\in V$, there exists a unique $h\in V$ such that $g=(\text{Id}-L_0)h$. Hence, 
ran($\text{Id}-L_0)=V$. Since $V$ is closed, the range of $\text{Id}-L_0$
is closed and so, by the Closed Range Theorem (Theorem 5.13, IV-\S 5.2,\cite{K}), 
we have $\text{ran}((\text{Id}-L_0)^*) = \ker(\text{Id}-L_0)^\perp=$ 
span$\{f_0\}^\perp$, which is a
co-dimension $1$ space, and $\ker((\text{Id}-L_0)^*)=$ ran($\text{Id}-L_0)^\perp = V^\perp=$ span$\{\mathbf{1}\}^{\perp\perp}=$ span$\{\mathbf{1}\}$, where the last equality follows from Corollary 1.41 in 
III-\S 1.8 \cite{K}
and the fact that span$\{\mathbf{1}\}$ is a finite-dimensional closed subspace of $L^2$.

To prove that $(\text{Id}-L_0^*)^{-1}:$ span$%
\{f_0\}^\perp\rightarrow$ span$\{f_0\}^\perp$ is bounded, we will use the 
Inverse Mapping Theorem (Theorem III.11, \cite{RS}). 
Since the integral operator $L_0^*$ has an $L^2$ kernel, by \eqref{KF} and the triangle inequality
it follows that $\text{Id}-L_0^*$ is bounded. Also, from the Fredholm alternative argument above, 
$\text{Id}-L_0^*:\spn\{f_0\}^\perp\to\spn\{f_0\}^\perp$ is surjective. Thus, to apply the Inverse Mapping Theorem, we just need 
to show that $\text{Id}-L_0^*$ is injective on span$\{f_0\}^\perp$. Let $f_1,f_2\in$ span$%
\{f_0\}^\perp$ be such that $(\text{Id}-L_0^*)f_1=(\text{Id}-L_0^*)f_2$. Thus, 
$f_1-f_2\in\ker(\text{Id}-L_0^*) =$ span$\{\mathbf{1}\}$ and so $f_1-f_2 = \gamma
\mathbf{1}$ for some $\gamma\in{\mathbb{R}}$. Since $f_1-f_2\in$ span$\{f_0\}^\perp$, we
have that $0 = \int (f_1(x)-f_2(x))f_0(x)dx=\gamma\int f_0(x)dx$ and so $%
\gamma=0$ (since $\int f_0(x)dx =1$), i.e.\ $f_1=f_2$; thus, $(\text{Id}-L_0^*)$ is injective and the
result follows.
\end{proof}

\begin{proof}[Proof of Theorem \protect\ref{thm:explicit-formula}]
We will use the method of Lagrange multipliers to derive the expression %
\eqref{opt-soln} from the first-order necessary conditions for optimality and then show
that such a $\dot{k}$ satisfies the second-order sufficient conditions. To
this end, we consider the following Lagrangian function
\begin{equation*}
\mathcal{L}(\dot{k},\mu ):=f(\dot{k})+\mu g(\dot{k}),
\end{equation*}%
where $f(\dot{k}):=-\big\langle c,R(\dot{k})\big\rangle_{L^{2}([0,1],\bR)},$ $g(%
\dot{k}):=\Vert \dot{k}\Vert _{L^{2}([0,1]^{2})}^{2}-1$ and $\dot{k}\in
V_{\ker }\cap S_{k_0,l}$.

\textit{Necessary conditions:} We verify the conditions in Theorem 2, \S 7.7, \cite{DL}. We want to find $\dot{k}$ and $\mu$ that
satisfy the first-order necessary conditions:

\begin{align*}
& g(\dot{k})=0 \\
D_{\dot{k}}\mathcal{L}(\dot{k},\mu )& \tilde{k}=0\text{ for all }\tilde{k%
}\in V_{\ker }\cap S_{k_0,l},
\end{align*}%
where $D_{\dot{k}}\mathcal{L}(\dot{k},\mu )\in \mathcal{B}%
(L^{2}([0,1]^{2}),{\mathbb{R}})$ is the Frechet derivative with respect to
the variable $\dot{k}$. 
Since $f$ is
linear, we have $(D_{\dot{k}}f)\tilde{k}=f(\tilde{k})$. Also, $%
(D_{\dot{k}}g)\tilde{k}=2\langle \dot{k},\tilde{k}\rangle
_{L^{2}([0,1]^{2})} $ since

\begin{equation*}
\begin{aligned} \frac{|g(\k+\tilde{k})-g(\k)-2\langle
\k,\tilde{k}\rangle_{L^2([0,1]^2)} |}{\|\tilde{k}\|_{L^2([0,1]^2)}} &=
\frac{|\|\k+\tilde{k}\|_{L^2([0,1]^2)}^2-\|\k\|_{L^2([0,1]^2)}^2-2\langle
\k,\tilde{k}\rangle_{L^2([0,1]^2)}|}{\|\tilde{k}\|_{L^2([0,1]^2)}}\\ &=
\frac{|\langle\k+\tilde{k},\k+\tilde{k}\rangle_{L^2([0,1]^2)}-\langle
\k,\k\rangle_{L^2([0,1]^2)}-2\langle
\k,\tilde{k}\rangle_{L^2([0,1]^2)}|}{\|\tilde{k}\|_{L^2([0,1]^2)}}\\ &=
\frac{|\langle
\tilde{k},\tilde{k}\rangle_{L^2([0,1]^2)}|}{\|\tilde{k}\|_{L^2([0,1]^2)}}
=\|\tilde{k}\|_{L^2([0,1]^2)}. \end{aligned}
\end{equation*}%
Thus, for the necessary conditions of the Lagrange multiplier method to be
satisfied, we need that
\begin{equation}
D_{\dot{k}}\mathcal{L}(\dot{k},\mu )\tilde{k}=(D_{\dot{k}}f)\tilde{k}%
+\mu (D_{\dot{k}}g)\tilde{k}=f(\tilde{k})+2\mu \langle \dot{k},%
\tilde{k}\rangle _{L^{2}([0,1]^{2})}=0  \label{eq:Lag-nece-cond}
\end{equation}%
for all $\tilde{k}\in V_{\ker }\cap S_{k_0,l}$ and
\begin{equation}
g(\dot{k})=0.  \label{eq:Lag-nece-cond-2}
\end{equation}

Noting Lemma \ref{sublem:adj-res-bnd} and the fact that $c\in$ span$\{f_0\}^\perp$, 
we have
\begin{equation}
\begin{aligned} f(\tilde{k})&+2\mu\langle
\dot{k},\tilde{k}\rangle_{L^2([0,1]^2)}\\
& = -\langle c,
R(\tilde{k})\rangle_{L^2([0,1],\bR)} +
2\mu\langle\dot{k},\tilde{k}\rangle_{L^2([0,1]^2)}\\ 
&= -\bigg\langle c,
(\text{Id}-L_0)^{-1}\int \tilde{k}(x,y)f_0(y)dy\bigg\rangle_{L^2([0,1],\bR)} +
2\mu\langle \dot{k},\tilde{k}\rangle_{L^2([0,1]^2)}\\ 
& = \int \int
-((\text{Id}-L_0^*)^{-1}c)(x)\tilde{k}(x,y)f_0(y)dydx+\int \int 2\mu
\dot{k}(x,y)\tilde{k}(x,y)dydx\\ 
&= \int\int
\left[-((\text{Id}-L_0^*)^{-1}c)(x) f_0(y) + 2\mu\dot{k}(x,y)
\right]\tilde{k}(x,y) dydx. \end{aligned}
\label{eq:reform-to-2-arg-inner-prod}
\end{equation}%

We claim that
\begin{equation*}
\dot{k}(x,y)=\frac{1}{2\mu}\mathbf{1}_{F_l}(x,y)f_0(y)\left(((%
\text{Id}-L_{0}^{\ast })^{-1}c)(x)-\frac{1}{m(F_l^y)}\int_{F_l^y}((\text{Id}%
-L_{0}^{\ast })^{-1}c)(z)dz \right)
\end{equation*}
satisfies the necessary condition \eqref{eq:Lag-nece-cond} and lies in $V_{\ker}\cap S_{k_0,l}$. Before we verify this, we show that 
\begin{equation*}
M(x,y):= \mathbf{1}_{F_l}(x,y)f_0(y)\left(((\text{Id}-L_{0}^{\ast
})^{-1}c)(x)-\hat{g}(y) \right),
\end{equation*}
where $\hat{g}(y):=\frac{1}{m(F_l^y)}\int_{F_l^y}((\text{Id}-L_{0}^{\ast })^{-1}c)(z)dz$,
is in $L^2([0,1]^2)$. Since $f_0,(\text{Id}-L_{0}^{\ast })^{-1}c\in L^2$, we just need to show that $\mathbf{1}_{F_l}(x,y)f_0(y)\hat{g}(y)$ is in $L^2([0,1]^2)$. First, we note that 
\begin{align*}
    |\hat{g}(y)|&\le \frac{1}{m(F_l^y)}\int\big|\mathbf{1}_{F_l^y}(z)((\text{Id}-L_{0}^{\ast })^{-1}c)(z)\big|dz\\
    &\le \frac{1}{m(F_l^y)}\|(\text{Id}-L_{0}^{\ast })^{-1}c\|_2\|\mathbf{1}_{F_l^y}\|_2\\
    &= \frac{1}{m(F_l^y)}\|(\text{Id}-L_{0}^{\ast })^{-1}c\|_2 \sqrt{m(F_l^y)}\\
    &=\frac{\|(\text{Id}-L_{0}^{\ast })^{-1}c\|_2}{\sqrt{m(F_l^y)}}
\end{align*}
and therefore
\begin{align*}
    \hat{g}(y)^2\le \frac{\|(\text{Id}-L_{0}^{\ast })^{-1}c\|^2_2}{m(F_l^y)}.
\end{align*}
We then have
\begin{align*}
    \int\int\mathbf{1}_{F_l}(x,y)\hat{g}(y)^2f_0(y)^2dxdy &= \int_{\Xi(F_l)}\int_{F_l^y}\hat{g}(y)^2f_0(y)^2dxdy\\
    &=\int_{\Xi(F_l)}m(F_l^y)\hat{g}(y)^2f_0(y)^2dy\\
    &\le \int_{\Xi(F_l)}m(F_l^y) \frac{\|(\text{Id}-L_{0}^{\ast })^{-1}c\|^2_2}{m(F_l^y)}f_0(y)^2dy\\
    &\le \|(\text{Id}-L_{0}^{\ast })^{-1}c\|^2_2\|f_0\|_2^2.
\end{align*}
Thus, $\mathbf{1}_{F_l}(x,y)f_0(y)\hat{g}(y)$ is in $L^2([0,1]^2)$ and therefore $M\in L^2([0,1]^2)$. 

Now, to verify $\k$ satisfies \eqref{eq:Lag-nece-cond}, we compute, for $\tilde{k}\in V_{\ker}\cap S_{k_0,l}$,
\begin{align*}
f&(\tilde{k})+2\mu\langle\dot{k},\tilde{k}\rangle_{L^2([0,1]^2)}\\
&= \int_{F_l}\left[-((\text{Id}-L_0^*)^{-1}c)(x) f_0(y) + 2\mu\dot{k}(x,y)\right]\tilde{k}(x,y) dxdy\\ 
&=  \int_{F_l} \bigg[-((\text{Id}-L_0^*)^{-1}c)(x) f_0(y) \\
&\qquad\quad + f_0(y)\left(((%
\text{Id}-L_{0}^{\ast })^{-1}c)(x)-\frac{1}{m(F_l^y)}\int_{F_l^y}((\text{Id}%
-L_{0}^{\ast })^{-1}c)(z)dz \right)\bigg]\tilde{k}(x,y) dxdy\\
&= -\int_{\Xi(F_l)}\left[\int_{F^y_l}\left(
f_0(y)\frac{1}{m(F_l^y)}\int_{F_l^y}((\text{Id}-L_{0}^{\ast
})^{-1}c)(z)dz\right)\tilde{k}(x,y)dx\right]dy\\ 
&= - \int_{\Xi(F_l)}\left(
f_0(y)\frac{1}{m(F_l^y)}\int_{F_l^y}((\text{Id}-L_{0}^{\ast
})^{-1}c)(z)dz\right)\left[\int_{F^y_l}\tilde{k}(x,y)dx\right]dy\\
&=0,    
\end{align*}
where 
the last equality follows from $\tilde{k}\in V_{\ker}\cap S_{k_0,l}$. 
To conclude checking that $\k$ satisfies the necessary condition
\eqref{eq:Lag-nece-cond}, we need to check that
$\mu\neq0$. Since $M\in L^2([0,1]^2)$, note that the necessary condition \eqref{eq:Lag-nece-cond-2} yields $\mu =\pm
\frac{1}{2}\Vert M\Vert_{L^2([0,1]^2)}$; thus, to finish the proof that
$\k$ satisfies both necessary conditions \eqref{eq:Lag-nece-cond}-\eqref{eq:Lag-nece-cond-2}, we will show that
$\Vert M\Vert_{L^2([0,1]^2)}\neq 0$. 
From the hypotheses on $f_0$ and $(\text{Id}-L_{0}^{\ast })^{-1}c$
we conclude that 
$$\Vert M\Vert_{L^2([0,1]^2)}^2 = \int_{F_l}f_0(y)^2\left(((\text{Id}-L_{0}^{\ast
})^{-1}c)(x)-\frac{1}{m(F_l^y)}\int_{F_l^y}((\text{Id}-L_{0}^{\ast
})^{-1}c)(z)dz\right)^2dxdy\neq 0.$$
Hence, $\mu=\pm
\frac{1}{2}\Vert M\Vert_{L^2([0,1]^2)}\neq 0$. The sign of $\mu $ is determined by
checking the sufficient conditions. 

We can now verify that $\dot{k}\in V_{\ker}\cap S_{k_0,l}$. We note from $M\in L^2([0,1]^2)$ and $\mu\neq0$ that $\k\in L^2([0,1]^2)$. By construction
supp$(\dot{k})\subseteq F_l$. Finally, 
we have
\begin{equation*}
\begin{aligned}
\int \dot{k}(x,y) dx& =\frac{1}{2\mu}f_0(y)\left( \int_{F_l^y}((\text{Id}%
-L_{0}^{\ast })^{-1}c)(x)dx - \int_{F_l^y}((\text{Id}%
-L_{0}^{\ast })^{-1}c)(z)dz\frac{\int_{F_l^y}\mathbf{1}_{F_l}(x,y)
dx}{m(F_l^y)} \right)\\
&=0.    
\end{aligned}
\end{equation*}

\textit{Sufficient conditions}: We want to show that $\dot{k}$ in \eqref{opt-soln} is a 
solution 
to the optimization problem \eqref{obj_lin_fun-min}-%
\eqref{constraint-bound} by checking that it satisfies the second-order
sufficient conditions. We first demonstrate the set of Lagrange
multipliers $\Lambda (\dot{k})$ (in Definition 3.8, \S 3.1 \cite{BS}) is not
empty in our setting; this will enable us to use the second-order sufficient
conditions of Lemma 3.65 \cite{BS}. Note that in terms of the notation used
in \cite{BS} versus our notation, $Q=X=V_{\ker }\cap S_{k_0,l}$, $x_{0}=\dot{k}$, $%
Y^{\ast }={\mathbb{R}}$, $G(x_{0})=g(\dot{k})$, $K=\{0\}$, $N_{K}(G(x_{0}))={%
\mathbb{R}} $, $T_{K}(G(x_{0}))=\{0\}$ and $N_{Q}(x_{0})=\{0\}$ (since $Q=X$,
see discussion in \S 3.1 following Definition 3.8). Thus, to show that $%
\Lambda (\dot{k})$ is not empty, we need to show that $\dot{k}$ and $\mu
$ satisfy
\begin{equation}
D_{\dot{k}}\mathcal{L}(\dot{k},\mu )\dot{k}=0,\;g(\dot{k})=0,\;\mu
\in \{0\}^{-},\;\mu g(\dot{k})=0,  \label{cond-exist-lag-mult}
\end{equation}%
where $\{0\}^{-}:=\{a \in \mathbb{R}:a x\leq 0\ \forall x\in \{0\}\}=\mathbb{%
R}$ (this simplification of conditions (3.16) in \cite{BS} follows from the
discussion following Definition 3.8 in \S 3.1 and the fact that $\{0\}$ is a
convex cone). Since the second condition in \eqref{cond-exist-lag-mult} implies
the fourth, and since $\mu \in {\mathbb{R}}$, we only need to check the
first two equalities in \eqref{cond-exist-lag-mult}. However, these two
conditions are implied from the first-order necessary conditions. Hence, $%
\Lambda (\dot{k})$ is not empty and thus, to show that $\dot{k}$ is a
solution to \eqref{obj_lin_fun-min}-\eqref{constraint-bound}, we need to
show that it satisfies the following second-order conditions (see Lemma 3.65): there exists
constants $\nu >0$, $\eta >0$ and $\beta >0$ such that
\begin{equation}
\sup_{|\mu |\leq \nu ,\ \mu \in \Lambda (\dot{k})}D_{\dot{k}\dot{k}%
}^{2}\mathcal{L}(\dot{k},\mu )(\tilde{k},\tilde{k})\geq \beta \Vert
\tilde{k}\Vert _{L^{2}([0,1]^{2})}^{2},\ \forall \tilde{k}\in C_{\eta }(\dot{%
k}),  \label{eq:sec-ord-suff}
\end{equation}%
where $C_{\eta }(\dot{k}):=\big\{v\in V_{\ker }\cap S_{k_0,l}:|2\langle \dot{%
k},v\rangle _{V_{\ker\cap S_{k_0,l}}}|\leq \eta \Vert v\Vert _{V_{\ker }\cap
S_{k_0,l}}\text{ and }f(v)\leq \eta \Vert v\Vert _{V_{\ker }\cap S_{k_0,l}}%
\big\}$ is the approximate critical cone (see equation (3.131) in \S 3.3
\cite{BS}). Since $D_{\dot{k}}\mathcal{L}(\dot{k},\mu )\tilde{k}=f(\tilde{k}%
)+2\mu \langle \dot{k},\tilde{k}\rangle _{L^{2}([0,1]^{2})}$ and $%
\langle \dot{k},\tilde{k}\rangle _{L^{2}([0,1]^{2})}$ is linear in $\dot{k}$%
, we have that $D_{\dot{k}\dot{k}}^{2}\mathcal{L}(\dot{k},\mu )(\tilde{k}%
,\tilde{k})=2\mu \langle \tilde{k},\tilde{k}\rangle _{L^{2}([0,1]^{2})}$%
. Thus, we conclude that the second-order condition \eqref{eq:sec-ord-suff}
holds with $\mu >0$, $\nu =|\mu |=\frac{1}{2}\Vert M\Vert _{V_{\ker
}\cap S_{k_0,l}}$, $\beta =2\mu $ and $\eta =\max \big\{2\Vert \dot{k}%
\Vert _{V_{\ker }\cap S_{k_0,l}},\|c\|_2\|f_0\|_2\|(\text{Id}-L_0)^{-1}\|_{V\rightarrow V}%
\big\}$. Since $\dot{k}$ satisfies the necessary conditions %
\eqref{eq:Lag-nece-cond} and \eqref{eq:Lag-nece-cond-2} with $\mu >0$, we conclude that
$\dot{k}$ is a solution to the optimization problem \eqref{obj_lin_fun-min}-%
\eqref{constraint-bound}. 

\emph{Uniqueness of the solution:} 
The set $P_l=V_{\ker}\cap S_{k_0,l}\cap B_1$ is a closed (Lemma \ref{Sk0lem}), bounded, strictly convex set, containing $\dot{k}=0$. The objective $\mathcal{J}(\dot{k})=\langle c,R(\dot{k})\rangle$ is continuous (since $\mathcal{J}$ is linear and $R$ is continuous  (see comment following \eqref{R}))
and not uniformly vanishing (Lemma \ref{lem:nonzeroobj}). Therefore by Propositions \ref{prop:exist} and \ref{prop:uniqe}, $\dot{k}$ is the unique optimum.



\textit{$L^\infty$ boundedness of the solution}: Suppose that $c\in W$ and $k_0\in L^\infty([0,1]^2)$. From $L_0f_0=f_0$ and $k_0\in
L^\infty([0,1]^2)$, we have by \eqref{KF2} that $f_0\in L^\infty$.
Let $V_1 := \{f\in L^1:\int f\ dm = 0\}$. We would like to show
that $(\text{Id}-L_0)^{-1}:V_1\rightarrow V_1$ is bounded. To obtain this, we first need the
exponential contraction of $L_0$ on $V_1$.
Since $L_0$ is integral preserving and compact on $L^1$, 
from the argument in the proof of Theorem \ref{th:linearresponse copy(1)}
we only need to verify the $L^1$ version of assumption $(A1)$ on $V_1$. To verify this, we note that
for $h\in V_1$, we have $\|L_0h\|_2\le \|L_0 h\|_\infty\le \|k_0\|_{L^\infty([0,1]^2)}\|h\|_1$
and therefore, $L_0h\in V$ since $L_0$ preserves the integral. Thus, for any $h\in V_1$, $\lim_{n\rightarrow\infty}\|L_0^nh\|_1\le \lim_{n\rightarrow\infty}\|L_0^{n-1}(L_0h)\|_2=0$ since
$L_0$ satisfies $(A1)$ on $V$. Hence, the $L^1$ version of $(A1)$ holds and
$L_0$ has exponential contraction on $V_1$.
We then have
\begin{equation}
\begin{aligned}
\|(\text{Id}-L_0)^{-1}\|_{V_1\rightarrow V_1} &\le \|\text{Id}\|_{V_1\rightarrow V_1} + \bigg\|\sum_{n=1}^\infty L_0^n\bigg\|_{V_1\rightarrow V_1}\\
&= 1 + \sup_{\substack{f\in V_1\\ \|f\|_1=1}} \bigg\|\sum_{n=1}^\infty L_0^n f\bigg\|_1\\
&\le 1 + \sup_{\substack{f\in V_1\\ \|f\|_1=1}} \sum_{n=1}^\infty Ce^{\lambda n}\|f\|_1\\
&=1+ \sum_{n=1}^\infty Ce^{\lambda n}<\infty,
\end{aligned}
\end{equation}
where the last inequality follows from $\lambda<0$; thus, ($\text{Id}-L_0)^{-1}:V_1\rightarrow V_1$ is bounded.

Next we would like to find the subspace where the operator ($\text{Id}-L_0^*)^{-1}$ is bounded. We will 
replicate the result of Lemma \ref{sublem:adj-res-bnd}, however,
($\text{Id}-L_0)^{-1}$ is now acting on $L^1$, so we note that
for a subspace $\mathcal{S}$ of $L^1$, we have that
\begin{equation}  \label{eq:perp-L1}
\mathcal{S}^{\perp} := \bigg\{h\in L^\infty: \int h(x)w(x)dx = 0\ \forall w\in \mathcal{S}\bigg\},
\end{equation}
where we are using the fact that $(L^1)^*=L^\infty$. Also, 
$\mathcal{S}^{\perp}$ is a closed subspace of $L^\infty$ (see III-\S 1.4, \cite{K}).


Now, as in the proof of Lemma \ref{sublem:adj-res-bnd}, we have $\ker(\text{Id}-L_0)=$
span$\{f_0\}$ and ran($\text{Id}-L_0)=V_1$. We also have $\text{%
ran((Id}-L_0)^*)=$ span$\{f_0\}^{\perp} =\{h\in L^\infty:\int
h(x)f_0(x)dx = 0\}=:W$ and $\ker((\text{Id}%
-L_0)^*)=V_1^{\perp}=\{h\in L^\infty:\int h(x)w(x)dx = 0\ \forall\
w\in V_1\}$. Next, for $h\in W$, we have
\begin{equation*}
\int (L_0^*h)(x)f_0(x)dx = \int h(x)(L_0f_0)(x)dx = \int h(x)f_0(x)dx = 0;
\end{equation*}
thus, ($\text{Id}-L_0^*)(W)\subset W$. 
We again, as in Lemma \ref{sublem:adj-res-bnd}, apply the Inverse Mapping Theorem
to prove that ($\text{Id}-L_0^*)^{-1}: W\rightarrow W$ is 
bounded. From \eqref{KF2}, and the triangle inequality, the operator $\text{Id}-L_0^*: W\rightarrow W$
is bounded. Noting that $V_1$ is a closed co-dimension 1
subspace of $L^1$, we have codim$(V_1)=$ dim$(V_1^\perp)$ (see Lemma 1.40 III-\S 1.8 \cite{K}); hence, dim$(\ker(\text{Id}-L_0^*))= $ dim$(V_1^\perp)=$ 
codim$(V_1)= 1$ 
and therefore, $1$ is a geometrically simple eigenvalue of $L_0^*$. 
Thus, $\ker(\text{Id}-L_0^*)=$ span$\{\mathbf{1}\}$ because $L_0^*\mathbf{1}=\mathbf{1}$. 
Since $\int f_0\ dm =1$, $\mathbf{1}\not\in$ span$\{f_0\}^\perp$ and so,
by the Fredholm alternative, $\text{Id}-L_0^*$ is a bijection on $W$.
Hence, by the Inverse Mapping Theorem, ($\text{Id}-L_0^*)^{-1}$ is bounded on $W$.  
Since $c\in W$, we have $\|(\text{Id}-L_0^*)^{-1}c\|_\infty<\infty $.

To conclude the proof, we now show that $ \hat{g}(y) 
:=\frac{1}{m(F_l^y)}\int_{F_l^y}((\text{Id}-L_{0}^{\ast })^{-1}c)(z)dz$
is in $L^\infty$. We compute
\begin{equation*}
|\hat g(y)| = \bigg| \frac{1}{m(F_l^y)}\int_{F_l^y}((\text{Id}%
-L_{0}^{\ast})^{-1}c)(z)dz\bigg|\le \|(\text{Id}-L_0^*)^{-1}c\|_{\infty}.
\end{equation*}
Since $(\text{Id}-L_0^*)^{-1}c\in L^{\infty}$, we conclude that $\hat g\in
L^\infty$; thus, $\dot{k}\in L^\infty([0,1]^2)$.
\end{proof}

\section{Proof of Theorem \ref{thm:explicit-formula-eval}}\label{apx:thm:explicit-formula-eval}

\begin{proof}
The optimization problem is very similar to that
considered in Theorem \ref{thm:explicit-formula}; thus, we will refer to the
proof of that theorem with the following modifications.

Consider the Lagrangian function
\begin{equation*}
\mathcal{L}(\dot{k},\mu ):=f(\dot{k})+\mu g(\dot{k}),
\end{equation*}%
where, in this setting, we have $f(\dot{k})=\langle \dot{k},E\rangle
_{L^{2}([0,1]^{2},\bR)}$ and $g(\dot{k})=\Vert \dot{k}\Vert
_{L^{2}([0,1]^{2},\bR)}^{2}-1$. Thus, for the necessary conditions of the
Lagrange multiplier method to be satisfied, we need that
\begin{equation}
f(\tilde{k})+2\mu \langle \dot{k},\tilde{k}\rangle
_{L^{2}([0,1]^{2})}=\langle \tilde{k},E\rangle _{L^{2}([0,1]^{2},\bR)}+2\mu
\langle \dot{k},\tilde{k}\rangle _{L^{2}([0,1]^{2},\bR)}=0
\label{eq:necc-cond-lag-2nd-eval-1}
\end{equation}%
for all $\tilde{k}\in V_{\ker }\cap S_{k_0,l}$ and
\begin{equation}
g(\dot{k})=0.  \label{eq:necc-cond-lag-2nd-eval-2}
\end{equation}%
We claim that
\begin{equation}
\dot{k}(x,y)=-\mathbf{1}_{F_l}(x,y)\frac{1}{2\mu }\left( E(x,y)-%
\frac{1}{m(F_l^y)}\int_{F_l^y}E(x,y)dx\right)  \label{eq:non-scaled-sol-2nd-eval}
\end{equation}%
satisfies the necessary condition \eqref{eq:necc-cond-lag-2nd-eval-1}, and lies in $V_{\ker}\cap S_{k_0,l}$. Before we verify this, we will show that
\begin{align*}
    M(x,y):=\mathbf{1}_{F_l}(x,y)(E(x,y)-h(y)),
\end{align*}
where $h(y):=\frac{1}{m(F^y_l)}\int_{F^y_l}E(x,y)dx$, is in $L^2([0,1]^2)$. Since $E\in L^2([0,1]^2)$, we just need to show that $\mathbf{1}_{F_l}(x,y)h(y)$ is in $L^2([0,1]^2)$. We have 
\begin{align*}
    \int\int\mathbf{1}_{F_l}(x,y)h(y)^2dxdy = \int_{\Xi(F_l)}\int_{F_l^y}h(y)^2dxdy=\int_{\Xi(F_l)}m(F_l^y)h(y)^2dy.
\end{align*}
Substituting \eqref{eq:almost-soln-2nd-evalue} into $h$, the terms in $h(y)^2$ are a linear combination of functions of the form $\tilde{g}_{i_1}(y)\tilde{g}_{i_2}(y)f_{i_3}(y)f_{i_4}(y)$, $i_1,\ldots,i_4\in\{1,\ldots,4\}$ where 
$f_j = \Re(\hat{e}),\Re(e),\Im(\hat{e})$ or $\Im(e)$, $j=1,\ldots,4$, respectively, and $\tilde{g}_j(y) = \frac{1}{m(F_l^y)}\int_{F_l^y}f_j(x)dx$, $j=1,\ldots,4$. 
Thus, to show $\mathbf{1}_{F_l}(x,y)h(y)$ is in $L^2([0,1]^2)$ (and therefore $M\in L^2([0,1]^2)$), we need to bound
\begin{align*}
    I:= \int_{\Xi(F_l)}m(F_l^y)|\tilde{g}_{i_1}(y)||\tilde{g}_{i_2}(y)|f_{i_3}(y)||f_{i_4}(y)|dy.
\end{align*}
We note that
\begin{align*}
    |\tilde{g}_j(y)|\le \frac{1}{m(F_l^y)}\int|\mathbf{1}_{F_l^y}(x)f_j(x)|dx\le \frac{1}{m(F_l^y)}\|\mathbf{1}_{F_l^y}\|_2\|f_j\|_2 = \frac{\|f_j\|_2}{\sqrt{m(F_l^y)}}.
\end{align*}
Thus, we have
\begin{align*}
    I&\le \int_{\Xi(F_l)}m(F_l^y)\frac{\|f_{i_1}\|_2}{\sqrt{m(F_l^y)}}\frac{\|f_{i_2}\|_2}{\sqrt{m(F_l^y)}}|f_{i_3}(y)||f_{i_4}(y)|dy\\
    &= \|f_{i_1}\|_2\|f_{i_2}\|_2\int_{\Xi(F_l)}|f_{i_3}(y)f_{i_4}(y)|dy\le \|f_{i_1}\|_2\|f_{i_2}\|_2\|f_{i_3}\|_2\|f_{i_4}\|_2.
\end{align*}
Since $f_j\in L^2$, $j=1,\ldots,4$, we conclude that $M\in L^2([0,1]^2)$.

Now, to verify $\k$ satisfies the first necessary condition, we compute, for $\tilde{k}\in V_{\ker }\cap S_{k_0,l}$, the central term in \eqref{eq:necc-cond-lag-2nd-eval-1}
\begin{equation*}
\begin{aligned} 
\langle\tilde{k}, E+2\mu \dot{k} \rangle_{L^2([0,1]^2,\bR)}
&= \int_{F_l}\tilde{k}(x,y)\left(E(x,y)+2\mu\dot{k}(x,y)\right)dxdy\\
&= \int_{\Xi(F_l)}\int_{F_l^y}\tilde{k}(x,y)\left(E(x,y)-E(x,y)+h(y)\right)dxdy\\
&= \int_{\Xi(F_l)}\left[\int_{F_l^y}\tilde{k}(x,y)dx\right]h(y)dy = 0,
\end{aligned}
\end{equation*}%
where the 
last equality is from $\tilde{k}\in V_{\ker }\cap S_{k_0,l}$. 
To conclude the check that $\k$ satisfies the necessary condition \eqref{eq:necc-cond-lag-2nd-eval-1}, we need to check that $\mu\neq  0$. Since $M\in L^2([0,1]^2)$, note that the necessary condition
\eqref{eq:necc-cond-lag-2nd-eval-2} yields $\mu =\pm\frac{1}{2}\Vert M\Vert _{L^{2}([0,1]^{2},\bR)}$; thus, to finish the proof that $\k$ satisfies both
necessary conditions \eqref{eq:necc-cond-lag-2nd-eval-1}-\eqref{eq:necc-cond-lag-2nd-eval-2}, we will
show that $\Vert M\Vert _{L^{2}([0,1]^{2},\bR)}\neq 0$.
From the hypotheses on $E$ we conclude 
$$\Vert M\Vert _{L^{2}([0,1]^{2},\bR)}^2 = \int_{F_l}\left(E(x,y) - \frac{1}{m(F_l^y)}\int_{F_l^y}E(z,y)dz\right)^2dxdy\neq 0.$$
Hence $\mu =\pm\frac{1}{2}\Vert M\Vert _{L^{2}([0,1]^{2},\bR)}\neq 0$.
The sign of $\mu$ is determined by checking the sufficient conditions.  

We can now verify that $\k\in V_{\ker }\cap S_{k_0,l}$. We note 
from $M\in L^2([0,1]^2)$ and $\mu\neq 0$, $\k\in L^2([0,1]^2)$. By construction, supp$(\k)\subseteq F_l$. Finally, we have 
\begin{equation*}
\begin{aligned} \int\dot{k}(x,y)dx &=
-\frac{1}{2\mu}\left(\int_{F_l^y}E(x,y)dx -
\frac{1}{m(F_l^y)}\int_{F_l^y}E(z,y)dz\int\mathbf{1}_{F_l}(x,y)dx\right)\\ &= -\frac{1}{2\mu}\left(\int_{F_l^y}E(x,y)dx -
\frac{1}{m(F_l^y)}\int_{F_l^y}E(z,y)dz\ m(F_l^y)\right)\\ &= 0.
\end{aligned}
\end{equation*}


For the sufficient conditions, we note that in this setting $D_{\dot{k}\dot{k%
}}^2\mathcal{L}(\dot{k},\lambda)(\tilde{k},\tilde{k})$ is the same as in the
proof of Theorem \ref{thm:explicit-formula} (since the objectives considered
in both this and the other optimization problem are linear). Hence, the
second order sufficient conditions are satisfied with $\mu>0$. Thus,
with $2\mu = \|M\|_{L^2([0,1]^2,\bR)}$, %
\eqref{eq:soln-2nd-eval} satisfies the necessary and sufficient conditions.
Next, we note that the set $P_l=V_{\ker}\cap S_{k_0,l}\cap B_1$ is a closed (Lemma \ref{Sk0lem}), bounded, strictly convex set, containing $\dot{k}=0$. The objective $\mathcal{J}(\dot{k})=\langle \dot{k},E\rangle_{L^{2}([0,1]^{2},\bR)}$ is continuous and not uniformly vanishing (Lemma \ref{lem:nonzeroobj}). 
Therefore by Propositions \ref{prop:exist} and \ref{prop:uniqe}, \eqref{eq:soln-2nd-eval} is the unique solution to the optimization
problem \eqref{obj_lin_fun-min-2nd-eval}-\eqref{constraint-bound-2nd-eval}.

We finally show that $E\in L^\infty([0,1]^2,\bR)$ by supposing $k_0\in L^2([0,1]^\infty,\bR)$. Recall that
$$E(x,y) = \big(\Re(\hat{e})(x)\Re(e)(y)+ \Im(\hat{e})%
(x)\Im(e)(y)\big)\Re(\lambda_{0})+ \big(\Im(\hat{e})(x)\Re(e)(y) -\Re(\hat{e})%
(x)\Im(e)(y)\big)\Im(\lambda_{0}).$$ Since $L_0e = \lambda_{0}e$ and $L_0^*\hat{e} = \lambda_{0}\hat{e}$, we have from inequality \eqref{KF2} that $e,\hat{e}\in L^\infty([0,1],\bC)$
since $k_0\in L^\infty([0,1]^2,\bR)$. Hence, we have that $\Re(e),\Re(\hat{e}),\Im(e),\Im(\hat{e})\in L^\infty([0,1],\bR)$ and thus $E\in L^\infty([0,1]^2,\bR)$.
\end{proof}

\section{Upper bound for the norm of the reflection operator}
\begin{lemma}\label{lem:upbndP}
Let $P_\pi$ be as in \eqref{eq:PF-pi} and assume that the support of $f\in L^2(\bR)$ is contained in $N$ intervals of lengths $a_j, j=1,\ldots,N$.
Then, $\|P_\pi f\|_{L^2([0,1])}\le \left(\sum_{j=1}^N\lceil a_j+1\rceil\right)\|f\|_{L^2(\bR)}$, where $\lceil x\rceil$ denotes the smallest integer greater than or equal to $x$. 
\end{lemma}
\begin{proof}
Using translation invariance of Lebesgue measure, and the fact that for each fixed $x$ there are at most $\sum_{j=1}^N\lceil a_j+1\rceil$ nonzero evaluations of $f$ in the infinite sum  below, 
\begin{eqnarray*}
\int_0^1 (P_\pi f)(x)^2\ dx &=&
     \int_0^1\left(\sum_{i\in 2\mathbb{Z}}( f(i+x)+f(i-x))\right)^2\ dx\\
     &\le&\int_{-\infty}^\infty \left(\sum_{j=1}^N\lceil a_j+1\rceil\right)^2 f(x)^2\ dx\\
&=&\left(\sum_{j=1}^N\lceil a_j+1\rceil\right)^2\|f\|_{L^2}^2.
\end{eqnarray*}
\end{proof}

\section{Proof of Theorem \ref{thm:explicit-formula2}}\label{apx:thm:explicit-formula2}
\begin{proof}
The proof will follow the structure of the proof of Theorem \ref{thm:explicit-formula}%
. To this end, we consider the following Lagrangian function
\begin{equation*}
\mathcal{L}(\dot{T},\mu):= f(\dot{T}) + \mu g(\dot{T}),
\end{equation*}
where $f(\dot{T}) := -\big\langle c, \widehat{R}(\dot{T})\big\rangle_{L^2([0,1],\bR)}, $ $g(\dot{T}) := \|%
\dot{T}\|^2_2-1$ and $\dot{T}\in S_{T_0,\ell}$.

\textit{Necessary conditions}: We want to find $\dot{T}$ and $\mu$ that
satisfy the first-order necessary conditions:

\begin{align*}
&g(\dot{T}) = 0 \\
D_{\dot{T}} \mathcal{L}(\dot{T},\mu)&\tilde{T}= 0\text{ for all }\tilde{T%
}\in S_{T_0,\ell},
\end{align*}
where $D_{\dot{T}} \mathcal{L}(\dot{T},\mu)\in \mathcal{B}(L^2,{\mathbb{R%
}})$ is the Frechet derivative with respect to the variable $\dot{T} $. Since $%
f $ is linear, we have $(D_{\dot{T}}f) \tilde{T} = f(\tilde{T})$. Also, we
have that $(D_{\dot{T}}g) \tilde{T} = 2\langle \dot{T},\tilde{T}%
\rangle_{L^2([0,1],\bR)} $ (following the computation in the proof of Theorem \ref%
{thm:explicit-formula}). Thus, for the necessary conditions of the Lagrange
multiplier method to be satisfied, we need that
\begin{equation}  \label{eq:Lag-nece-cond2}
D_{\dot{T}}\mathcal{L}(\dot{T},\mu)\tilde{T} = (D_{\dot{T}}f)\tilde{T} +
\mu (D_{\dot{T}}g)\tilde{T} = f(\tilde{T})+2\mu\langle \dot{T},%
\tilde{T}\rangle_{L^2([0,1],\bR)} = 0
\end{equation}
for all $\tilde{T}\in S_{T_0,\ell}$ and
\begin{equation}  \label{eq:Lag-nece-cond-2-2}
g(\dot{T}) = 0.
\end{equation}

Following the proof of Theorem \ref{thm:explicit-formula}, we will solve for $%
\dot{T}$ by rewriting $f(\tilde{T})+2\mu\langle \dot{T},\tilde{T}%
\rangle_{L^2([0,1],\bR)}$ as an inner product on $L^2$. To this end, we have
that
\begin{equation}  \label{eq:reform-to-2-arg-inner-prod2}
\begin{aligned} f(\tilde{T}&)+2\mu\langle \dot{T},\tilde{T}\rangle_{L^2([0,1],\bR)}\\
&= \bigg\langle c, (\text{Id}-L_0)^{-1}\int_0^1 \left(P_\pi\left(\tau_{-T_0(y)}\frac{d\rho}{dx}\right) \right)(x)\tilde{T}(y)f_0(y)dy\bigg\rangle_{L^2([0,1],\bR)} +
2\mu\langle \dot{T},\tilde{T}\rangle_{L^2([0,1],\bR)}\\ &= \bigg\langle
(\text{Id}-L_0^*)^{-1}c, \int_0^1 \left(P_\pi\left(\tau_{-T_0(y)}\frac{d\rho}{dx}\right) \right)(x)\tilde{T}(y)f_0(y)dy\bigg\rangle_{L^2([0,1],\bR)} + \langle
2\mu \dot{T},\tilde{T}\rangle_{L^2([0,1],\bR)}\\ & = \int_0^1 \int_0^1
((\text{Id}-L_0^*)^{-1}c)(x)\left(P_\pi\left(\tau_{-T_0(y)}\frac{d\rho}{dx}\right) \right)(x)\tilde{T}(y)f_0(y)dydx+\langle
2\mu \dot{T},\tilde{T}\rangle_{L^2([0,1],\bR)}\\ &= \int_0^1
\left[\int_0^1((\text{Id}-L_0^*)^{-1}c)(x)\left(P_\pi\left(\tau_{-T_0(y)}\frac{d\rho}{dx}\right) \right)(x)dxf_0(y) + 2\mu\dot{T}(y) \right]\tilde{T}(y) dy\\
&=\int_0^1
\left[f_0(y)\mathcal{G}((\text{Id}-L_0^*)^{-1}c)(y)+ 2\mu\dot{T}(y) \right]\tilde{T}(y) dy.
\end{aligned}
\end{equation}
We note that since $c\in$ span$\{f_0\}^\perp$, we have from Lemma \ref{sublem:adj-res-bnd} that
$(\text{Id}-L_0^*)^{-1}c\in L^2$ and the above expression is well defined.  
Now, from
\eqref{eq:reform-to-2-arg-inner-prod2}, we have that $f(\tilde{T})+2\mu\langle \dot{T},%
\tilde{T}\rangle_{L^2([0,1],\bR)} = \langle f_0\ \mathcal{G}((\text{Id}-L_0^*)^{-1}c)+2\mu \dot{T}, \tilde{T}\rangle_{L^2([0,1],\bR)}$.
From this we can conclude that finding $\dot{T}$ and $\mu$ that
satisfy \eqref{eq:Lag-nece-cond2} and \eqref{eq:Lag-nece-cond-2-2} reduces
to finding $\dot{T}\in S_{T_0,\ell}$ and $\mu\in\bR$ that satisfy
$\langle f_0\ \mathcal{G}((\text{Id}-L_0^*)^{-1}c)+2\mu \dot{T},
\tilde{T}\rangle_{L^2([0,1],\bR)}=0$ for all $\tilde{T}\in S_{T_0,\ell}$ and %
\eqref{eq:Lag-nece-cond-2-2}. Using the non-degeneracy of the inner product,
we find that
\begin{equation*}
\dot{T}= -\frac{M}{2\mu},
\end{equation*}
where 
\begin{equation}
    \label{MM}
    M = \mathbf{1}_{\widetilde{F}_\ell}f_0\ \mathcal{G}((\text{Id}-L_0^*)^{-1}c).
\end{equation} To conclude that the above $\dot{T}$ 
satisfies the necessary condition \eqref{eq:Lag-nece-cond2}, we need to check 
that $\mu\neq 0$. 
Since $M\in L^\infty$ (see the \emph{Boundedness of the solution} paragraph below), the necessary condition \eqref{eq:Lag-nece-cond-2-2} yields $\mu = \pm
\frac{1}{2}\|M\|_2$; thus, to finish the proof that $\dot{T}$ satisfies both
necessary conditions \eqref{eq:Lag-nece-cond2}-\eqref{eq:Lag-nece-cond-2-2},
we will show that $\|M\|_2\neq 0$. 
From the hypotheses on $f_0$ and $(\text{Id}-L_0^*)^{-1}c$, and recalling that
$\mathcal{P}(x,y)=P_\pi \left(\tau_{-T_0(y)}\frac{d\rho}{dx}\right)(x)$, we conclude that
$$\|M\|_2^2 = \int_{\widetilde{F}_\ell}f_0(y)^2\left(\int \mathcal{P}(x,y)((\text{Id}-L_0^*)^{-1}c)(x)dx\right)^2dy\neq 0.$$
Hence $\mu = \pm\frac{1}{2}\|M\|_2\neq 0$; the sign of $\mu$ is determined by checking the sufficient conditions.
We thus have verified that $\dot{T}\in S_{T_0,\ell}$ because $\dot{T}\in L^2$ and the term $\mathbf{1}_{\widetilde{F}_\ell}$ in (\ref{MM}) guarantees supp$(\dot{T})\subseteq \widetilde{F}_\ell$. 

\textit{Sufficient conditions}: As in the proof of Theorem \ref{thm:explicit-formula}, we
will show that $\dot{T}$ in \eqref{opt-soln2} is the solution to the optimization problem %
\eqref{obj_lin_fun-min2}-\eqref{constraint-bound2} by checking that it
satisfies the second-order sufficient conditions. We first note that in this
setting we have $Q=X=S_{T_0,\ell}$, $x_0 = \dot{T}$, $Y^* = {\mathbb{R}}$, $G(x_0)=g(%
\dot{T})$, $K = \{0\} $, $N_K(G(x_0)) = {\mathbb{R}}$, $T_{K}(G(x_0))=\{0\}$
and $N_Q(x_0) = \{0\}$.
Thus, to show that $\Lambda(\dot{T})$ is not empty, we need to show that $%
\dot{T}$ and $\mu$ satisfy
\begin{equation}  \label{cond-exist-lag-mult2}
D_{\dot{T}}\mathcal{L}(\dot{T},\mu)\dot{T} = 0,\; g(\dot{T})=0,\;
\mu\in \{0\}^-,\; \mu g(\dot{T}) =0,
\end{equation}
where $\{0\}^-:=\{\alpha\in \mathbb{R}: \alpha x\le 0\ \forall x\in \{0\}\}=%
\mathbb{R}$. Following the argument in the proof of Theorem \ref{thm:explicit-formula}, it is easily verifiable 
that $\Lambda(\dot{T})$ is not empty. Thus, to show that $\dot{T}$ is a
solution to \eqref{obj_lin_fun-min2}-\eqref{constraint-bound2}, we need to
show that it satisfies the following second-order conditions: there exists
constants $\nu>0$, $\eta>0$ and $\beta>0$ such that
\begin{equation}  \label{eq:sec-ord-suff2}
\sup_{|\mu|\le \nu,\ \mu\in\Lambda(\dot{T})} D^2_{\dot{T}\dot{T}}%
\mathcal{L}(\dot{T},\mu)(\tilde{T},\tilde{T})\ge \beta \|\tilde{T}%
\|^2_2,\ \forall\ \tilde{T}\in C_\eta(\dot{T}),
\end{equation}
where $C_\eta(\dot{T}):= \big\{v\in S_{T_0,\ell}: |2\langle\dot{T},v\rangle_{S_{T_0,\ell}}|\le
\eta\|v\|_{S_{T_0,\ell}}\text{ and } f(v)\le \eta\|v\|_{S_{T_0,\ell}} \big\}$ is the
approximate critical cone. Since $D_{\dot{T}}\mathcal{L}(\dot{T},\mu)\tilde{%
T} = f(\tilde{T})+2\mu\langle \dot{T},\tilde{T}\rangle_{L^2([0,1],\bR)} $ and $%
\langle \dot{T},\tilde{T}\rangle_{L^2([0,1],\bR)}$ is linear in $\dot{T}$, we have that
$D^2_{\dot{T}\dot{T}}\mathcal{L}(\dot{T},\mu)(\tilde{T},\tilde{T}) =
2\mu \langle \tilde{T},\tilde{T}\rangle_{L^2([0,1],\bR)}$. Thus, we conclude that
the second-order condition \eqref{eq:sec-ord-suff2} holds with $\mu>0$, $%
\nu=|\mu| = \frac{1}{2}\|M\|_{S_{T_0,\ell}} $, $\beta = 2\mu$ and $\eta = \max%
\big\{2\|\dot{T}\|_{S_{T_0,\ell}},\|M\|_{S_{T_0,\ell}}\big\}$. Since $\dot{T}$ satisfies
the necessary conditions \eqref{eq:Lag-nece-cond2} and %
\eqref{eq:Lag-nece-cond-2-2}, with $\mu>0$, $\dot{T}$ is a solution to
the optimization problem \eqref{obj_lin_fun-min2}-\eqref{constraint-bound2}.
Using Lemma \ref{lem:nonzeroobjmap} and Proposition \ref{solutionT}, we conclude that this solution is unique.

\textit{Boundedness of the solution:} 
We have that $\left(P_\pi\left(\tau_{-T_0(y)}\frac{d\rho}{dx}\right) \right)(x)\in L^\infty([0,1]^2)$
(see proof of Lemma \ref{lem:cont-resp-ref-bnd}). From inequality \eqref{KF2}, with the kernel
$\left(P_\pi\left(\tau_{-T_0(y)}\frac{d\rho}{dx}\right) \right)(x)$,
we have that $\mathcal{G}h\in L^\infty$ for any $h\in L^2$.
Since $f_0\in L^\infty$,
we have that $f_0\ \mathcal{G}((\text{Id}-L_0^*)^{-1}c)\in L^\infty$. 
Thus, $M=\mathbf{1}_{\widetilde{F}_\ell} f_0\ \mathcal{G}((\text{Id}-L_0^*)^{-1}c)\in L^\infty$ and therefore $\dot{T}\in L^\infty$. 
\end{proof}



\section{Proof of Theorem \ref{extremiser2}}\label{apx:extremiser2}
\begin{proof}
We use arguments similar to those in the proofs of Theorems \ref{thm:explicit-formula2} and \ref{thm:explicit-formula-eval}.
Let
$\widehat{E}$ be as in \eqref{eq:pre-opt-sol-eval}. For
the necessary conditions, we will need that
\begin{equation}  \label{eq:Lag-nec-2-eval-add-1}
\langle\tilde{T},\widehat{E} + 2\mu \dot{T}\rangle_{L^2([0,1],\bR)} = 0
\end{equation}
for all $\tilde{T}\in S_{T_0,\ell}$ and
\begin{equation}  \label{eq:Lag-nec-2-eval-add-2}
\|\dot{T}\|_2^2 = 1.
\end{equation}
Thus, from \eqref{eq:Lag-nec-2-eval-add-1} and the nondegeneracy of the inner product
we have that
$\dot{T} = -\mathbf{1}_{\widetilde{F}_\ell}\frac{\widehat{E}}{2\mu}$.
To conclude that $\dot{T}$ 
satisfies the necessary condition \eqref{eq:Lag-nec-2-eval-add-1}, we need to check 
that $\mu\neq 0$. 
Since $\widehat{E}\in L^2$ (as it is essentially bounded, see Proposition \ref{prop:add-noise-eval-obj-mix}),
 the necessary condition \eqref{eq:Lag-nec-2-eval-add-2} yields $\mu= \pm
\frac{1}{2}\|\widehat{E}\|_2$.
Thus, to finish the proof that $\dot{T}$ satisfies both
necessary conditions \eqref{eq:Lag-nec-2-eval-add-1}-\eqref{eq:Lag-nec-2-eval-add-2}, we will show that $\|\widehat{E}\|_2\neq 0$. 
From the hypotheses on $E$, and recalling that
$\mathcal{P}(x,y)=P_\pi \left(\tau_{-T_0(y)}\frac{d\rho}{dx}\right)(x)$, we conclude that
$$\|\widehat{E}\|_2^2=\int_{\widetilde{F}_\ell}\left(\int_0^1\mathcal{P}(x,y)E(x,y)dx\right)^2dy\neq 0.$$
Hence $\mu= \pm\frac{1}{2}\|\widehat{E}\|_2\neq 0$ and 
$\dot{T} = \mp\mathbf{1}_{\widetilde{F}_\ell}\frac{\widehat{E}}{\big\|\widehat{E}\big\|_2}$; the sign of $\mu$ is determined by checking the sufficient conditions. 
Clearly $\dot{T}\in L^2$ and has support contained in $\tilde{F}_l$, thus $\dot{T}\in S_{T_0,l}$.
For the sufficient conditions, as in the proof of Theorem \ref{thm:explicit-formula2},
since the objective is linear, we require that  $\mu>0$. Using Lemma \ref{lem:nonzeroobjmap} and Proposition \ref{solutionT2} we 
conclude that \eqref{eq:soln-2nd-eval-add} is the unique solution. The essential boundedness of
$\dot{T}$ follows from the essential boundedness of $\widehat{E}$ (see Proposition \ref{prop:add-noise-eval-obj-mix}).
\end{proof}

\bibliographystyle{abbrv}

\bibliography{opt_lin_resp}

\end{document}